\documentclass{article}
\usepackage[utf8]{inputenc}
\usepackage{amsmath,amsthm, amssymb,latexsym,enumitem,hyperref}
\usepackage{bm}
\usepackage[title,titletoc,toc]{appendix}
\usepackage{mathtools}
\usepackage{graphics}
\usepackage[english]{babel}
\usepackage{xcolor}
\usepackage{amsthm}
\usepackage{multicol}
\usepackage{times}
\usepackage{hyperref}
\usepackage[T1]{fontenc}
\usepackage{textpos}
\usepackage{subfigure}
\usepackage{framed} 
\usepackage{color}
\usepackage{tikz}
\usepackage{tikz-cd}
\usepackage{mathtools}
\usepackage{commath}
\usepackage{nicematrix}
\usepackage{enumitem}
\usepackage{geometry}
\usepackage{float}

\usepackage[doi=false,isbn=false,url=false,eprint=false, style=alphabetic, backend=bibtex,maxbibnames=99]{biblatex}
\DeclareMathAlphabet{\mathbbold}{U}{bbold}{m}{n}

\theoremstyle{definition}

\newcommand{\ITM}{\mathrm{ITM}}
\newcommand{\IET}{\mathrm{IET}}

\newcommand{\C}{\mathbb{C}}

\newcommand{\N}{\mathbb{N}}
\newcommand{\R}{\mathbb{R}}

\newcommand{\betas}{\beta_*}
\newcommand{\betass}{\beta_{**}}

\newtheorem{theorem}{Theorem}[section]
\newtheorem{question}[theorem]{Question}
\newtheorem{definition}[theorem]{Definition}

\newtheorem{example}[theorem]{Example}
 
\newtheorem{proposition}[theorem]{Proposition}
\newtheorem{conjecture}[theorem]{Conjecture}
\newtheorem*{conjecture*}{Conjecture}
\newtheorem*{mtheorem}{Topological Prevalence of Finite Type Maps}
\newtheorem*{mconjecture}{Boshernitzan--Kornfeld Conjecture}
\newtheorem*{mconjecture2}{Irrational Rotations Conjecture}
\newtheorem*{mtheoremI}{Main Theorem I}
\newtheorem*{mtheoremII}{Main Theorem II}

\newtheorem*{theoremA}{Theorem A} 
\newtheorem*{theoremB}{Theorem B} 
\newtheorem*{theoremC}{Theorem C}
\newtheorem*{theorem*}{Theorem}
\newtheorem{lemma}[theorem]{Lemma}
\newtheorem{corollary}[theorem]{Corollary}

\usepackage{xpatch}
\makeatletter
\AtBeginDocument{\xpatchcmd{\@thm}{\thm@headpunct{.}}{\thm@headpunct{}}{}{}}
\makeatother

\newcounter{lstv}
\newenvironment{lstv}{%
\refstepcounter{lstv}%
\begin{center}
\begin{minipage}{.9\textwidth}}{%
\end{minipage}%
\makebox[.1\textwidth][r]{(*)}%
\end{center}}

\newcounter{lsta}
\newenvironment{lsta}{%
\refstepcounter{lsta}%
\begin{center}
\begin{minipage}{.9\textwidth}}{%
\end{minipage}%
\makebox[.1\textwidth][r]{(**)}%
\end{center}}

\author{
Kostiantyn Drach\footnote{Universitat de Barcelona (Gran Via de les Corts Catalanes, 585, 08007 Barcelona, Spain) and Centre de Recerca Matem\`atica (Edifici C, Carrer de l'Albareda, 08193 Bellaterra, Spain), email: kostiantyn.drach@ub.edu}, 
Leon Staresinic\footnote{Imperial College London (180 Queen's Gate, South Kensington, London SW7 2AZ, UK), email: l.staresinic21@imperial.ac.uk}, 
and Sebastian van Strien\footnote{Imperial College London (180 Queen's Gate, South Kensington, London SW7 2AZ, UK), email: s.van-strien@imperial.ac.uk}}
\title{Density of Stable Interval Translation Maps} 
\date{\today}

\begin{document}

\maketitle

\begin{abstract} An \emph{Interval Translation Mapping} (ITM) is a piece-wise translation $T \colon I \to I$ defined on a finite partition $I_1, \ldots, I_r$ of an interval $I$ into $r \ge 2$ sub-intervals. We do not assume that the images of these intervals are disjoint. These maps naturally generalize classical Interval Exchange Transformations (IETs) by removing the bijectivity assumption. Let $\ITM(r)$ be the space of all interval translation mappings, where we fix $r$ but not the intervals $I_1,\ldots,I_r$, nor the translations on each of them. The set $X(T) :=\bigcap_{n\ge 0} T^n(I)$ is either a finite union of intervals (and $T$ behaves like an IET on those intervals), in which case the map is called \emph{of finite type}, or is a disjoint union of finitely many intervals and a Cantor set, in which case the map is called \emph{of infinite type}. In this paper, for an arbitrary $r \ge 2$, we show that the set of finite type maps contains an open and dense set of $\ITM(r)$. This resolves in positive a topological version of a long-standing conjecture due to Boshernitzan and Kornfeld. More precisely, we show that there exists an open and dense subset $\mathcal{S}(r)$ of $\ITM(r)$ consisting of \emph{stable maps} such that each $T\in \mathcal{S}(r)$
\begin{itemize}[leftmargin=4mm,labelsep=2mm,itemsep=1mm]
\item[-] is of finite type; 
\item[-] the first return map to any component of $X(T)$  corresponds to a  circle rotation;
\item[-] $\mathcal{S}(r) \ni T \mapsto X(T)$ is continuous in the Hausdorff topology. 
\end{itemize}
 
\end{abstract}

\setcounter{tocdepth}{1} 
\tableofcontents  


\section{Introduction}

\textit{Interval Translation Mappings} ($\ITM$s) are orientation-preserving piecewise isometries of an interval. As such, they are the generalization of the well-known Interval Exchange Transformations ($\IET$s) obtained by dropping the bijectivity assumption on the mapping. They were introduced by Boshernitzan and Kornfeld in \cite{MR1356616}. The graphs of an $\ITM$ (left) and $\IET$ (right) on $4$ intervals are displayed in the figures below.

\begin{figure}[H]
    \centering
    \begin{minipage}{0.45\textwidth}
        \centering
        \vspace{-1mm}
        \includegraphics[width=1.1\textwidth, trim={18 15 15 10},clip]{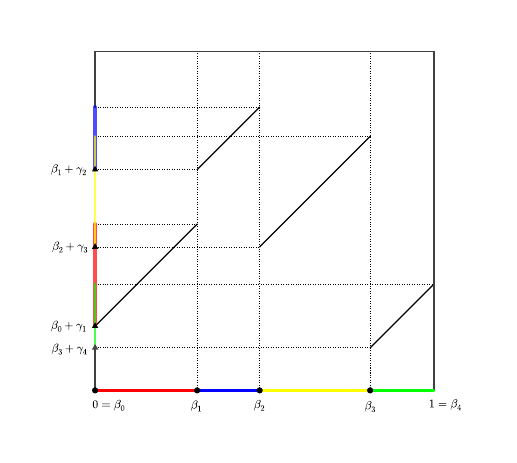}
        \vspace*{-6mm}
        \caption{$\ITM$ on $4$ intervals}\label{fig:ITM}
    \end{minipage}\hfill
    \begin{minipage}{0.45\textwidth}
        \centering
        \vspace{-2mm}
        \includegraphics[width=1.1\textwidth, trim={20 15 20 10},clip]{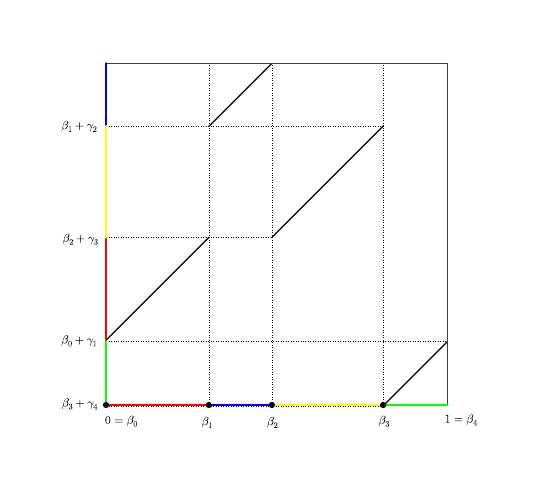}
        \vspace*{-6mm}
        \caption{$\IET$ on $4$ intervals}\label{fig:IET}
    \end{minipage}
\end{figure}

\noindent More precisely, an $\ITM$ on $r$ intervals is a map $T \colon [0,1) \to [0,1)$ defined by two sets of parameters:

\begin{itemize}
    \item Discontinuities: $0 < \beta_1 < \beta_2 < \dots < \beta_{r-1} < 1$;
    \item Translation factors: $\gamma_1, \gamma_2, \dots, \gamma_3$;
\end{itemize}

\noindent so that $T(x) = x + \gamma_i$ for $x \in [\beta_{i-1},\beta_{i})$, where $1 \le i \le r$. The set of all such parameters forms the parameter space $\ITM(r)$ of all Interval Translation Mappings on $r$ intervals. Naturally, $\ITM(r)$ is a closed convex polytope in $\mathbb{R}^{2r-1}$. 

The basic difference between $\ITM$s and $\IET$s is the fact that for a typical $\ITM$ $T$ we have that $T(I) \subsetneq I$, where $I = [0,1)$, which implies that we have the following chain of inclusions: $I \supsetneq T(I) \supseteq T^2(I) \supseteq T^3(I) \supseteq \dots$. This leads to a natural dichotomy for $\ITM$s: A map $T$ is said to be of \textit{finite type} if there exists an integer $n$ such that $T^n(I) = T^{n+1}(I)$, and is of \textit{infinite type} if $T^n(I) \supsetneq T^{n+1}(I)$ for all $n \ge 0$. The first example of an infinite type map was constructed in \cite{MR1356616} by Boshernitzan and Kornfeld. They conjectured that such behavior should be rare:

\begin{mconjecture}
\label{conj:inf-type-zero}
For all $r \ge 2$, the set of all infinite type $\ITM$s on $r$ intervals is a measure zero subset of $\ITM(r)$.    
\end{mconjecture}

This conjecture has been the focus of a large part of research in the field of $\ITM$s, but remains widely open, except in a few cases. Namely, it was first established for $r=3$ and a special $2$-parameter family of $\ITM$s in \cite{MR2013352}. It was shown that this family supports a renormalization scheme analogous to the Rauzy--Veech induction for $\IET$s (see also \cite{artigiani2024renormalization}). The dynamical properties of this renormalization imply that the measure of the set of all infinite type maps must be zero, analogous to how the analysis of the Rauzy--Veech induction shows that almost all minimal $\IET$s are uniquely ergodic (see \cite{MR644018}, \cite{MR644019}). A similar method of renormalization was used for $r \le 4$ and a more complicated family of \textit{double rotations} in \cite{MR2152403} and \cite{MR2966738} (see also \cite{MR4397159}). The conjecture was established in full for $r=3$ in \cite{MR3124735} by showing that almost every $\ITM$ on three intervals can be renormalized to a double rotation. Unfortunately, there has been little progress on developing a renormalization scheme for $\ITM$s on an arbitrary number of intervals (except in some special cases \cite{MR2308208}).

In this paper, instead of using renormalization, we develop a theory of stability for $\ITM$s (Section \ref{sec:theoremB}) and show a transversality result that allows for detailed control over global dynamics under perturbation (Section \ref{sec:lin-indep}, Theorem \ref{thm:lin-dep}). Using these results, we establish the topological version of Boshernitzan--Kornfeld Conjecture, the main result of this paper:

\begin{mtheorem}
\label{mtheorem}
For every $r \ge 2$, the set of all finite type $\ITM$s on $r$ intervals contains an open and dense subset of $\ITM(r)$.    
\end{mtheorem}

We say that a map $T$ corresponds to an irrational circle rotation if the return map to every component interval of its non-wandering set is equivalent to an irrational circle rotation. In particular, such maps are of finite type. The methods used in the proof of Topological Prevalence of Finite Type Maps also suggest a significant refinement of the Boshernitzan--Kornfeld Conjecture:

\begin{mconjecture2}
\label{conj:rotations-full-meas}
The set of all $\ITM$s on $r$ intervals that correspond to irrational circle rotations has full measure in $ITM(r)$.
\end{mconjecture2}

It is widely believed that the Boshernitzan--Kornfeld Conjecture is true. In Section \ref{sec:ae-fin-stable} we will show that if it is true, then the stability theory developed in this paper implies the Irrational Rotations Conjecture.

We give the formal introduction and statements of our results in the following subsections.

\subsection{Interval Translation Mappings}
\label{subsec:itm}

A map $T:[0,1)\to [0,1)$ is called an \textit{Interval Translation Mapping ($\ITM$)} if there exists a partition of $I = [0,1)$ into finitely many intervals $I_i=[\beta_{i-1},\beta_{i})$, where $0=\beta_0<\beta_1<\dots < \beta_{r-1}<\beta_r=1$ and $1 \le i \le r$, such that the restriction of $T$ to each of these intervals is a translation. In other words, we have that:
\[
T\vert_{I_i}(x) = x+\gamma_i,
\]
for $1 \le i \le r$. Note that we must have that $\gamma_i \in [-\beta_{i-1}, 1-\beta_i]$ for all $1 \le i \le r$, since the image of $T$ is contained in $[0,1)$ by definition. These conditions mean that the parameter space $\ITM(r)$ for $\ITM$s on $r$ intervals is a convex polytope in $\R^{2r-1}$ and we endow $\ITM(r)$ with the corresponding subspace topology. In the special case when $T$ is bijective, we get the well-known Interval Exchange Transformations ($\IET$s).

For $n \ge 1$, let  $X_n(T) := T^n(I)$ and note that $X_{n+1} \subset X_n$. Let 
\[
X:=\bigcap^{\infty}_{n=0} X_n.
\]
Note that each $X_n$ is a union of finitely many intervals and that $X$ is therefore non-empty. An interval translation mapping is said to be of \textit{finite type} if $X_{n+1}=X_n$ for some $n$ and of \textit{infinite type} when $X_{n+1} \subsetneq X_n$ for all $n$. 

It is easy to see that for $r=2$ all $\ITM$s are of finite type and $X$ is a single interval on which the induced map is a rotation. In general, it is known that $\overline{X}$ is equal to the union of two sets $A_1$ and $A_2$ (one of which can be empty) so that $A_1$ is a union of a finite number of intervals and $A_2$ is a topological Cantor set (see Lemma \ref{lem:x-structure}). It was shown in \cite{MR1796167} that $T$ is of finite type if and only if $X$ is a finite union of intervals, and that if $T$ is of infinite type and topologically transitive then $\overline{X}$ is a Cantor set. In Lemma \ref{lem:nonwandering} we will show that $\overline{X}$ is equal to the non-wandering set of $T$. 

The first explicit example of an infinite type $\ITM$ was given in \cite{MR1356616}. In \cite{MR2013352} for $r=3$ Bruin and Troubetzkoy introduced a special two-parameter family of maps with a renormalisation operator acting on it; we call this family the Bruin--Troubetzkoy family. It was shown that within this family, infinite type $\ITM$s are those that under successive iterations of this renormalisation operator never enter a particular open subset of the parameter space. Moreover, it was shown that the set of infinite type maps has Lebesgue measure zero in parameter space. Figures \ref{fig:bt-fin} and \ref{fig:bt-inf} show the parameter space of the special family from \cite{MR2013352}: the right figure (taken from \cite{MR2013352}) shows the set of infinite type maps, while each open triangle in the left figure\footnote{The picture by Bj\"orn Winckler.} corresponds to a region of equivalent maps, in the sense defined in Subsection \ref{subsec:bt-family}, that are all of finite type. These regions are dynamically related through renormalisation, see \cite{MR2013352} for details.

\begin{figure}[H]
    \centering
    \begin{minipage}{0.45\textwidth}
        \centering
        \includegraphics[width=0.9\textwidth]{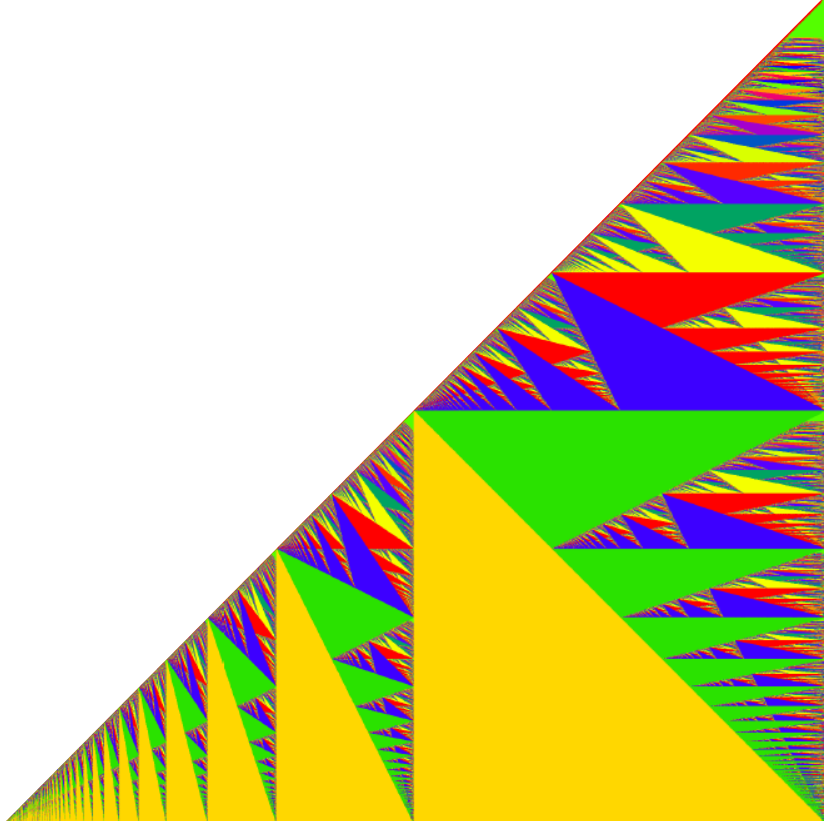}
        \caption{Finite type maps from the family in \cite{MR2013352}}\label{fig:bt-fin}
    \end{minipage}\hfill
    \begin{minipage}{0.45\textwidth}
        \centering
        \includegraphics[width=\textwidth]{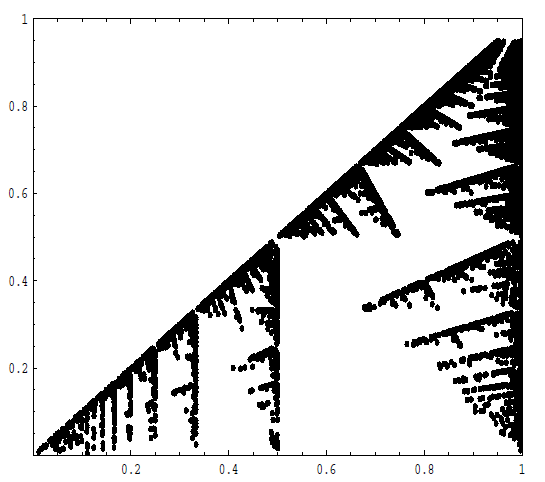}
        \caption{Infinite type maps from the family in \cite{MR2013352}}\label{fig:bt-inf}
    \end{minipage}
\end{figure}

Linear recurrence and weakly mixing properties for the maps the Bruin--Troubetzkoy family were recently studied in \cite{bruin2023interval}. Estimates on the Hausdorff dimension of the set of all infinite type maps in the family were obtained in \cite{artigiani2024renormalization}.

A similar analysis was carried out for the more complicated family of \textit{double rotations} in \cite{MR2152403}, \cite{MR2966738} and \cite{MR4397159}. In \cite{MR3124735}, it was shown that almost every $\ITM$ on three intervals is of finite type. In \cite{MR2308208}, the same result was shown for special families on an arbitrary number of intervals.

This paper initiates a systematic study of $\ITM$s on $r$ intervals without restrictions on $r$.

\subsection{Stability of $\ITM$s}

One of the main aims in the field of dynamical systems is to determine which dynamical systems are stable, and if possible, to prove that such systems are {\lq}typical{\rq}. It turns out that $\ITM$s are never structurally stable (see Corollary \ref{cor:itm-not-struct-stab}). For this reason, we define a map to be {\em stable} in this context if the next best thing holds: if -loosely speaking- its non-wandering set depends continuously on the system. This notion of stability is weaker than the usual ones in $1$-dimensional dynamics because it does not require that the system and its perturbation are conjugate on their non-wandering sets (see \cite{MR1239171}, \cite{MR1312365}). Note that for $\IET$'s, the non-wandering set is always the entire interval, so the main result of this paper becomes trivial.

That density of structural stability holds for smooth real 1D systems is proven in \cite{MR2342693}; that it holds for complex holomorphic maps in 1D  was proven by McMullen and Sullivan \cite{MR1620850}; it does not hold in higher dimension due to examples by Newhouse (and others). That it holds for $\ITM$'s is proved in this paper.

{\small \begin{center}
\begin{tabular}{l | l |c  }
stability notion & appropriate context &  density holds for \\ 
\hline 
structural stability & smooth systems & real 1D: $\checkmark$, $\C$1D: $\checkmark$,  real $>1$D: $\times$ \\
stability & piecewise isometry & ITM: $\checkmark$ 
\end{tabular}
\end{center}}
\noindent  

Whether or not a map in $\ITM$ is stable is equivalent to 
two conditions: Absence of Critical Connections (ACC) and the Matching condition (M) (see Section \ref{sec:theoremB} and Theorem \ref{thm:stability=accm} for more details). Here ACC means that no iterate of a discontinuity (critical point) is mapped to another discontinuity. In the $\IET$ literature this is often called the Keane condition. The Matching condition means that the return map to a component of the non-wandering set has at most one discontinuity. This condition has no clear analogue for invertible systems. It rules out that one connected component of the non-wandering set consists of two pieces with independent dynamics which join up for {\lq}coincidental{\rq} reasons. This characterization of stability implies in particular that stable maps are simple: their dynamics on their non-wandering set corresponds to (a union of) circle rotations.

{\small \begin{center}
\begin{tabular}{c |c |c |c}
stability notion & context &  condition & density  \\ 
\hline 
structural stability &  smooth systems >1D & hyperbolicity and transversality & $\times$ \\ 
structural stability & smooth 1D systems & hyperbolicity and ACC  & $\R$: $\checkmark$, $\C$: ? \\
stability & ITM & ACC and M & $\checkmark$
\end{tabular}
\end{center}}

So in this sense, the results for $\ITM$s of this paper are the analogue of the density of hyperbolicity in the real one-dimensional dynamics (proved in \cite{MR2342693}). The corresponding question for systems on the complex plane (such as quadratic maps) is the Fatou conjecture and is still wide open. Hence the ACC and M conditions can be considered as the analogues of hyperbolicity in our setting.

\begin{definition}[Stable maps]
We say that a map $T$ is \textit{stable} if there a neighbourhood $\mathcal{U}$ of $T$ in $\ITM(r)$ such that:

\begin{enumerate}
\item The mapping $\overline{X}$ assigning to each $\ITM$ in $\mathcal{U}$ its non-wandering set is continuous with respect to the Hausdorff topology on compact subsets of $I$. Moreover, for each $\tilde{T} \in \mathcal{U}$, $\overline{X}(\Tilde{T})$ is homeomorphic to $\overline{X}(T)$.
\item The number of discontinuities in $I \setminus \overline{X}(\Tilde{T})$ is constant in $\mathcal{U}$.
\end{enumerate}
\end{definition}  

It is easy to see that stable maps are of finite type (see Corollary \ref{cor:stable-fin-type}). Moreover, the number of components of $\overline{X}(T)$ is constant in $\mathcal{U}$. In general, the map $T\mapsto \overline{X}(T)$ is neither upper nor lower semi-continuous. For example, it is possible to have that the set $\overline{X}(T)$ can jump up or down as is shown in Example \ref{ex:ghost-preimage}.

Let $\mathcal{S}(r)$ be the set of all stable maps in $\ITM(r)$. For a stable map $T$, denote by $[T]$ the connected component containing $T$ in $\mathcal{S}(r)$. We call $[T]$ the \textit{stability component} of $T$. When we do not need to emphasize a specific map $T$, we call a connected component of $\mathcal{S}(r)$ a \textit{stable region}. Note that the minimal number $n$ so that $T^n(I)=X(T)$ does not need to be constant
within a stable region.  It is clear that $[T]$ is the maximal neighbourhood $\mathcal{U}$ for which the conditions of stability hold. It is easy to prove that for $r=2$ the entire parameter space $\ITM(2)$ is equal to the closure of a single stability component. A very interesting open problem is to describe the geometry of stable regions (see subsection \ref{subsec:param-space}).

In the space of $\IET$s, the definition of stability is vacuous, as the non-wandering set is always equal to $[0,1)$, so all of the conditions for stability are trivially satisfied. On the other hand, by the Characterization of Stability Theorem \ref{thm:stability=accm}, every $\IET$ is not stable when considered as a map in the space of $\ITM$s.

\subsection{Summary of results}
\label{subsec:summary-of-results}

The results and pictures discussed in Subsection \ref{subsec:itm} suggest that the finite type maps should be prevalent in the parameter space for an arbitrary number of intervals. As mentioned at the beginning, this was already conjectured in the measure-theoretic sense by Boshernitzan and Kornfeld in the first paper on $\ITM$s (\cite{MR1356616}).

Our main result shows that this conjecture is true in the topological sense, i.e. that the set of all finite type maps contains an open and dense subset of the full parameter space $\ITM(r)$. This is achieved by developing the theory of stability of $\ITM$s.

\begin{mtheoremI}[Density of Stable Interval Translation Maps]
\label{thm:density-of-stable}
The set $\mathcal{S}(r)$ of all stable maps interval translations maps on $r$ intervals is dense in the parameter space $\ITM(r)$.
\end{mtheoremI} 

It is easy to show that the stable maps are of finite type (see Corollary \ref{cor:stable-fin-type}), from which we obtain the main theorem of this paper:

\begin{mtheoremII}[Topological Prevalence of Finite Type Maps] 
Finite type maps contain an open and dense set of the parameter space $\ITM(r)$, while the set of infinite type maps has empty interior in $\ITM(r)$.    
\end{mtheoremII}

\begin{proof}
By Main Theorem I, the stable maps form a dense subset of $\ITM(r)$. Since the set $\mathcal{S}(r)$ is open by definition and stable maps are of finite type by Corollary \ref{cor:stable-fin-type}, the result follows.
\end{proof}

The proof of Main Theorem I is divided into the following three steps. 

\begin{theoremA}[Eventually periodic maps are dense]
\label{thmA}
The set $EP(r)$ of all eventually periodic maps is dense in $\ITM(r)$. 
\end{theoremA}

An eventually periodic map is the one for which every point in $I$ lands on a periodic point.

\begin{theoremB}[Stability $\iff$ ACC + Matching]
\label{thmB}
A finite type map is stable if and only if it satisfies the Absence of Critical Connections and Matching conditions.
\end{theoremB}

The ACC and Matching conditions will be defined in Section \ref{sec:theoremB}.

\begin{theoremC}[Stable maps approximate Eventually Periodic maps]
\label{thmC}
For any eventually periodic map $T$, there exists an arbitrarily small perturbation $\Tilde{T}$ of $T$ that is of finite type and satisfies the ACC and Matching conditions, and in particular, is stable by Theorem B.
\end{theoremC}

\begin{proof}[Proof of the Main Theorem]
By the theorems above we have that:
\begin{align*}
&\overline{\mathcal{S}(r)} \\
\overset{\text{Thm B}}{=}\, & \overline{\{\text{Maps satisfying ACC and Matching} \}} \\
\overset{\text{Thm C}}{\supseteq} \, & \overline{EP(r)} \\
\overset{\text{Thm A}}{=} \, & \ITM(r).
\end{align*}
\end{proof}

\begin{corollary}
\label{cor:itm-not-struct-stab}
No map in $\ITM(r)$ is structurally stable or $\Omega$-stable. 
\end{corollary} 
\begin{proof} Near each map there is an eventually periodic map, by Theorem A, and one with rationally independent parameters (which implies that there are no periodic points). 
\end{proof} 

In Section \ref{sec:theoremA}, we prove that the eventually periodic maps are of finite type and that the set of all eventually periodic maps contains a countable dense subset of $\ITM(r)$. By the Characterization of Stability Theorem \ref{thm:stability=accm}, we also know that the return maps to intervals of $X$ must correspond to circle rotations. The stable maps which are also eventually periodic therefore correspond to rational rotations, in the sense that the return map to any component interval of $X$ is a rational rotation. In Section \ref{sec:ae-fin-stable} we show that the stable maps form a full measure subset of the set of all finite type maps.

In Subsection \ref{subsec:bt-family}, we apply our results to the special family from \cite{MR2013352} to show that the connected components of stability are exactly the coloured triangle regions in Figure~\ref{fig:bt-fin}. Moreover, any point not contained in a closure of a triangle is of infinite type. In Subsection \ref{subsec:param-space}, we will discuss the full parameter space $\ITM(r)$ of $\ITM$s on $r$ intervals, with a focus on the geometry of stable regions. 

\subsection{Main idea of the proof}
\label{subsec:main-idea}
The basic problem we deal with everywhere in the paper is: how to perturb a map $T$ so that some of its dynamical properties change, while some others do not? In analogy to smooth one-dimensional dynamics (e.g. \cite{MR1239171}), the dynamical properties of a map depend on the behaviour of critical points. For us, the discontinuities of the map serve as the critical points. Thus the problem is about controlling critical itineraries: we want some of them to change, and some of them to stay the same up to certain return times. Because our maps are piecewise linear, this can be reformulated in terms of the linear independence of certain vectors defined by the itineraries of critical points (see Section \ref{sec:lin-indep} for details). The main difficulty is the following: how does the dynamics of a map $T$ prevent the linear dependence of a set of vectors? We deal with this in the proof of Theorem \ref{thm:lin-dep}. This is a transversality result that is crucial in the proofs of Theorem B and Theorem C. The proof of this linear independence relies on a rather detailed and subtle dynamical analysis.

\subsection{Overview of $\IET$s, $\ITM$s and piecewise isometries}
\label{subsec:history}

The subject of interval exchange transformations ($\IET$s) goes back to 
the 60's. Some the earliest publications were by Katok \cite{MR0331438}, Keane \cite{MR0357739}, \cite{MR0435353} and Veech \cite{MR0516048}. For a broad overview of $\IET$s, we recommend the survey papers \cite{MR2219821}, \cite{MR2648692}.  It is well-known that $\IET$s are connected to billiards, see for example \cite{MR0399423} and \cite{MR0644840}, \cite[Ch. 5]{MR0832433} and \cite[Ch.2, \S 4D]{MR0889254}, and to flows on flat surfaces, see \cite{MR2000471}, \cite{MR644019}, \cite{MR644018}, \cite{MR1393518}, \cite{MR1733872}, \cite{MR2261104}.

In Subsection \ref{subsec:itm} we covered the results about generic behavior and renormalization for $\ITM$s, so we mention the remaining topics here. The topological properties of general $\ITM$s were first studied by Schmeling and Troubetzkoy in \cite{MR1796167}. Buzzi proved in \cite{MR1855837} that all piecewise isometries on polytopes have zero topological entropy (for $\ITM$s this was remarked already in \cite{MR1356616}). Buzzi and Hubert proved in \cite{MR2054049} that all piecewise monotone maps without periodic points (with some necessary assumptions) are semi-conjugate to $\ITM$s (possibly with flips). They also showed that the upper bound for the number of ergodic measures of an $\ITM$ on $r$ intervals is $r-1$, in contrast to $\IET$s for which the bound is $\frac{r}{2}$ (for the precise result see \cite{MR2648692}). This bound was realized by the family considered by Bruin in \cite{MR2308208}. $\ITM$s were connected to billiards with {\lq}one-sided mirrors{\rq}, and analysed in \cite{MR3449199} and \cite{MR3403406} (see Appendix~\ref{appendix:billiards} for the explanation of this connection). Connecting $\ITM$s with flows on surfaces is a very interesting open problem.

$\ITM$s also appeared in the work of Levitt in 1993 (see \cite{MR1231840}) in the context of group action on $\mathbb{R}$-trees. There he considered a two-parameter family of $\ITM$s related to the holonomy map of certain 1-dimensional foliations.

The notion of Matching also appears in previous works, see for example \cite{MR3893724}, \cite{MR3597033}, \cite{MR2422375}. Interestingly, the authors were unaware of this at the time of defining Matching in the context of this paper.

$\ITM$s are a special class of piecewise isometries. There is a vast literature on such systems, and so we mention just a small subset of it: \cite{MR1905204}, \cite{MR1938473}, \cite{MR1772421}, \cite{MR1738947}, \cite{MR2091702}, \cite{MR1992662}, \cite{MR2039048}, \cite{MR2221800}, \cite{MR2486783}, \cite{MR3010377}, \cite{MR4032960}, \cite{MR4075314}, \cite{MR4441154}. In \cite{MR4082258}, the authors connected 1D and 2D piecewise isometries, by embedding $\IET$s into planar piecewise isometries.

\subsection{Organization of the paper}
\label{subsec:organisation}

The structure of the paper is as follows. In Section~\ref{sec:nw} we prove that $\overline{X}$ is equal to the non-wandering set of and a few other auxiliary results used throughout the paper. In Section~\ref{sec:theoremA} we introduce the product notation, define eventually periodic maps, and prove Theorem $A$. In Section~\ref{sec:lin-indep} we prove the crucial Theorem \ref{thm:lin-dep} about the linear dependence of critical itinerary vectors. In Section~\ref{sec:theoremB} we define the ACC and Matching conditions and prove Theorem $B$, using a special case of Theorem \ref{thm:lin-dep}, while in Section~\ref{proofoftheoremC} we prove Theorem $C$, using the full version of Theorem \ref{thm:lin-dep}. In Section \ref{sec:ae-fin-stable} we show that the set of stable maps has full measure in the set of finite type maps, which strengthens the Boshernitzan--Kornfeld Conjecture to the Irrational Rotations Conjecture. In Section~\ref{sec:future} we discuss the parameter space and possible future research. The connection between $\ITM$s and billiards is contained Appendix~\ref{appendix:billiards}. Finally, in Appendix~\ref{appendix:program} we give a link to the GitHub repository of the program written by Bj\"orn Winckler which allows one to experiment with the family of $\ITM$s introduced in \cite{MR2013352}.

\subsection*{Acknowledgments}
The first author is partially supported from grants 2021 SGR 00697 (Generalitat de Catalunya), PID2023-147252NB-I00 (AEI), CEX2020-001084-M (Maria de Maeztu Excellence program), and the ERC Advanced Grant ``SPERIG'' (\#885707). The second author acknowledges 
the support by the Roth PhD scholarship from Imperial College London. The third author acknowledges a partial sponsorship via 
Bj\"orn Winckler's Marie Curie postdoctoral fellowship \#743959.

\section{Preliminaries}
\label{sec:nw}

\subsection{Notation and conventions} 

Since the map $T$ is discontinuous, it will be useful to introduce the notion of \textit{signed points}. Consider the relation $\sim$ defined on all half-open intervals of contained $I$, such that $I_1 \sim I_2$ if and only if the left endpoint of $I_1$ is equal to the left endpoint of $I_2$. For a point $x \in I$, we define $x^+$ as the equivalence class of all intervals $[x,x+\epsilon)$, for $\epsilon > 0$ sufficiently small so that $[x,x+\epsilon) \in I$. We refer to points of this form as the $+$-type points. Analogously, we define $x^-$ as the equivalence class (under the relation of having the same right endpoint) of all intervals $[x-\epsilon,x)$, where $\epsilon$ is sufficiently small so that $[x-\epsilon,x) \in I$, and we refer to these points as the $-$-type points. We will also say that $x^+$ is the `$+$-part' of $x$ and that $x^-$ is the `$-$-part' of $x$. We use the term `geometric point' to refer to points contained in the geometric interval $I = [0,1)$ and the term `signed point' to refer to the equivalence class $x^+$ or $x^-$ of a point $x \in I$. We will simply use the term `point' if it does not matter what type we are considering.

We adopt the convention that the signed point $x^+$ is immediately to the right $x$, while $x^-$ is immediately to the left of $x$. With this convention, the definition of intervals contained in $I$ can be extended to signed points in the following way. By $[a,b]$, where $a$ and $b$ can be signed points or geometric points, we will denote the set of all points (signed or geometric) between $a$ and $b$, including $a$ and $b$. For example, the set $[x^-,y^-]$ contains $x^-,x,x^+$ and $y^-$, but not $y$ and $y^+$, and we will write $x^-, x, x^+, y^- \in [x^-,y^-]$ and $y, y^+ \notin [x^-,y^-]$. Analogously, $(a,b)$ will denote the set of all points between $a$ and $b$, excluding $a$ and $b$. We also define the sets $[a,b)$ and $(a,b]$ in the obvious way. We will refer to all of these sets as \textit{intervals}. Thus the term `intervals' will refer to these more general sets if we are dealing with signed points, and to subsets of $I$ if we are not. Moreover, we will use the convention that the intervals of $I$ are half-open and of the form $[a,b)$, unless stated otherwise.

We will write $a \sim b$ if the set $\{a,b\}$ is equal to $\{x^+, x^-\}$ for some point $x \in I$. We will also say that such a pair of signed points \textit{touches} (or \textit{is touching}).

The definition of any map $T$ can be extended to signed points in the following way. For a point $x \in I$, let $z_1 := \lim_{y \downarrow x}T^m(y)$ and $z_2:= \lim_{y \uparrow x} T^m(y)$. Since $T$ is discontinuous, $z_1^+ \sim z_2^-$ does not necessarily hold. Thus it makes sense to define:

\begin{align*}
T^m(x^+) &:= z_1^+ \\
T^m(x^-) &:= z_2^-.
\end{align*} 
The definition of limits of signed points is also straightforward. If $\lim_{i \to \infty} z_i = z$ for some sequence $(z_i) \in I$, then we can define $\lim_{i \to \infty} z_i^{\pm} := z^{\pm}$. Finally, the distance $d(a,b)$ between two signed points $a,b$ is defined as the distance between the geometric points in $I$ corresponding to $a$ and $b$.

We now define the \textit{critical set} $\mathcal{C}$ of $T$ as the following set of signed points:

\[
\mathcal{C} := \{\beta_1^-, \beta_1^+,\dots,\beta_{r-1}^-, \beta_{r-1}^+\}.
\]
We will refer to the elements of $\mathcal{C}$ as either the critical points or the discontinuities, depending on the context. We will sometimes use the labels $\beta_0^+ := 0^+$ and $\beta_r^- := 1^-$, but we do not consider them as critical points. We denote by $\mathcal{C}^+$ the set of all $+$-type points in $\mathcal{C}$, and by $\mathcal{C}^-$ the set of all $-$-\textit{type} points in $\mathcal{C}$. We will also continue to use the term `discontinuity' to refer to geometric points in $I$ at which the map $T : I \to I$ is discontinuous. If we want to highlight that we are not considering a signed point, we will use the term `geometric discontinuity'.

If a point $x \in I$ does not land on a discontinuity of $T$ up to some time $n > 0$, then for our purposes there is no difference between iterates $T^n(x^-), T^n(x)$ and $T^n(x^+)$. If a geometric point $x$ does land on a discontinuity of $T$ at some time $n$, then we will not consider iterates of $x$ for any time larger than $n$, but iterates of $x^+$ and $x^-$ instead. In particular, we do not consider iterates of any geometric discontinuity $\beta$, but only of $\beta^+$ and $\beta^-$. Thus whenever we are considering iterates of some discontinuity, it is assumed that it is a signed point.

Most of the time, we do not need to know the index of a discontinuity with respect to the order in $I$ nor whether the discontinuity is of $+$-type or $-$-type. That is why we will often use labels $\beta$, $\betas$ and $\betass$ to denote the discontinuities we are dealing with. If we care about the sign of a discontinuity, we will use the labels $\beta^+, \betas^-$, etc. In that case, we will use the same label without a sign to denote the geometric discontinuity in $I$ corresponding to the signed discontinuity, e.g. we denote by $\beta$ the geometric discontinuity such that $\beta^+$ is the $+$-part of $\beta$. We will use the notation $\text{ind}(\beta) \in \{0,1, \dots, r\}$ to mean the index of $\beta$ with respect to this order inside $I$.

The perturbation of (or a map sufficiently close to) some starting map $T$ will be denoted by $\Tilde{T}$, and by $\Tilde{Z}$ we will denote the continuation of some object of interest $Z$, e.g. $\Tilde{\beta}^+$, $\Tilde{X}$, $\Tilde{J}$ and similar.

We associate to each point $x \in I$ its \textit{itinerary}. Here it does not matter if a point is signed or geometric. The itinerary of $x$ is an infinite sequence of integers $(i_0(x), i_1(x), \dots, i_n(x), \dots)$, with $1 \le i_n(x) \le r$, where $i_n(x) = s$ means that $T^n(x) \in I_s$. We will often use `itinerary up to time $n$' to refer to the first $n$ elements of this sequence. 

Finally, for a point $x$ (can be either signed or geometric) we will refer to the set $O(x) := \{ x, T(x), T^2(x), \dots \}$ as the \textit{$T$-orbit} of $x$. Moreover, if $S$ is a subset of $I$, we will refer to the set $O(S) := \bigcup_{x \in S} O(x)$ as the $T$-orbit of $S$. If the map $T$ is clear from the context, we will omit it and simply use the term `orbit'. We will refer to the set $O(x,n) := \{ x, \dots, T^{n-1}{x} \}$, where $n \ge 0$, as the orbit up to time $n$ of $x$, and to the set $O(S,n) := \bigcup_{x \in S} O(x,n)$ as the orbit up to time $n$ of $S$. 

\subsection{The non-wandering set of $T$}

As usual, define the non-wandering set $NW(T)$ of $T$ to be the set of $x \in I$ so that for each neighbourhood $U$ of $x$ there exists $n>0$ so that $T^n(U)\cap U\ne \emptyset$. The set $X = \bigcap_{n=1}^{\infty} X_n$ is by definition the  \textit{attractor} of $T$. Here we prove that $\overline{X}$ is also the non-wandering set of $T$. We first prove a lemma about the structure of $X$ that is mentioned in several places in the literature (\cite{MR1356616}, \cite{MR2013352}, \cite{MR2308208}), but no explicit proof is available. 

\begin{lemma}
\label{lem:x-structure}
For any map $T$, we have that $\overline{X} = A_1 \cup A_2$, where $A_1$ is a finite union of intervals and $A_2$ is a Cantor set. One of the sets $A_1, A_2$ is allowed to be empty.
\end{lemma}

Recall that all of the intervals we consider are half-open and of the form $[a,b)$, unless stated otherwise. Before going into the proof, we introduce some standard definitions.

\begin{definition}[Return map]
\label{defn:rj}
Let $J$ be an interval contained in $X$. By $R_J$, we denote the first return map to $J$.
\end{definition}

It is easy to see that the domain of $R_J$ is the entire interval $J$, that it is bijective, and that $J$ is partitioned into finitely many maximal half-open subintervals such that no point in their interiors lands on a discontinuity of $T$. The return map $R_J$ is thus continuous on these intervals, so for simplicity, we will refer to these intervals as `continuity intervals of $R_J$'.

\begin{proof}[Proof of Lemma \ref{lem:x-structure}]
Let $A_1$ be the union of all maximal intervals contained in $X$. We first prove that $A_1$ always contains at most finitely many intervals. Indeed, we claim that $A_1$ is equal to the union of orbits of intervals of $X$ that contain discontinuities of $T$. More precisely, $A_1 = \bigcup_{\beta \in X} O(J_{\beta})$, where $J_{\beta}$ is the component interval of $X$ containing $\beta \in X$. Since the return map $R_J$ to any interval $J$ of $X$ always has finitely many branches, this proves the claim. Let $J$ be some interval of $X$. As a boundary point of any continuity interval of $R_J$ must at some point land onto a discontinuity in $R_J$, we have that each such interval eventually lands into an interval $J_{\beta}$. Thus every continuity interval must be contained in the orbit of some $J_{\beta}$, as $T$ is a bijection on any $O(J_{\beta})$. Thus the entire interval $J$ must be in $A_1$, which proves the claim.

Let $A_2 := \overline{X}\setminus A_1$, and assume that it is non-empty. Recall the well-known characterization of a topological Cantor set: it is a compact, metrizable and totally disconnected set with no isolated points. We now show $A_2$ satisfies all of these properties. Compactness follows by definition, while metrizability follows from it being a subset of $I$. It is totally disconnected because the image of any path between two different points in $A_2$ must be an interval in $A_2$, which is impossible as $\overline{X} \setminus A_1$ is nowhere dense. Indeed, if it were dense in some interval, then this entire interval must be contained in $X_n$ for all $n$, as $X_n$ is a finite union of intervals. This gives that this interval is in $X$ as well, so we get a contradiction with the definition of $A_2$.

Finally, we prove that there are no isolated points by showing that for any point $x \in A_2$, there is a sequence in $A_2$ accumulating on $x$. First of all, we know that $x$ is not periodic, as it would then be contained in a maximal periodic interval. As it is contained in $X$, the set of all of its preimages must therefore be infinite and contained in $A_2$. Thus its set of preimages must have an accumulation point $y$ in $A_2$. If the orbit of $y$ accumulates on $x$, we are done, so assume that this is not the case. This means that there exists a strictly monotone sequence $x_{n_k} \to y$ of $n_k$-preimages, with $n_1 < n_2 \dots$, such that the itinerary of $x_{n_k}$ and $y$ up to time $n_k$ is different, i.e. the iterates of the interval $[x_{n_k},y)$ before time $n_k$ contain a discontinuity. Since there are finitely many discontinuities, there is a discontinuity $\beta$ that occurs for infinitely many $k$. Thus the orbit of $y$ accumulates on $\beta$, so $\beta$ is contained in $A_2$ and the orbit of $\beta$ accumulates on $x$. Indeed, since $y^-$ is not periodic (preimages of $x$ would then be contained in a periodic interval), we know that in any finite time, the iterates of $y$ must have bounded non-zero distance between discontinuities. Thus the times at which the interval $[x_{n_k},y)$ lands on $\beta$ must be unbounded, so $x$ is not isolated.  
\end{proof}

\begin{lemma}
\label{lem:nonwandering} 
$\overline{X}$ is equal to the non-wandering set $NW(T)$.
\end{lemma}

\begin{proof}
We first prove that $\overline{X} \supseteq NW(T)$. Let $x$ be a point such that $x \notin \overline{X}$. This means that $\overline{\mathcal{U}} \cap \overline{X} = \emptyset$ for any sufficiently small neighbourhood $\mathcal{U}$ of $x$. Take one such sufficiently small open neighbourhood $\mathcal{U}_0$ of $x$. Then there exists $n_0 \ge 1$ such that $X_n$ has definite distance from $\mathcal{U}_0$ for all $n \ge n_0$, and thus we have that $T^n(\mathcal{U}_0) \cap \mathcal{U}_0 = \emptyset$ for all $n \ge n_0$, since $T^n(\mathcal{U}_0) \subset X_n$.

Next, we know that neither $x^+$ nor $x^-$ is periodic, since $X = \bigcap_{n=1}^{\infty} T^n(I)$. Because of this, for every $n \ge 1$, there exists a sufficiently small neighbourhood $\mathcal{U}_n$ of $x$ such that $T^i(\mathcal{U}_n) \cap \mathcal{U}_n = \emptyset$ for all $1 \le i \le n$. Let $\mathcal{U}_{n_0} \subset \mathcal{U}_0$ be such a sufficiently small neighbourhood. By our choice of $\mathcal{U}_0$, this means that $T^n(\mathcal{U}_{n_0}) \cap \mathcal{U}_{n_0} = \emptyset$ for all $n \ge 1$. Thus the point $x$ is not in the non-wandering set.

We now prove that $\overline{X} \subseteq NW(T)$. Let $x$ be a point in $\overline{X}$. By Lemma \ref{lem:x-structure}, we have that $\overline{X} = A_1 \cup A_2$, where $A_1$ is a finite union of intervals and $A_2$ is a Cantor set (we may assume both are non-empty). Thus for any sufficiently small neighbourhood $\mathcal{U}$ of $x$ we have that exactly one of the following holds:
\begin{enumerate}
    \item $\mathcal{U} \cap \overline{X}$ contains exactly one closed interval $J$ of definite length, and $x \in J$;
    \item $\mathcal{U} \cap \overline{X}$ contains no interval.
\end{enumerate} 
In the first case, we know that there must exist $n > m \ge 1$ such that $T^m(J) \cap T^n(J)$ contains an interval of definite length since $J$ has definite length and $T$ is a piecewise isometry. Since $X$ is $T$-invariant, $\text{int}(J) \subset X$ and $T|X$ is a bijection, we know that $T^{m-1}(J) \cap T^{n-1}(J)$ contains an interval of definite length as well. Thus we also have that $J \cap T^{n-m}(J) \neq \emptyset$, so $x \subset NW(T)$.

In the second case, we show that $x \in \omega(\beta)$ for some discontinuity $\beta \in \mathcal{C}$. This is sufficient, since all $\omega$-limit sets are contained in the non-wandering set, and thus $x \in NW(T)$ as well. Let $\mathcal{U}$ be a neighbourhood of $x$ as above. For any $n \ge 0$, we know that the boundary points of $X_n$ consist of iterates of discontinuities. If $x \notin \omega(\beta)$ for every $\beta \in \mathcal{C}$, then there is a definite distance between the boundary of $\overline{X_n}$ and $x$ for all sufficiently large $n$. Since $x$ is contained in every $\overline{X_n}$, an interval of definite length containing $x$ must also be contained in every $X_n$. Thus $\overline{X} \cap \mathcal{U}$ must contain an interval, which is a contradiction with our choice of $\mathcal{U}$. Thus $x \in \omega(\beta)$ for some discontinuity $\beta \in \mathcal{C}$, which finishes the proof.
\end{proof}  

\subsection{Properties of finite type $\ITM$s}

In this subsection, we state or prove several simple results that will be continuously used throughout the paper.

\begin{lemma}[Orbit Classification Lemma]
\label{lem:orb-class}
For each point $x \in I$, either singed or geometric, at least one of the following three possibilities holds:
\begin{enumerate}
\item (Precritical) $x$ lands on a critical point of $T$;
\item (Preperiodic) $x$ lands on a periodic point of $T$;
\item (Accumulation) There exists a critical point $\beta \in \mathcal{C}$ such that $\beta \in \omega(x)$.
\end{enumerate}
\end{lemma}

\begin{proof} Assume that none of the above holds for a point $x \in I$. Since the first and the third conditions do not hold, the orbit of $x$ stays a positive distance $\delta$ away from every critical point of $T$. This and the second condition give that the orbit of $x$ also stays at least $\delta$ away from every periodic point of $T$. Indeed, if the point $T^n(x)$, for some $n \ge 0$, is less than $\delta$ away from some periodic point $y$ of minimal period $p$, then the closed interval between $T^n(x)$ and $y$ cannot map forward continuously until time $p$, because this would mean that $T^n(x)$ is periodic. Thus some iterate of this interval must contain a discontinuity, which means that the orbit of $x$ gets less than $\delta$ close to $\mathcal{C}$, which is a contradiction.

Thus there exists a maximal interval $J_0$ of length at least $\delta$ and containing $x$ in its interior such that no point in this interval lands on a critical or a periodic point. Thus the interval $J_0$ maps forward continuously for all time, and every interval $T^i(J_0)$ has the property that none of its points land on critical or periodic points. Let $J_i$ be the maximal interval containing $T^i(x)$ such that no point in it lands on a critical or periodic point, and note that $T^i(J_0) \subset J_i$. Since all of the intervals $J_i$ have length at least $\delta$, there exists $n > m \ge 0$ such that $J_n \cap J_m \neq \emptyset$. By maximality, $J_n = J_m$ and thus $T^{n-m}(J_m) = J_m$. Since $T^{n-m}_{\vert J_m}$ is continuous and a translation, it must be the identity, so $T^{n-m}(T^m(x)) = T^m(x)$, which is a contradiction.
\end{proof} 

The following theorem was proven in \cite[Theorem 2.1]{MR1796167}:

\begin{theorem}[Characterization of Finite type]
\label{thm:fin-type-char}
An $\ITM$ $T$ is of finite type if and only if $X$ consists of a finite union of intervals. \qed
\end{theorem}

For a finite type map $T$, every component $J$ of $X$ is of the form $[x^+,y^-]$. We will refer to $x^+,y^-$ as the `boundary points of $X$' (or as the points in the boundary of $X$). The set of all boundary points of $X$ will be denoted by $\partial X$. The points in the set $\text{int}(X) := X \setminus \partial X$ will be referred to as the `internal points of $X$'. The next simple result describes the boundary and internal dynamics of an interval $J$ of $X$.  

\begin{lemma}
\label{lem:J-dynamics}  
Let $T$ be a finite type $\ITM$. Then the following two properties hold: 
\begin{enumerate}[label=(\alph*)]
    \item Every component of $X$ is of the form $[T^{k_1}(\beta^+), T^{k_2}(\betas^-)]$, respectively, for some $\beta^+, \betas^- \in X \cap \mathcal{C}$ and $k_1,k_2 \ge 0$;
    \item If $T^{l_1}(\beta)$ is an interior point of $X$ for some $\beta \in X \cap \mathcal{C}$ and $l_1 \ge 0$, then there exist $\betas \in X \cap \mathcal{C}$ and $l_2 \ge 0$ such that $T^{l_1}(\beta) \sim T^{l_2}(\betas)$.
\end{enumerate}
\end{lemma}

Recall that $T_{\vert X}$ is a bijection, so the inverse $T^{-1}_{\vert X}$ is well defined.

\begin{proof} 
Since $T$ is of finite type, the preimage of every boundary point of $X$ is either a boundary point of $X$ or a discontinuity of $T$. Fix $x$ in the boundary of $X$ and assume that $T^{-n}(x) \notin \mathcal{C}$ for every $n \ge 0$. Since there are finitely many boundary points, this means that $T^{-n}(x)$ is eventually periodic. Because $T_{\vert X}$ is a bijection, this means that $x$ is also periodic. By assumption, it does not land on a discontinuity. Thus there is an interval of points around $x$ that are all periodic with the same period. This contradicts $x$ being a boundary point of $X$, as all periodic points are contained in $X$. Thus $x = T^n(\beta)$ for some discontinuity $\beta \in \mathcal{C}$, which proves (a).

For (b), let $x$ be a point that is either a discontinuity contained in $X$ or a boundary point in $X$. Since $T$ of finite type and bijective on $X$, if the image of $x$ is contained in the interior of $X$, then there is point $y$ which is either a discontinuity or a boundary point of $X$ such that $T(x) \sim T(y)$. By part (a), we have that $T(y)$ is of the form $T^{l_2}(\betas)$ in both cases. 

Assume that $T^{l_1}(\beta) \in \text{int}(X)$ for some $n>0$, $\beta \in X \cap \mathcal{C}$. We may assume $\beta$ is the last discontinuity in $O(\beta, l_1)$. Thus $T^i(\beta) \notin \mathcal{C}$ for $0 < i < l_1$. We may further assume that $T^i(\beta) \notin \text{int}(X)$ for $0 < i < l_1$. Indeed, if some iterate $T^i(\beta)$, for $0 < i < l_1$, was in $\text{int}(X)$ and satisfied (b), then every iterate $T^j(\beta)$, for $i < j \le l_1$, is in $\text{int}(X)$ and also satisfies (b). This is because $\beta$ does not land on a discontinuity before time $l_1$. This means that $T^{l_1-1}(\beta)$ is in the boundary of $X$, so (b) once again follows from the paragraph above. 
\end{proof} 

\section{Product notation and eventually periodic maps}
\label{sec:theoremA} 

\subsection{Product notation and perturbations}
\label{subsec:prod}

In this subsection, we introduce the product notation, which will be used throughout the paper. It is a very convenient way to represent and control iterates of critical points.

For each  $T$,  $x \in I$, $s=1,\dots,r$ and $n\ge 1$, we define:
\[
k_s(x,n,T) := \# \{ T^i(x) \in I_s; 0\le i < n \}.
\]
Thus $k_s(x,n,T)$ represents the number of entries of $x$ into $I_s$ up to time $n$. This is an analogue of the first term in the continued fraction expansion for circle rotations, see \cite[Section 1.1]{MR1239171}. When the map $T$ is clear from the context, we will usually simply write $k_s(x,n)$.

Let $W(r) := \R{}^{r} \bigoplus \R{}^{r-1}$ be the $(2r-1)$-dimensional real vector space that contains the \textit{coefficient vectors} (introduced below), and let $(\bm{e}_1, \dots, \bm{e}_r)$ and $(\bm{f}_1, \dots, \bm{f}_{r-1})$ be the canonical bases for $\R{}^{r}$ and $\R{}^{r-1}$, respectively. Recall that $\ITM(r)$ is the parameter space of $\ITM$s on $r$ intervals, and that it is a convex polytope contained in $\mathbb{R}^{2r-1}$. We call the elements of this space \textit{parameter vectors}. These vectors also have canonical coordinates coming from the ambient space $\mathbb{R}^{2r-1}$. We will use the shorthand $(\gamma \, \beta)$ for a parameter vector $(\gamma_1 \dots \gamma_r \, \beta_1 \dots \beta_{r-1})$.

Let $\langle \cdot, \cdot \rangle$ be the standard scalar product on $\mathbb{R}^{2r-1}$. Since $W(r) = \R{}^{r} \bigoplus \R{}^{r-1}$ and $\ITM(r)$ is a subset of $\mathbb{R}^{2r-1}$, it makes sense to write $\langle v, (\gamma \, \beta) \rangle =  \sum_{s=1}^r v_s \gamma_s + \sum_{s=1}^{r-1} v_{s+r} \beta_s$ for a coefficient vector $v = \sum_{s=1}^r v_s \bm{e}_s + \sum_{s=1}^{r-1} v_{s+r} \bm{f}_s \in W(r)$ and a parameter vector $(\gamma \, \beta) \in \ITM(r)$. We call $\langle v, (\gamma \, \beta) \rangle$ the \textit{product} of $v$ and $(\gamma \, \beta)$.

The first reason for introducing this product is to obtain a formula for any iterate $T^n(\betas)$ of some discontinuity $\betas$. Define the vector of coefficients in $v(\betas,n,T) \in W(r)$ of this iterate as:

\[
v(\betas,n,T) \coloneqq \left(\sum_{1 \le s \le r} k_s(\betas, n, T) \, \bm{e}_s, \,  \bm{f}_{\text{ind}(\betas)}\right).
\]
Then the following holds:

\begin{align*}
T^n(\betas) &= \sum_{s=1}^r k_s(\betas,n,T) \gamma_s + \betas \\
&= \langle \, v(\betas,n,T) , (\gamma \, \beta) \, \rangle,
\end{align*}
where the vector $(\gamma \, \beta)$ corresponds to the defining coefficients of $T$. With this notation, we can associate to any landing of some discontinuity $\betas$ on some other discontinuity $\betass$ a coefficient vector in the following way. If $T^n(\betas) = \betass$, then the landing vector $L$ is defined as:

\[
L \coloneqq \left(\sum_{s = 1}^r \, k_s(\betas, n) \, \bm{e}_s, \, \bm{f}_{\text{ind}(\betas)} - \bm{f}_{\text{ind}(\betass)} \right).
\]
By the above formula for the iterate $T^n(\betas)$, we have that:

\begin{align*}
0 &= T^n(\betas) - \betass \\
&= \sum_{s=1}^r k_s(\betas,n,T) \gamma_s + \betas - \betass \\
&= \langle \, L , (\gamma \, \beta) \, \rangle.
\end{align*}

The second reason why the scalar product notation is useful is that it gives explicit formulas for how the $T$-iterates of points change after a perturbation $\delta$ of $T$. Let $\{ \bm{g_1}, \dots, \bm{g_r}, \bm{h_1}, \dots, \bm{h_{r-1}} \}$ be the canonical basis of the tangent space $\bm{T}_T \ITM(r)$ at point $T = (\gamma \, \beta)$ in $ITM(r)$. Each perturbation $\Tilde{T}$ of a map $T$ in $ITM(r)$ is given by a linear functional $\bm{\delta}: \bm{T}_T \ITM(r) \to \R$, so that $\Tilde{\gamma}_s = \gamma_s + \bm{\delta}(\bm{g_s})$ and $\Tilde{\beta}_s = \beta_s + \bm{\delta}(\bm{h_s})$. Thus $\Tilde{T}$ is obtained from $T$ by perturbing in the direction $\delta := (\bm{\delta}(\bm{g}_1) \, \dots \, \bm{\delta}(\bm{g}_r) \, \bm{\delta}(\bm{h_1}) \, \dots \, \bm{\delta}(\bm{h}_{r-1})) \in \bm{T}_T \ITM(r)$. Hence, after identifying $\bm{T}_T \ITM(r)$ and $\mathbb{R}^{2r-1}$, it makes sense to write $(\gamma \, \beta) + \delta = (\Tilde{\gamma}\, \Tilde{\beta})$.

If the itineraries of $\beta$ under $T$ and $\Tilde{\beta}$ under $\Tilde{T}$ are the same up to some finite time $n \ge 1$, then we have the following formula: 

\begin{align*}
\Tilde{T}^n(\Tilde{\beta}) - T^n(\beta) &= \sum_{s = 1}^r k_s(\beta,n) \Tilde{\gamma}_s + \Tilde{\beta} - \sum_{s = 1}^r k_s(\beta,n) \gamma_s - \beta \\
&= \sum_{s = 1}^r k_s(\beta,n) (\Tilde{\gamma}_s - \gamma_s) + \Tilde{\beta} - \beta \\
&= \sum_{s = 1}^r k_s(\beta,n) \bm{\delta}(\bm{g_s}) + \bm{\delta}(\bm{h_{\text{ind}(\beta)}}) \\
&= \langle \, v(\beta,n,T) , \delta \,  \rangle.
\end{align*}

Thus if we want $T^n(\betas) = \betass$ to hold after perturbation as well, we must have that (assuming the itinerary of $\betas$ up to time $n$ does not change):

\[
\langle \, L , \delta \, \rangle = 0.
\]
Product notation can also be used to define parameter families of $\ITM$s. Assume that the parameter vector $(\gamma \, \beta)$ of $T$ satisfies some equation:

\begin{align*}
\sum_{i=1}^r a_i \gamma_i + \sum_{i=1}^{r-1} b_i \beta_i = 0,
\end{align*}
where $a_,b_i \in \mathbb{R}$. This is equivalent to saying:

\begin{equation}
\label{eq:param-eq}
\langle \, v , (\gamma \, \beta) \, \rangle = 0    
\end{equation}
for the associated vector of coefficients $v := (a_1, \dots, a_r, b_1, \dots, b_{r-1})$. By a \textit{parameter family}, we will refer to the subspace of $\ITM(r)$ consisting of all parameter vectors for which equation \eqref{eq:param-eq} holds for all vectors $v$ in some set $\mathcal{E}$. For a map $T$ and a set of vectors $\mathcal{E}$, we will write $\langle \, \mathcal{E} , T \, \rangle = 0$ if the parameter vector corresponding to the map $T$ satisfies \eqref{eq:param-eq} for all vectors $v \in \mathcal{E}$.

We will say that a family is \textit{rational} if all of the vectors in $\mathcal{E}$ have rational coefficients. They are important because of the following simple lemma:

\begin{lemma}
\label{lem:rat-fam}
Parameters with rational coefficients are dense in rational families.
\end{lemma}

 \begin{proof} 
 Recall that the set $IMT(r)$ of all $\ITM$s on $r$ intervals can be viewed as a closed convex polytope in $\mathbb R^{2r-1}$. Therefore, a family $\mathcal T$ given by a set of vectors $\mathcal E$ can be viewed as a convex polytope of co-dimension $t=\text{dim}(\text{span}(\mathcal E))$ as the intersection of $\ITM(r)$ and the linear sub-space $E$ spanned by $\mathcal E$. If $\mathcal T$ is now a rational family, then there exists an invertible matrix $A$ with rational coefficients that maps $E$ to a coordinate sub-space, say, $x_1 =\ldots=x_{2r-1-t}=0$ (where we denoted by $x_i$ some Cartesian coordinates in $\mathbb R^{2r-1}$). It is clear that the points with rational coordinates are dense in the image $A(\mathcal T)$. Since the matrix $A$ has rational coefficients, so does $A^{-1}$, and thus the points with rational coordinates are dense in $A^{-1}\cdot A(\mathcal T)=\mathcal T$. Hence, parameters with rational coefficients are dense in the rational family $\mathcal T$.
\end{proof}

\subsection{Density of eventually periodic maps}

We say that $T$ is \textit{eventually periodic} if every point in $\mathcal{C}$ (recall that this is a set of signed points) is eventually periodic. Eventually periodic maps are the simplest non-trivial maps, an analogue of Morse-Smale systems from smooth interval dynamics (see \cite{MR1239171}). We now prove the parametric version of Theorem A:

\begin{theorem}[Eventually periodic maps are dense - Parametric version]
\label{thm:ep-dense-param}
Let $T$ be a map contained in a rational family $\mathcal{T}$ defined by a set of vectors $\mathcal{E}$. Then arbitrarily close to $T$ in $\mathcal{T}$, there is an eventually periodic map $\tilde{T}$.
\end{theorem}

This theorem is a straightforward consequence of the fact that the maps with rational parameters are dense. Recall that the distance between signed points is defined as the distance of geometric points in $I$ that the signed points correspond to.

\begin{proof}
Because of Lemma \ref{lem:rat-fam}, we only need to prove that maps $T$ for which the parameter vector $(\gamma \beta)$ has rational coefficients are eventually periodic. In fact, it is enough that only the coefficients $\gamma_1, \dots, \gamma_r$ are rational. Let $\mathbb{N}(\gamma)$ be the set of all linear combinations of $\gamma_1, \dots, \gamma_r$ with integer coefficients. This set is discreet, because the minimal distance between any of its elements, if non-zero, is at least $\frac{1}{q^r}$, where $q$ is the largest denominator of the numbers $\gamma_1, \dots, \gamma_r$. Thus for any $\beta \in \mathcal{C}$, we have that the set $\beta + \mathbb{N}(\gamma) \cap [0,1)$ is finite. As any iterate of $\beta$ is contained in this set, we must have that $\beta$ is eventually periodic.
\end{proof}

\begin{definition}
\label{defn:c1c2}
For $T \in \ITM(r)$, we define the following sets:
\begin{itemize}
    \item $\mathcal{C}_1$ is the set of all $\beta \in \mathcal{C}$ that eventually land on a discontinuity;
    \item $\mathcal{C}_2$ is the set of all $\beta \in \mathcal{C}$ that never land on a discontinuity, but are eventually periodic;
    \item $\mathcal{C}_0 = \mathcal{C}_1 \cup \mathcal{C}_2$;
    \item $\mathcal{C}_i^{\pm} = \mathcal{C}^{\pm} \cap \mathcal{C}_i$, for $i \in \{ 0, 1, 2 \}$.
\end{itemize}
\end{definition}

It is easy to see $\mathcal{C} = \mathcal{C}_0$ if and only if $T$ is eventually periodic. The next proposition shows that a stronger property holds: if the $+$-type discontinuities are eventually periodic, then so is every other point in $I$. Analogously, it is also enough that the $-$-type discontinuities are eventually periodic.

\begin{proposition}
\label{prop:+enough}
If $\mathcal{C}^+ = \mathcal{C}^+_0$, then $T$ is eventually periodic. Moreover, every point in $I$ is eventually periodic.
\end{proposition}

Thus this proposition shows that the name `eventually periodic' is justified.

\begin{proof}
We first prove that there exists an $\epsilon > 0$ such that for each $\beta^+ \in \mathcal{C}^+$, the interval $[\beta, \beta+\epsilon)$ is eventually periodic and maps forward continuously for all times. Let $N$ be the smallest time such that $T^N(\beta^+)$ is a periodic point for all $\beta^+$, and let $p$ be the maximal period of these points. Let $\epsilon > 0$ be the closest positive distance the orbit of any $\beta^+$ up to time $N+p$ gets to the left of a point in $\mathcal{C}^+$. Note that $\beta^+$ can land on another critical point, but in our definition of $\epsilon$ we ignore these times and look only at the positive distances. We claim that this is our required $\epsilon$. Indeed, if the iterates of any interval $[\beta, \beta + \epsilon)$ contain a discontinuity before time $N+p$, it would mean that the distance between a point in the orbit of $\beta^+$ and another critical point would be less than $\epsilon$, contradicting the definition of $\epsilon$. Thus, every such interval maps forward continuously up to time $N+p$. Since $T^N(\beta^+)$ is periodic and has period $\le p$, we see that the entire interval $[T^N(\beta), T^N(\beta) + \epsilon)$ consists of points of the same period, which shows that $\epsilon$ has the required property.

Now assume, for the sake of contradiction, that there is some $x \in I_s$ which is not eventually periodic. By assumption, $\beta_{s-1}^+$ is eventually periodic. We define a finite sequence of eventually points $x_0 = \beta_{s-1}^+ < x_1 < \dots < x_t < x$ and times $n_0 = 0 < n_1 < \dots < n_t$ such that $T^{n_i}(x_i) \in \mathcal{C}$ and $d(T^{n_i}(x_i), T^{n_{i-1}}(x_{i-1})) \ge \epsilon$ for all $1 \le i \le t$ by the following inductive procedure. 

If $x_i$ is already $\epsilon$-away from $x$, we stop the procedure. Otherwise, define $n_{i+1}$ be the minimal time the orbit of the interval $[x_i, x)$ contains a critical point. Such a time exists because $x$ and $x_i$ must have different itineraries, as $x_i$ is eventually periodic, while $x$ is not. Let $x_{i+1}$ be the rightmost point in $[x_i, x)$ which lands on a critical point at time $n_{i+1}$. We have that $n_{i+1} > n_i$, because $x_i$ is the rightmost point in $[x_{i-1}, x)$ which lands on a critical point by time $n_i$, which means that the interval $[x_i,x)$ maps forward continuously up to time $n_i$. Moreover, we have that $n_{i+1} - n_i \le N + p$, as $x_i$ becomes periodic with period at most $p$, after at most $N$ iterations. We claim that $x_{i+1} - x_i \ge \epsilon$. Indeed, by construction, we have that $T^{n_i}([x_i,x)) = [\beta, y)$, for some $\beta^+ \in \mathcal{C}^+$. The interval $[\beta, \beta + \epsilon)$ maps forward continuously for all times, which means that only the points in $[\beta + \epsilon, y)$ can land on critical points. Thus $x_{i+1} - x_i \ge \epsilon$, which means that the inductive procedure must terminate at some finite time. Since $d(x_t, x) \le \epsilon$ and this interval maps forward continuously up to time $n_t$ when $x_t$ lands on some critical point $\beta^+$, we see that all points in this interval must be eventually periodic, contradicting our choice of $x$.
\end{proof}

An important corollary of the proof of this proposition is that eventually periodic maps are of finite type. Thus Lemma $1$ already gives that the finite type maps are dense in $\ITM(r)$.

\begin{corollary}
\label{cor:ep-fin-type}
Eventually periodic maps are of finite type.
\end{corollary}

\begin{proof}
The second part of the proof of Proposition \ref{prop:+enough} shows that under iterates of an eventually periodic map, each point in $I$ eventually lands into one of the finitely many periodic intervals. Thus $X$ must be equal to the union of orbits of these periodic intervals, which shows that it is of finite type by Theorem \ref{thm:fin-type-char}. 
\end{proof}

\section{Linear independence of itinerary vectors}
\label{sec:lin-indep}

\subsection{Statement and setup}

The crucial property that allows us to change one part of the dynamics while keeping the other parts as they are is the linear independence of itinerary vectors of discontinuities. This is the content of Theorem \ref{thm:lin-dep} and Corollary \ref{cor:lin-indep}. To state these results precisely, we need some preliminary results and notation. When introducing objects that depend on some interval $J$, the dependence will be explicit in the notation if the object in question will be used in the context when there are several $J$-intervals, and it will be implicit otherwise.

Let $J = [x^{J,+},y^{J,-}]$ be an interval contained in $X$. Recall that the return map $R_J$ to $J$ is well-defined (see \ref{defn:rj}) and that there are finitely many points $a_{1}^{J}, \dots, a_{N_J-1}^J$ in the interior of $J$ that land on discontinuities before returning to $J$. Let $J_1^J, \dots, J_{N_J-1}^J$ be the maximal intervals in $J$ such that no point in their interior lands on a discontinuity, and let $r_j^J$ be the return time of $J_j^J$ to $J$. Let $a_{0}^{J,+} := x^+$ and $a_{N_J}^{J,-} := y^-$ be the boundary points of $J$. For each $1 \le j \le N_J-1$, let $m_{j}^{J,+}$ be the number of discontinuities that $a_{j}^{J,+}$ lands on before returning to $J$ and, for $1 \le k \le m_{j}^{J,+}$, let $\beta^{J,+}(j,k)$ be the $k$-th discontinuity along the orbit up to return time to $J$ of $a_j^{J,+}$. Define $m_{j}^{J,-}$ and $\beta^{J,-}(j,k)$ similarly, for $1 \le j \le N_J-1$ and $1 \le k \le m_{j}^{J,-}$. Note that for each $1 \le j \le N_J$, we have that $\beta^{J,+}(j,1)$ and $\beta^{J,-}(j,1)$ are the $+$ and $-$ part of a single discontinuity, which we denote by $\beta^J(j)$. Figure \ref{fig:points-in-J-orbit} illustrates these points for an interval $J$ for which $R_J$ has three branches and $\sigma = (3 2 1)$.

\begin{figure}
    \centering
    \includegraphics[width=\linewidth]{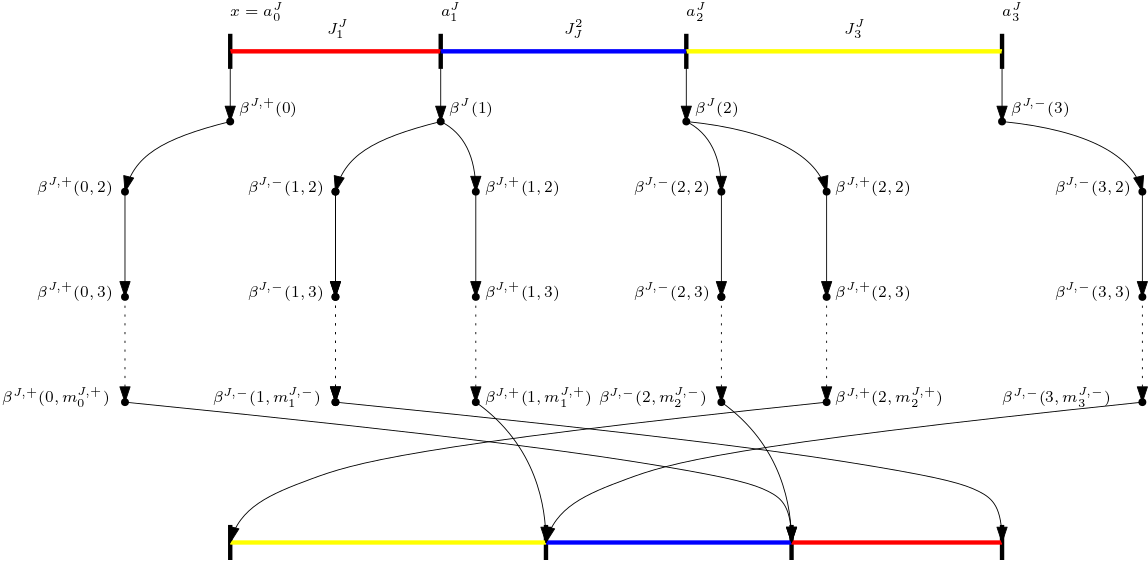}
    \caption{All dynamically important points in the orbit of an interval $J$ for which the return map $R_J$ has three branches.}
    \label{fig:points-in-J-orbit}
\end{figure}

Recall that $W(r) = \R{}^{r} \bigoplus \R{}^{r-1}$ is the space of coefficient vectors, and that $(\bm{e}_1, \dots, \bm{e}_r)$ and $(\bm{f}_1, \dots, \bm{f}_{r-1})$ are the canonical bases for $\R{}^{r}$ and $\R{}^{r-1}$, respectively.

We now define three types of dynamically defined vectors associated to the orbit of $J$. The first ones are the \textit{L-vectors}, that correspond to the first landings of points $a_1^J, a_2^J, \dots, a_{N_J-1}^J$ to discontinuities:

\begin{definition}[Landing vectors]
For $1 \le j \le N_J-1$, let $l_{j}^J$ be the landing time of $a_{j}^J$ to $\beta^J(j)$ and let $L_{j}^{J} \in W(r)$ be the associated landing vector:

\[
L_{j}^{J} \coloneqq \left(\sum_{s = 1}^r k_s(a_j^J, l_{j}^J) \, \bm{e}_s, \, - \bm{f}_{\text{ind}(\beta^J(j))}\right).
\]
\end{definition}

Recall that ind$(\beta)$ is the index of $\beta$ with respect to the order of critical points inside of $I$. Note that if $a_j$ is a discontinuity of $T$, i.e. if the landing time $l_j^J$ is zero, then $L_j^{J} = \left( 0, - \bm{f}_{\text{ind}(\beta^J(j))}\right)$ by definition. We call $L_j^{J}$ the landing vector of $a_j^J$ to $\beta^J(j)$ because the following holds:

\begin{align*}
0 &= a_j^J + \sum_{s = 1}^r k_s(a_j^J, l_{j}^J) \gamma_s - \beta^J(j) \\
&= \left(\sum_{s = 1}^r k_s(a_j^J, l_{j}^J) \gamma_s - \beta^J(j)\right) + a_j^J \\
&= \langle L_j^{J}, (\gamma \, \beta) \rangle + a_j^J.
\end{align*}

The second type of vectors we need to consider are the \textit{C-vectors}, which correspond to landings of discontinuities in the orbit of $J$ onto other discontinuities (which are thus also in the orbit of $J$).

\begin{definition}[Critical connection vectors]
For each $1 \le j \le N_J-1$ and $1 \le k < m_{j}^{J,+}$, let $q^{J,+}(j,k)$ be the landing time of $\beta^{J,+}(j,k)$ to $\beta^{J,+}(j,k+1)$ and let $C^{J,+}(j,k) \in W(r)$ be the associated landing vector:

\[
C^{J,+}(j,k) \coloneqq \left(\sum_{s=1}^r k_s(\beta^{J,+}(j,k), q^{J,+}(j,k)) \, \bm{e}_s, \, \bm{f}_{\text{ind}(\beta^{J,+}(j,k))} - \bm{f}_{\text{ind}(\beta^{J,+}(j,k+1))}\right).
\]
\end{definition}%
The following holds by construction:

\begin{align*}
0 &= \beta^{J,+}(j,k) + \sum_{s=1}^r k_s(\beta^{J,+}(j,k), q^{J,+}(j,k)) \gamma_s - \beta^{J,+}(j,k+1) \\
&= \sum_{s=1}^r k_s(\beta^{J,+}(j,k), q^{J,+}(j,k)) \gamma_s + \beta^{J,+}(j,k) - \beta^{J,+}(j,k+1) \\
&= \langle C^{J,+}(j,k), (\gamma \, \beta) \rangle.
\end{align*}
Define $C^{J,-}(j,k) \in W(r)$ analogously for $1 \le j \le N_J-1$ and $1 \le k < m_{j}^{J,-}$.

The third type of dynamical vector we need to consider are the \textit{R-vectors}, which correspond to the return of points $a_1^{J,-}, a_1^{J,+}, \dots, a_{N_J-1}^{J,+}, a_{N_J-1}^{J,-}$ to the interval $J$.

\begin{definition}[Return vectors]
For each $1 \le j \le N_J-1$, let $r_{j}^{J,+}$ be the time at which $\beta^+(j,m_{j}^{J,+})$ lands into $J$ and let $R_{j}^{J,+} \in W(r)$ be the landing vector:

\[
R_{j}^{J,+} \coloneqq \left(\sum_{s=1}^r k_s(\beta^{J,+}(j,m_{j}^{J,+}), r_{j}^{J,+}) \, \bm{e}_s, \, \bm{f}_{\text{ind}(\beta^{J,+}(j,m_{j}^{J,+}))}\right).
\]
\end{definition}

Thus we have that:
\[
\langle R_j^{J,+}, (\gamma \, \beta) \rangle = \sum_{s=1}^r k_s(\beta^{J,+}(j,m_{j}^{J,+}), r_{j}^{J,+}) \gamma_s + \beta^+(j,m_{j}^{J,+}) \in J.
\]
Define $r_{j}^{J,-}$ and $R_{j}^{J,-} \in W(r)$ analogously for $1 \le j \le N_J-1$.

For the boundary points $x^{J,+} = a_0^{J,+}$ and $y^- = a^{J,-}_{N_J}$, the definitions of the corresponding dynamical vectors depend on whether they land on discontinuities of $T$ before returning to $J$ or not. If $a_0^{J,+}$ lands on a discontinuity of returning to $J$, then we may analogously as for other $a_j^{J,+}$, where $1 \le j \le N_J-1$, define the vectors $L_0^J$, $C^{J,+}(0,k)$ and $R^{J,+}_0$. In the case when $a_0^{J,+}$ does not land on a discontinuity of $T$ before returning to $J$, we only define the return vector $R_0^{J,+}$ to $J$:

\[
R_0^{J,+} := \left(\sum_{s=1}^r k_s(a_0^{J,+}, r_{0}^{J,+}) \, \bm{e}_s, \, 0 \right),
\]
where $r_0^+$ is the return time of $a_0^+$ to $J$. Analogously, if $a_{N_J}^{J,-}$ lands on a discontinuity of $T$, we may define the vectors $L_{N_J}^J, C^{J,-}(N_J,k)$ and $R^{J,-}_{N_J}$. If $a_{N_J}^{J,-}$ does not land on a discontinuity of $T$ before returning to $J$, then define $R_{N_J}^{J,-}$ analogously as for $a_0^{J,+}$. 

Because the definitions of the dynamical vectors associated to the boundary points the interval $J_0$ from Theorem \ref{thm:lin-dep} are different depending on whether they land on discontinuities of $T$ or not, this leads to different statements of Theorem \ref{thm:lin-dep}, depending on it these landings happen or not. For simplicity, we will assume that these boundary points land on discontinuities of $T$, because this case is more complicated, and we state Theorem \ref{thm:lin-dep} with this assumption. The proof of Theorem \ref{thm:lin-dep} in the case when at least one of the boundary points does not land on a discontinuity is analogous. Moreover, we also assume that all of the points $a_j^{0,\pm}$, for $0 \le j \le N_0$, land on at least two discontinuities before returning to $J_0$. If this is not true for some $a_j^{0,\pm}$, then there are no critical connection vectors associated to this point, leading to a different statement of Theorem \ref{thm:lin-dep}. As this case is also simpler, we will assume that there are always critical connection vectors associated to every $a_j^{0,\pm}$. All of these vectors are illustrated in Figure \ref{fig:vectors-in-J-orbit}, for the same interval $J$ as in Figure \ref{fig:points-in-J-orbit}.

\begin{figure}
    \centering
    \includegraphics[width=\linewidth]{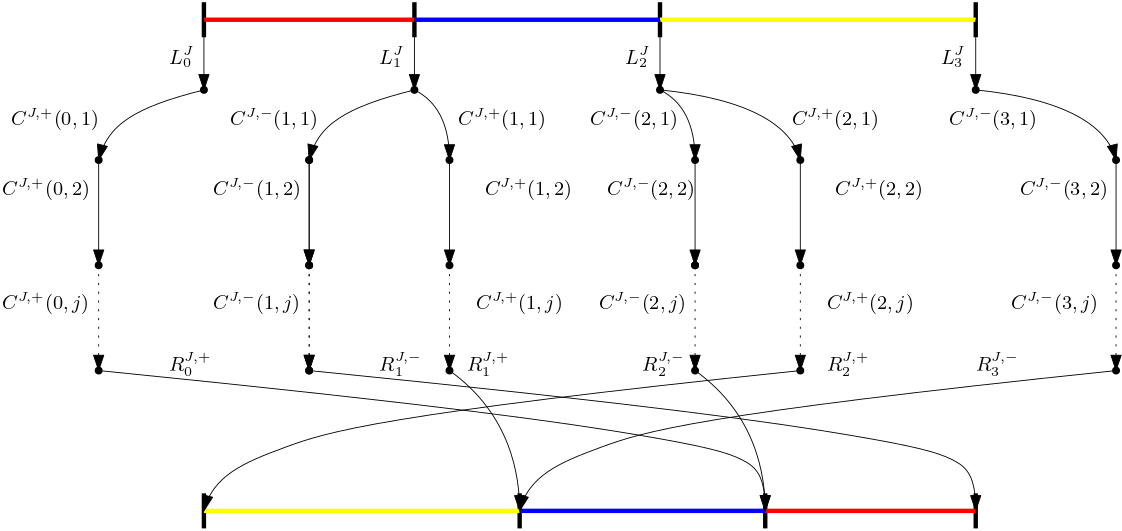}
    \caption{Vectors associated to the orbit of an interval $J$ for which the return map $R_J$ has three branches.}
    \label{fig:vectors-in-J-orbit}
\end{figure}

We would also like to have control over the dynamics of the discontinuities not contained in $X$. For this purpose, we also introduce the following more general notation for a critical connection vector. Recall from Section \ref{sec:theoremA} that we say that $\beta \in \mathcal{C}_1$ if $\beta$ eventually lands on a discontinuity. Let $\beta'$ denote the first such discontinuity in the orbit of $\beta$, and let $q(\beta)$ be the landing time of $\beta$ to $\beta'$. Note that $\beta$ and $\beta'$ are signed discontinuities, by the definition of $\mathcal{C}_1$ \eqref{defn:c1c2}.

\begin{definition}[General critical connection vectors]
Let $\beta$ be a discontinuity of $T$ which eventually lands on another discontinuity. Define the vector $v_{\beta} \in W$ as:

\[
C_{\beta} := \left(\sum_{s=1}^r k_s(\beta,q(\beta))\bm{e}_s, \, \bm{f}_{\text{ind}(\beta)} - \bm{f}_{\text{ind}(\beta')}\right),
\]
\end{definition}

By definition, we have that:

\begin{align*}
0 &= \beta + \sum_{s=1}^r k_s(\beta,q(\beta)) \gamma_s - \beta' \\
&= \langle C_{\beta}, (\gamma \, \beta) \rangle
\end{align*}
Recall also that $\mathcal{C}_2$ denotes the set of all discontinuities that never land on other discontinuities, but are eventually periodic. The dynamically important vector related to these discontinuities is the landing vector of such a discontinuity into $X$. We do not include these vectors in Theorem \ref{thm:lin-dep}. The reason is that they land into interiors of intervals which we can control with Theorem \ref{thm:lin-dep}. Thus we can already control the dynamics of these discontinuities with this theorem, without the need to include them into the statement.

In the statement of Theorem \ref{thm:lin-dep}, there are two types of intervals contained in $X$. There is a distinguished interval $J_0$, which can be any interval contained in $X$, and there are in intervals $J_1, \dots, J_n$ that are \textit{maximal periodic intervals}. Here is the definition:

\begin{definition}
\label{defn:max-per-int}
A maximal interval $J \subset I$ consisting of periodic points with the same itinerary is called a \textit{maximal periodic interval}. It is clear that $J$ is a half-open interval and that its boundary points have the property that they land on discontinuities of $T$ or on the boundary points of $I$. Moreover, no point in the interior of $P$ lands on a discontinuity. 
\end{definition}

To ease the notation, we will not explicitly state the ranges for all indices that appear in the statement, as they are the same as in the discussion above. Moreover, we will abbreviate the dependence of vectors and numbers on an interval $J_i$, for $0 \le i \le n$, to the dependence on the number $i$, e.g. we will write $C^{i,+}(j,k)$ instead of $C^{J_i,+}(j,k)$. 

For the maximal periodic intervals $J_1, \dots, J_n$, the only points that land on discontinuities are the boundary points, so $N_{i}= 1$ for $1 \le i \le n$. For each $J_i$, we will consider the vector $C^{i,+}(0,m_0^{i,+}) := R^{i,+}_0 + L^{i}_0$, instead of the landing vector $L^{i}_0$ and the return vector $R^{i,+}_0$. The reason for this is that the vector $C^{i,+}(0,m_0^{i,+})$ is equal to the critical connection vector corresponding to the landing of $\beta^{i,+}(0,m_0^{i,+}-1)$ to $\beta^{i,+}(0,1)$. This is the special property of maximal periodic intervals: their return and landing vectors combine into a critical connection vector, which is not true in general. This will be crucially used in the proof of Proposition \ref{prop:refine}, which is one of the main steps in the proof of Theorem \ref{thm:lin-dep}. We define $C^{i,-}(1,m_1^{i,-})$ analogously. 

\begin{theorem}[Coefficients of Linear Dependence]
\label{thm:lin-dep}
Let $T$ be an $\ITM$. Let $J_0, J_1, J_2, \dots, J_n$ be intervals contained in $X$ that have pairwise disjoint orbits, and let $\mathcal{C}_{\neg X}$ be the set of all discontinuities in $\mathcal{C}_1$ that are not contained in $X$. Assume that $J_0$ is a component interval of $X$ and that $J_1, \dots, J_n$ are maximal periodic intervals. As above, let the following be the vectors related to their dynamics:

\begin{lstv}
\label{eq:vectors}
\begin{itemize}
    \item $L^{0}_j, R^{0,+}_j, R^{0,-}_j, C^{i,+}(j,k), C^{i,-}(j,k)$;
    \item $C_{\beta},\text{ for } \beta \in \mathcal{C}_{\neg X}$.
\end{itemize}
\end{lstv}
Assume that there exist real coefficients:

\begin{lsta}
\label{eq:a-coefficients}
\begin{itemize}
    \item $\alpha_{j}^{0}, \alpha_j^{0,+}, \alpha_j^{0,-}, \alpha^{i,+}(j,k), \alpha^{i,-}(j,k)$;
    \item $\alpha_{\beta}, \text{ for } \beta \in \mathcal{C}_{\neg X}$.
\end{itemize}
\end{lsta}
such that the vectors in \hyperref[eq:vectors]{(*)} satisfy:

\begin{align}
\label{eq:lin-dep-sum}
\begin{split}
&\sum_{i=1}^n \left( \sum_{k=1}^{m_0^{i,+}} \alpha^{i,+}(0,k) C^{i,+}(0,k) \right) + \sum_{i=1}^n \left( \sum_{k=1}^{m_1^{i,-}} \alpha^{i,-}(1,k) C^{i,-}(1,k) \right) \\
+&\sum_{j=0}^{N_0-1} \left( \sum_{k=1}^{m_j^{0,+}-1} \alpha^{0,+}(j,k) C^{0,+}(j,k) + \alpha_j^{0,+} R_j^{0,+} \right) \\ 
+&\sum_{j=1}^{N_0} \left( \sum_{k=1}^{m_j^{0,-}-1} \alpha^{0,-}(j,k) C^{0,-}(j,k) + \alpha_j^{0,-} R_j^{0,-} \right) \\
+& \left( \sum_{j=0}^{N_0} \alpha_j^0 L^0_j \right) + \sum_{\beta \in \mathcal{C}_{\neg X}} \alpha_{\beta} \mathcal{C}_{\beta} = 0.    
\end{split}
\end{align}
Then it must hold that:
\begin{align}
\label{eq:lin-dep-equality1}
\begin{split}
&\alpha^{i,+}(0,1) = \dots = \alpha^{i,+}(0,m_0^{i,+}) = -\alpha^{i,-}(1,1) = \dots = - \alpha^{i,-}(1,m_{1}^{i,-}); \\
&\alpha^{0,+}(j,1) = \dots = \alpha^{0,+}(j,m_j^{0,+}-1) = a_j^{0,+} \\
&=-\alpha^{0,-}(j+1,1) = \dots = - \alpha^{0,-}(j+1,m_{j+1}^{0,-}-1) = -\alpha_{j+1}^{0,-},
\end{split}
\end{align}
for all $1 \le i \le n$ and $0 \le j \le N_0-1$, respectively. Moreover, we have that:
\begin{align}
\label{eq:lin-dep-equality2}
\begin{split}
& \alpha_0^0 = \alpha^{0,+}(0,1); \\
& \alpha_{N_0}^0 = \alpha^{0,-}(N_0,1); \\
& \alpha_j^0 = \alpha^{0,+}(j,1) + \alpha^{0,-}(j,1); \\
& \alpha_{\beta} = 0,
\end{split}
\end{align}
for all $1 \le j \le N_0-1$ and for all $\beta \in \mathcal{C}_{\neg X}$, respectively.
\end{theorem}

\begin{figure}
    \centering
    \includegraphics[width=\linewidth]{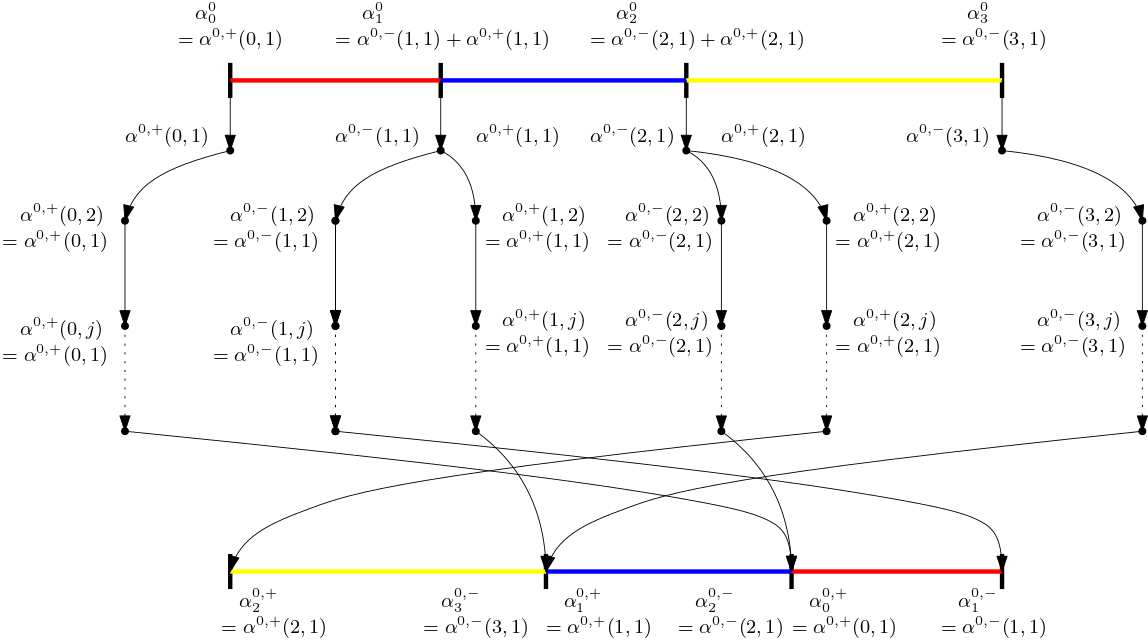}
    \caption{The coefficients for the landing vectors and critical connection vectors are at the start of their corresponding arrows, while the coefficients for the return vectors are at the end. The indicated equalities between the coefficients come from the conclusion of Theorem \ref{thm:lin-dep}.}
    \label{fig:return-map-vectors}
\end{figure}


The conclusion of Theorem \ref{thm:lin-dep} for $J_0$ (which is in fact the same interval as in Figures \ref{fig:points-in-J-orbit} and \ref{fig:vectors-in-J-orbit}) with three branches of the return map is demonstrated in Figure \ref{fig:return-map-vectors}. The assumption $J_1, \dots, J_n$ are maximal periodic intervals is necessary for our proof to work. A natural question to ask is whether this is necessary.

\begin{question}
Does Theorem \ref{thm:lin-dep} hold if we assume that $J_1, \dots, J_n$ are component intervals of $X$? 
\end{question}

The proof of this theorem is split into the next three subsections. The important consequence of Theorem \ref{thm:lin-dep} is the following corollary:

\pagebreak

\begin{corollary}[Linear independence of dynamical vectors]
\label{cor:lin-indep}
Let $T$ be an $\ITM$. Let $J_0, J_1, J_2 \dots, J_n$ be intervals contained in $X$ with pairwise disjoint orbits. Assume that $J_0$ is a component interval of $X$ and that $J_1, \dots, J_n$ are maximal periodic intervals. Then the vectors:

\begin{itemize}
    \item $L_{j}^0, R^{0,+}_j, C^{0,+}(j,k), C^{0,-}(j,k)$;
    \item $C^{i,+}(0,k)$ for $1 \le k \le m_0^{i,+}$ and $C^{i,-}(1,k)$ for $1 \le k < m_0^{i,-}$;
    \item $C_{\beta}$, for $\beta \in \mathcal{C}_{\neg X}$,
\end{itemize}
form a linearly independent set.
\end{corollary}

\begin{proof}
For the sake of contradiction, assume that they are linearly dependent with coefficients:
\begin{itemize}
    \item $\alpha_{j}^{0}, \alpha_j^{0,+}, \alpha^{0,+}(j,k), \alpha^{0,-}(j,k)$;
    \item $\alpha^{i,+}(0,k)$ for $1 \le k \le m_0^{i,+}$ and $\alpha^{i,-}(1,k)$ for $1 \le k < m_0^{i,-}$;
    \item $\alpha_{\beta}, \text{ for } \beta \in \mathcal{C}_{\neg X}$.
\end{itemize}
This linear dependence can be extended to include the vectors $R^{0,-}_j$ and $C^{i,-}(1,m_0^{i,-})$ by setting:
\begin{itemize}
    \item $\alpha_{j}^{0,-} = 0$;
    \item $C^{i,-}(1,m_0^{i,-}) = 0$,
\end{itemize}
for all $1 \le j \le N_0$ and $1 \le i \le n$, respectively. Then by equation \eqref{eq:lin-dep-equality1} from Theorem \ref{thm:lin-dep}, we must have that:
\begin{align*}
&\alpha^{i,+}(0,1) = \dots = \alpha^{i,+}(0,m_0^{i,+}) = -\alpha^{i,-}(1,1) = \dots = - \alpha^{i,-}(1,m_{1}^{i,-}) = 0; \\
&\alpha^{0,+}(j,1) = \dots = \alpha^{0,+}(j,m_j^{0,+}-1) = a_j^{0,+} \\
&=-\alpha^{0,-}(j+1,1) = \dots = - \alpha^{0,-}(j+1,m_{j+1}^{0,-}-1) = -\alpha_{j+1}^{0,-} = 0,
\end{align*}
for all $1 \le i \le n$ and $0 \le j \le N_0-1$, respectively. Thus by equation \eqref{eq:lin-dep-equality2} from Theorem \ref{thm:lin-dep}, we also have that:
\begin{align*}
& \alpha_0^0 = \alpha^{0,+}(0,1) = 0; \\
& \alpha_{N_0}^0 = \alpha^{0,-}(N_0,1) = 0; \\
& \alpha_j^0 = \alpha^{0,+}(j,1) + \alpha^{0,-}(j,1) = 0 + 0 = 0; \\
& \alpha_{\beta} = 0,
\end{align*}
for all $1 \le j \le N_0-1$ and for all $\beta \in \mathcal{C}_{\neg X}$, respectively. Thus all of the coefficients of the linear dependence are equal to zero, which is what we wanted to prove.
\end{proof}

\subsection{Refining partitions of $I$}
\label{subsec:refining-part}

For the purpose of the proof, we will simplify the notation used in the statement of Theorem \ref{thm:lin-dep}. Denote by $\mathcal{V}$ the set of all vectors from \hyperref[eq:vectors]{(*)} and by $\mathcal{A}$ the set of all coefficients from \hyperref[eq:a-coefficients]{(**)}. 

A critical connection vector $v \in \mathcal{V}$ corresponds to the orbit of a discontinuity of $T$ up to the time of landing on another discontinuity. For such $v$, let $P(v)$ be the ordered set of points along the orbit corresponding to $v$. The landing vectors correspond to the orbit of a discontinuity of the return map $R_{J}$ up to the time of landing onto discontinuity of $T$, critical connection vectors corresponding to the orbit of a discontinuity of $T$ up to the landing time to another discontinuity of $T$, and the return vectors correspond to the orbit of a discontinuity up to the time of landing into $J$. We adopt the convention that the first point in the orbit corresponding to the vector is included in $P(v)$, while the last one is not. For example, this means that the $P(v)$ associated to the critical connection vector $C^{0,+}(0,1)$ is equal to $\{ \beta^{0,+}(0,1), T(\beta^{0,+}(0,1)), \dots, T^{q^{0,+}(0,2) - 1}(\beta^{0,+}(0,1)) \}$. If a discontinuity of $T$ is already contained in $J_0$, then the set $P(v)$ for the associated landing vector is empty. We will exclude such vectors from the set $\mathcal{V}$. For a vector $v \in \mathcal{V}$, let us denote by $v_{first}$ the first element of $P(v)$ and by $v_{last}$ the last element of $P(v)$.

\begin{definition}
Let $\mathcal{V}$ be the set of all vectors as defined above. Then we define the set of all \textit{distinguished points} (with respect to $\mathcal{V}$) as the union of all points in $P(v)$:
\[
\mathcal{P} := \bigcup_{v \in \mathcal{V}} P(v).
\]
\end{definition}
Note that $\mathcal{P}$ is by definition a set that includes both signed (from critical connection and return vectors) and geometric points (from landing vectors) points. For a vector $v \in \mathcal{V}$, we will denote by $\alpha(v)$ its corresponding coefficient. The assumption \eqref{eq:lin-dep-sum} of Theorem \ref{thm:lin-dep} then reads as:

\begin{equation}
\label{eq:lin-dep}
\sum_{v \in \mathcal{V}} \alpha(v) v = 0.
\end{equation}
Let $v_s$ be the $s$-th $\bm{e}$-coefficient of $v$. Equation \eqref{eq:lin-dep} implies:

\[
\sum_{v \in \mathcal{V}} \alpha(v) v_s = 0, \text{ for all } 1 \le s \le r.
\]
Note that $v_s$ also equals the number of points from $P(v)$ contained in $I_s$. This means that we have the initial partition $\mathcal{I}^0$ of $I$ into intervals $I_s^0 := I_s$, for $1 \le s \le r$, such that a linear equation on the number of points from each $P(v)$ holds on each interval of this partition. This motivates the following two definitions:

\begin{definition}
\label{N}
Let $I^*$ be any interval in $I$. Then we define the number of points from $P(v)$ in $I^*$ as:
\[
N(I^*,v) := |P(v) \cap I^*|.
\]
\end{definition}

Next, we introduce intervals for which the number of points from each $P(v)$ satisfy a linear equation with respect to the coefficients in $\mathcal{A}$:

\begin{definition}[$\mathcal{A}$-subordinate interval]
Let $I^*$ be an interval in $I$. We say that $I^*$ is an \textit{$\mathcal{A}$-subordinate interval} if the following equation holds: 
\[
\sum_{v \in \mathcal{V}} \alpha(v) N(I^*,v) = 0.
\]
\end{definition}

Let $\mathcal{I}^* = \bigcup^k_{s=1} I_s^*$ be a partition (not necessarily dynamical) of $I$ into intervals $I_s^*$, i.e. $I = \bigsqcup_{s=1}^k I_s^*$. We say $\mathcal{I}$ is an \textit{$\mathcal{A}$-subordinate partition} if every interval $I_s^*$ is an $\mathcal{A}$-subordinate interval. By the discussion above, we know that $\mathcal{I}^0$ is an $\mathcal{A}$-subordinate partition. Our goal will be to refine the level-$0$ partition $\mathcal{I}^0$ into $\mathcal{A}$-subordinate partitions which consist of more intervals.

More precisely, we will define a sequence of partitions $\mathcal{I}^m = \bigcup^{r+m}_{s=1} I^m_s$ of $I$ into intervals $I_s^m$ so that:

\begin{enumerate}
    \item $\mathcal{I}^m$ is a refinement of $\mathcal{I}^{m-1}$;
    \item $\mathcal{I}^m$ is an $\mathcal{A}$-subordinate partition.
\end{enumerate}
This procedure will terminate at some maximal level $M$, described in Definition \ref{defn:max-part}. Our first goal is to find a criterion for when the partition $\mathcal{I}^m$ can be refined into $\mathcal{I}^{m+1}$ by splitting some interval $I_s^m$ of $\mathcal{I}^m$ into two intervals. We will show that this can be done by splitting along the gap interval between a certain pair of two neighbouring, but not touching, distinguished points. Recall that two signed points $x$ and $y$ touch if $\{ x,y \}$ = $\{z^+, z^-\}$ for some point $z$. The following lemma shows how to find such a pair:

\begin{lemma}[Distinguished pair]
\label{lem:dist-pair}
Let $\mathcal{I}^m$ be an $\mathcal{A}$-subordinate partition. Assume that there exists a triple $(I_s^m, p_1, p_2)$, where $I^m_s$ is an element of the partition, and $p_1,p_2$ is a pair of signed and distinct distinguished points in $I^m_s$, with $p_1$ to the left of $p_2$, such that they are:
\begin{enumerate}[label=(\alph*)]
    \item Neighbouring, but not touching, in $I_s^m$;
    \item There exist no $1 \le j \le N_0-1$ and $0 \le k < r_j^0$ such that $p_1$ and $p_2$ are contained in $T^k(J_j^0)$, and no $i \ge 1$ and $k \ge 0$ such that $p_1$ and $p_2$ are contained in $T^k(J_i)$.
\end{enumerate}
Then one can find a triple, still denoted by $(I^m_s, p_1, p_2)$, that additionally satisfies:

\begin{enumerate}[label=(\alph*), resume]
    \item The gap interval $G = [p_1,p_2]$ between the points has maximum length among all triples satisfying (a) and (b);
   
    \item Either at least one of the points $p_1$, $p_2$ has no preimage in $\mathcal{P}$ or no pair of preimages $q_1, q_2 \in \mathcal{P}$ of $p_1, p_2$ is contained in the same element of the partition $\mathcal{I}^m$.
\end{enumerate}
\end{lemma}

We call a pair of points from the conclusion of Lemma \ref{lem:dist-pair} a \textit{distinguished pair}.

\begin{proof}

By assumption, we can always choose $p_1$ and $p_2$ such that they satisfy (a) and (b), and by choosing the pair with maximal distance, we can assume they also satisfy (c). Assume that this pair does not satisfy assumption (d). We will show that such a pair can be `pulled back' to find a different pair of points that also satisfies (d).

By assumption, the pair of points $p_1, p_2$ has at least one pair of $T$-preimages $q_1, q_2 \in \mathcal{P}$ that are contained in a single interval of $\mathcal{I}^m$. Let us show that this pair $q_1, q_2$ also satisfies properties (a), (b) and (c). Indeed, if they are contained in a single element of the partition, then $T$ maps $[q_1,q_2]$ continuously to $G$, as $\mathcal{I}^m$ is a refinement of $\mathcal{I}^0$. Thus $q_1$ and $q_2$ are not touching and not contained in an iterate of the form $T^k(J_j^0)$ or $T^k(J_i)$, so (b) follows. Moreover, if they are neighbouring, i.e. if there is no point from $\mathcal{P}$ between them, then they also satisfy (a) and (c). 

Assume the contrary, that there is some point $q \in \mathcal{P} \cap [q_1, q_2]$ between points $q_1$ and $q_2$. Since this interval maps forward continuously, $T(q) \in G$ is not contained in $\mathcal{P}$. This is only possible if $q = v_{last}$, where $v$ is either a critical connection vector corresponding to landing on a point in $\mathcal{C}_2$, or $q = v_{R,last}$, where $v$ is a return vector. In the first case, $T(q)$ must be a discontinuity, which contradicts $p_1$ and $p_2$ being in the same partition element of $\mathcal{I}^m$. In the second case, we know that $T(q) \in J_0$, for some $1 \le i \le n$, so the interval $(p_1,p_2)$ must intersect $J_0$. We will show that this means that $p_1$ and $p_2$ are not neighbouring in $\mathcal{P}$. Indeed, if neither of them is contained in $J_0 = [x^{0,+},y^{0,-}]$, then the boundary points of $J_0$ are between them. If only $p_1$ is contained in $J_0$, then $y^{0,-}$ is between them, since it must be contained in $(p_1, p_2)$. If only $p_2$ is in $J_0$, then $x^{0,+}$ is between them. Finally, if they are both contained in $J_0$, then by (b) they are not contained in the same interval $J_j^0$, so they are not neighbouring in this case either. Thus the pair $q_1,q_2$ satisfies (c) as well. 

Thus if the points $p_1, p_2$ satisfy (a), (b) and (c), but not (d), we can always find their pullback to another pair of points $q_1,q_2$ that also satisfies (a), (b) and (c). If the pair $q_1,q_2$ does not satisfy (d), then we can pull back this pair as well. If this can be performed indefinitely, then some pair of points $p_1$, $p_2$ must appear twice. Then the interval $G = [p_1,p_2]$ must be periodic, as the gap interval between the pullback of points maps continuously over the gap interval between the points themselves. Moreover, no point in $[p_1,p_2]$ lands on a discontinuity. Thus $p_1$ and $p_2$ are contained in $X$ and have the same itineraries. This is only possible if they are both contained in a single iterate of the form $T^k(J_j^0)$ or $T^k(J_i)$, as points from different iterates of this kind must have different itineraries. This is a contradiction with (b), so the pullback cannot be performed indefinitely. Thus by pulling back, at some point we arrive at a pair of points that also satisfies (d).
\end{proof}

Let $\mathcal{P}^m_s := I^m_s \cap \mathcal{P}$ and $G^{\circ} := G \setminus \{p_1, p_2\}$. Define $\frac{p_1+p_2}{2}$ as the average of geometric points in $I$ corresponding to $p_1$ and $p_2$. Note that $\frac{p_1+p_2}{2}$ is the midpoint of $G$. Let $I^l$ be the interval in $I$ to the left of $\frac{p_1+p_2}{2}$ and let $I^r$ be the interval in $I$ to the right of $\frac{p_1+p_2}{2}$.

\begin{lemma}[Left or right]
\label{lem:l-or-r}
Let $\mathcal{I}^m$ be an $\mathcal{A}$-subordinate partition as in Lemma \ref{lem:dist-pair}, and let $G = [p_1, p_2]$ be a gap interval. Then for any element $I^m_s$ of this partition, we have that the set $T(\mathcal{P}_s^m)$ is either fully to the left or fully to the right of $\frac{p_1+p_2}{2}$. More precisely, we have that either $T(\mathcal{P}_s^m) \cap I^l = \emptyset$ or $T(\mathcal{P}_s^m) \cap I^r = \emptyset$.
\end{lemma}

\begin{proof}
Assume the contrary, that there exists an interval $I^m_s$ for which we have that $T(\mathcal{P}_s^m)$ intersects both $I^l$ and $I^r$. Let $q_1, q_2 \in \mathcal{P}_s^m$ be two neighbouring (not necessarily signed) points in $I^m_s$ such that their images are contained in $\mathcal{P}$, and $q_1$ gets mapped to $I^l$ and $q_2$ gets mapped to $I^r$. Assume first that the image $T(q_1)$ is contained in $G^{\circ}$, with the case for $T(q_2)$ being analogous. Then $T(q_1)$ is not a point in $\mathcal{P}$, since $p_1$ and $p_2$ are neighbouring. This is only possible if $q_1 = v_{last}$, where $v$ is either a critical connection vector corresponding to landing on a point in $\mathcal{C}_2$ or a return vector. The second case was already shown to be impossible in the proof of Lemma \ref{lem:dist-pair}, while the first case is impossible since $p_1$ and $p_2$ are by assumption contained in a single interval of the partition $\mathcal{I}^m$. Thus, neither of the points $T(q_1), T(q_2)$ is contained in $G^{\circ}$. As this pair gets mapped forward continuously, this means that the image of the gap between these intervals is at least as large as $G$. By property (d) of Lemma \ref{lem:dist-pair}, it is impossible that $q_1$ and $q_2$ get mapped over $p_1$ and $p_2$, respectively. If $q_1$ and $q_2$ are both signed, then this is a contradiction with maximality of $G$. We may without loss of generality assume that $q_1$ is not signed, with the other case being analogous. This means that $q_1 \in P(v)$ for a landing vector $v$, so it is contained in the left boundary of an iterate of the form $T^k(J^0_j)$. As there is no other points in $\mathcal{P}$ contained in the interior $T^k(J^0_j)$, this means that $p_1$ and $p_2$ must both be contained in $T^{k+1}(J^0_j)$, which is a contradiction.
\end{proof}

We now show that we can get an $\mathcal{A}$-subordinate partition of level $m+1$ if we split the interval $I^m_{s'}$ containing the pair $p_1,p_2$ along the gap interval $G$.

\begin{proposition}[Refining a partition]
\label{prop:refine} 
Let $\mathcal{I}^m$ be an $\mathcal{A}$-subordinate partition of level $m$ and assume that there exists at least one interval $I_{s'}^m$ of the partition as in Lemma \ref{lem:dist-pair}. Then the partition $\mathcal{I}^{m+1}$ obtained from $\mathcal{I}^m$ by splitting $I^m_{s'}$ along $\frac{p_1+p_2}{2}$ is an $\mathcal{A}$-subordinate partition of level $m+1$.
\end{proposition}

\begin{proof}
Let $I^{m,l}_{s'}$ be the maximal (half-open) interval to the left of $\frac{p_1+p_2}{2}$ inside $I^m_{s'}$, and let $I^{m,r}_{s'}$ be the maximal (half-open) interval to the right of $\frac{p_1+p_2}{2}$ inside $I^m_{s'}$. Then the intervals of the partition $\mathcal{I}^{m+1} = \bigcup_{1 \le s \le r+m+1} I^{m+1}_s$ are defined as:

\begin{equation*}
I^{m+1}_s := \begin{cases}
    I^m_s,& \text{For $s < s'$} \\[2pt]
    I^{m,l}_{s'},& \text{For $s = s'$} \\[2pt]
    I^{m,r}_{s'},& \text{For $s = s'+1$} \\[2pt]
    I^m_{s-1},& \text{For $s > s'+1$.}
\end{cases}
\end{equation*}
By construction, $\mathcal{I}^{m+1}$ is a refinement of $\mathcal{I}^0$. As only one of the intervals was modified, we know that:

\begin{equation}
\label{eq:m+1=m}
\begin{split}
\sum_{v \in \mathcal{V}} &\alpha(v) N(I^{m+1}_{s},v) = \sum_{v \in \mathcal{V}} \alpha(v) N(I^{m}_{s},v) = 0 \text{ for } s < s'; \\
\sum_{v \in \mathcal{V}} &\alpha(v) N(I^{m+1}_{s},v) = \sum_{v \in \mathcal{V}} \alpha(v) N(I^{m}_{s-1},v) = 0 \text{ for } s > s' + 1.
\end{split}
\end{equation}
Thus, in order to show that $\mathcal{I}^{m+1}$ is an $\mathcal{A}$-subordinate partition, we only need to prove:

\begin{equation}
\label{eq:l,r-eq}
\begin{split}
\sum_{v \in \mathcal{V}} &\alpha(v) N(I^{m+1}_{s'},v) = 0 \\
\sum_{v \in \mathcal{V}} &\alpha(v) N(I^{m+1}_{s'+1},v) = 0.
\end{split}
\end{equation}
Note that it is clearly enough to only prove one of the above equations holds, as the other follows from:

\begin{align*}
&\sum_{v \in \mathcal{V}} \alpha(v) N(I^{m+1}_{s'},v) + \sum_{v \in \mathcal{V}} \alpha(v) N(I^{m+1}_{s'+1},v) \\
= &\sum_{v \in \mathcal{V}} \alpha(v) N(I^{m}_{s'},v) = 0.
\end{align*}
Because of this, we may without loss of generality assume that the interval $J_0$ is contained in $I^r$, because we can otherwise just do the computation that follows for $I^r$ instead of $I^l$. This assumption will be important when we compute $\Delta_v$ (defined in \eqref{eq:n}) for landing and return vectors. Let $\mathcal{P}^l := \mathcal{P} \cap I^l$ be the set of all distinguished points to the left of the midpoint of $G$. Let Ind$^l$ be the set of all $s \in \{ 1, \dots, m+r \}$ such that $T(\mathcal{P}_s^m) \subset I^l$. Note that this definition makes sense because of Lemma \ref{lem:l-or-r}. Let $\mathcal{P}^{l,-1} := \bigcup_{s \in \text{Ind}^l} \mathcal{P}^m_s = T^{-1}(\mathcal{P}^l) \cap \mathcal{P}$ be the set of all distinguished points that $T$ maps into $\mathcal{P}^l$. Consider the following equation:

\begin{equation}
\label{eq:n}
\Delta_v := \sum_{s \in \text{Ind}^l} N(I^{m}_s, v) - \sum_{s \le s'} N(I^{m+1}_s, v).
\end{equation}
This equation states that the number of elements of $P(v)$ contained in $I^l$ equals the number of elements of $P(v)$ that get mapped under $T$ into $I^l$ up to an error term $\Delta_v$. From \eqref{eq:n}, we get:

\begin{align*}
&\sum_{v \in \mathcal{V}} \alpha(v) N(I^{m+1}_{s'},v) \\
= &\sum_{v \in \mathcal{V}} \alpha(v) \left( \sum_{s \in \text{Ind}_L} N(I^m_s, v) - \sum_{s < s'} N(I^{m+1}_s, v) - \Delta_v \right) \\
= &\sum_{s \in \text{Ind}_L} \left( \sum_{v \in \mathcal{V}} \alpha(v) N(I^m_s, v) \right) \\
&- \sum_{s < s'} \left( \sum_{v \in \mathcal{V}} \alpha(v) N(I^{m+1}_s, v) \right) - \sum_{v \in \mathcal{V}} \alpha(v) \Delta_v \\
= &-\sum_{v \in \mathcal{V}} \alpha(v) \Delta_v.
\end{align*}
where the sum in the third line is zero because of the inductive assumption that each interval $I^m_s$ is $\mathcal{A}$-subordinate, and the sum in the fourth line is zero because of equation \eqref{eq:m+1=m}. Thus we only need to prove that:

\begin{equation}
\label{eq:delta-v}
\sum_{v \in \mathcal{V}} \alpha(v) \Delta_v = 0.
\end{equation}
We now compute the error term $\Delta_v$ for all vectors $v \in \mathcal{V}$. For this purpose, we define the following sets:

\begin{align*}
P(v)^{l,-1} &:= P(v) \bigcap \left( \bigcup_{s \in \text{Ind}^l} I^{m}_s \right)\\
P(v)^l &:= P(v) \bigcap \left( \bigcup_{s \le s'} I^{m}_s \right).
\end{align*}
Note that the following holds:
\begin{align*}
|P(v)^{l,-1}| &= \sum_{s \in \text{Ind}^l} N(I^{m}_s, v) \\
|P(v)^l| &= \sum_{s \le s'} N(I^{m+1}_s, v).
\end{align*}
We first compute $\Delta_v$ for a critical connection vector $v \in \mathcal{V}$. Note that this includes the vectors in $\mathcal{C}_{\neg X}$. For any point $p$ in $P(v)^l$ that is not equal to $v_{first}$ (we do not assume it is contained in $P(v)^l$), we have that there exists exactly one point $q$ in $P(v)^{l,-1}$ such that $T(q) = p$. Indeed, this is because each element of $P(v)$ except the first one has a preimage in $P(v)$. Moreover, for each point $q$ in $P(v)^{l,-1}$ that is not equal to $v_{last}$, we have that $T(q) \in P(v)^l$. Thus we get the following formula for $\Delta_v$ for a critical connection vector $v \in \mathcal{V}$:

\begin{equation}
\label{eq:delta-v-crit}
\Delta_v = \begin{cases}
    0,& \text{If $v_{first} \in \mathcal{P}^l$ and $v_{last} \in \mathcal{P}^{l,-1}$} \\[2pt]
    -1,& \text{If $v_{first} \in \mathcal{P}^l$ and $v_{last} \notin \mathcal{P}^{l,-1}$} \\[2pt]
    1,& \text{If $v_{first} \notin \mathcal{P}^l$ and $v_{last} \in \mathcal{P}^{l,-1}$} \\[2pt]
    0,& \text{If $v_{first} \notin \mathcal{P}^l$ and $v_{last} \notin \mathcal{P}^{l,-1}$.}
\end{cases}
\end{equation}
Now we compute $\Delta_v$ for a return vector $v \in \mathcal{V}$. For every point $q$ in $P(v)^l$ that is not equal to $v_{first}$, we again have that there again exists exactly one point $q$ in $P(v)^{l,-1}$ such that $T(q) = p$. Moreover, for each point $q$ in $P(v)^{l,-1}$ that is not equal to $v_{last}$, we have that $T(q) \in P(v)^l$. Recall that we have assumed that $J_0$ is contained in $I^r$. This means that $T(v_{last}) \notin \mathcal{P}^l$, and so $v_{last} \notin \mathcal{P}^{l,-1}$. Thus $\Delta_v$ in this case only depends on whether $v_{first} \in \mathcal{P}^l$, so we get the following formula:

\begin{equation}
\label{eq:delta-v-return}
\Delta_v = \begin{cases}
    -1,& \text{If } v_{first} \in \mathcal{P}^l \\[2pt]
    0,& \text{If } v_{first} \notin \mathcal{P}^l.
\end{cases}
\end{equation}
Analogously, for a landing vector $v$, we know that $v_{first} \notin \mathcal{P}^l$, so the value of $\Delta_v$ depends only on whether $v_{last} \in \mathcal{P}^{l,-1}$:
\begin{equation}
\label{eq:delta-v-landing}
\Delta_v = \begin{cases}
    1,& \text{If $v_{last} \in \mathcal{P}^{l,-1}$} \\[2pt]
    0,& \text{If $v_{last} \notin \mathcal{P}^{l,-1}$}. \\[2pt]
\end{cases}
\end{equation}
Let $w_{\beta}$, for a (signed or geometric) discontinuity $\beta \in \mathcal{P}$, be the unique vector in $\mathcal{V}$ such that $w_{first} = \beta$, and let $w_{\beta} = 0, \alpha(w_{\beta}) = 0$ for all $\beta \notin \mathcal{P}$. We claim that the sum in \eqref{eq:delta-v} is equal to:

\begin{equation}
\label{eq:beta-L}
\sum_{\beta \in I^l} \left( - \alpha(w_{\beta^+}) - \alpha(w_{\beta^-}) + \sum_{\substack{w \in \mathcal{V}, \\ {T(w_{last}) \in \{ \beta, \beta^+, \beta^-\}}}} \alpha(w) \right).
\end{equation}
Here, we sum over all geometric discontinuities $\beta \in I^l$. Note for every $v \in \mathcal{V}$, the coefficient $\alpha(v)$ is equal to at most one $\alpha(w_{\beta^{\pm}})$ in \eqref{eq:beta-L}. Such a $\alpha(w_{\beta})$ exists if and only if $v_{first}$ is a signed discontinuity in $I^l$. Moreover, $\alpha(v)$ is also equal to at most one $\alpha(w)$ in the inner sum of \eqref{eq:beta-L}, and such $\alpha(w)$ exists if and only if $T(v_{last})$ is a (signed or geometric) discontinuity in $I^l$. We now prove that the following two claims hold:

\begin{enumerate}
    \item $\Delta_v = -1$ in \eqref{eq:delta-v} if and only if $\alpha(v)$ is equal to some $\alpha(w_{\beta})$ and not equal to any $\alpha(w)$ in \eqref{eq:beta-L};
    \item$\Delta_v = 1$ in \eqref{eq:delta-v} if and only if $\alpha(v)$ is not equal to any $\alpha(w_{\beta})$ and equal to some $\alpha(w)$ in \eqref{eq:beta-L}.
\end{enumerate} 
The equality of \eqref{eq:delta-v} and \eqref{eq:beta-L} clearly follows from these two claims, because then the contribution of $\alpha(v)$ to equations \eqref{eq:delta-v} and \eqref{eq:beta-L} is equal for all vectors $v \in \mathcal{V}$. We prove the first claim, while the proof of the second one is analogous.

We know that if $\Delta_v = -1$, then $v_{first} \in \mathcal{P}^l$ and $v_{last} \notin \mathcal{P}^{l,-1}$. In particular, $v$ is not a landing vector, by \eqref{eq:delta-v-return}. Thus $v_{first}$ is a signed discontinuity in $I^l$, while $T(v_{last})$ is not a discontinuity in $I^l$, so $\alpha(v)$ is equal to some $\alpha(w_{\beta})$ and not equal to any $\alpha(w)$. Conversely if $\alpha(v)$ is equal to some $\alpha(w_{\beta})$ and not equal any $\alpha(w)$ then the $v_{first}$ is a signed discontinuity in $I^l$, so $v_{first} \in \mathcal{P}^l$, and $T(v_{last})$ is not a discontinuity in $I^l$, so $v_{last} \notin \mathcal{P}^{l,-1}$. Thus $\Delta_v = -1$, so the claim follows.

We now show that the sum in \eqref{eq:beta-L} is equal to zero, which finishes the proof. Indeed, we know that the $\bm{f}$-coefficient at ind$(\beta)$ for some $\beta \in I^l$ for the resulting vector in \eqref{eq:lin-dep} has to be zero. The only vectors in this sum for which this $\bm{f}$-coefficient is non-zero are those vectors $v$ for which $v_{first} = \beta^{\pm}$ or those for which $T(v_{last}) \in \{ \beta, \beta^+, \beta^- \}$. Thus taking the $\bm{f}$-coefficient at ind$(\beta)$ in equation \eqref{eq:lin-dep}, we get that:

\[
\left( - \alpha(w_{\beta^+}) - \alpha(w_{\beta^-}) + \sum_{\substack{w \in \mathcal{V}, \\ {T(w_{last}) \in \{ \beta, \beta^+, \beta^- \}}}} \alpha(w) \right) = 0
\]
As we split along an interval $G^{\circ}$ that does not contain a discontinuity in its interior, we know that $\beta^+ \in I^l$ if and only if $\beta^- \in I^l$. Thus we must have that:
\[
\sum_{\beta \in I^l} \left( - \alpha(w_{\beta^+}) - \alpha(w_{\beta^-}) + \sum_{\substack{w \in \mathcal{V}, \\ {T(w_{last}) \in \{ \beta, \beta^+, \beta^-\}}}} \alpha(w) \right) = 0,
\]
which finishes the proof.
\end{proof}

Let us now show that Proposition \ref{prop:refine} can be applied only finitely many times.

\begin{lemma}
\label{lem:finite-refine}
We can apply Proposition \ref{prop:refine} finitely many times if we start with the initial partition $\mathcal{I}^0$ given by the continuity intervals.
\end{lemma}

\begin{proof}
We will show that applying Proposition \ref{prop:refine} always increases the number of elements of the partition that have a non-empty intersection with $\mathcal{P}$. Since $\mathcal{P}$ is a finite set, this clearly shows that Proposition \ref{prop:refine} can only be applied finitely many times. By construction, we can apply Proposition \ref{prop:refine} to a partition $\mathcal{I}^m$ if and only if we find a pair of distinguished points $p_1,p_2$ as in Lemma \ref{lem:dist-pair}. These points are contained in the same element $I^m_{s'}$ of the partition, so after splitting the partition along the gap interval $G$ between them, we get two intervals $I^{m+1}_{s'}$ and $I^{m+1}_{s'}$ that contains $p_1$ and $p_2$, respectively. Thus we have increased the number of elements of the partition that have a non-empty intersection with $\mathcal{P}$, so the proof follows.
\end{proof}
 
Thus this refinement procedure must stop at some maximal level:

\begin{definition}[Maximal partition]
\label{defn:max-part}
Let $M$ be the level at which the procedure of inductively applying Proposition \ref{prop:refine} stops. We call the partition $\mathcal{I}^M$ obtained by this procedure the \textit{maximal partition}.
\end{definition}
It is of maximal level in the sense that for the partition $\mathcal{I}^M$ of this level, there is no pair of points satisfying (a) and (b) from Lemma \ref{lem:dist-pair}. Recall that $\mathcal{C}_{\neg X}$ is the set of all discontinuities in $\mathcal{C}_1$ not contained in $X$. The maximality of $\mathcal{I}^M$ gives the following corollary:

\begin{corollary}
\label{cor:alpha_x=0}
All of the coefficient $\alpha(v_{\beta})$, for $\beta \in \mathcal{C}_{\neg X}$, are equal to zero.
\end{corollary}

\begin{proof}
We will prove this for a $+$-type discontinuity $\beta$, with the proof for a $-$-type discontinuity being analogous. Let $\mathcal{I}^M$ be the maximal partition as in Definition \ref{defn:max-part}. Assume that there is another point $p \in \mathcal{P}$ contained in the element of the partition $I^M_s$ containing $\beta$. As $\beta$ is a discontinuity, it must be the leftmost point $I^M_s$, as $\mathcal{I}^M$ is a refinement of $\mathcal{I}^0$. Thus we may assume that $\beta$ and $p$ are neighbouring, but not touching. As $\beta \notin X$, we know that this pair of points is not contained in a periodic interval. Thus this pair of points satisfies properties (a) and (b) from Lemma \ref{lem:dist-pair}, so the partition $\mathcal{I}^M$ can in fact be refined, which is a contradiction with its maximality. Thus there is no other distinguished point in $I^M_s$. As $\mathcal{I}^M$ is an $\mathcal{A}$-subordinate partition, we must have that $\alpha(v_{\beta}) = 0$, as this is exactly the equation which $I^M_s$ needs to satisfy to be $\mathcal{A}$-subordinate.
\end{proof}
As all of these coefficients are zero, they do not contribute to the equation \eqref{eq:lin-dep}. Thus we may exclude their corresponding vectors and points from $\mathcal{V}$ and $\mathcal{P}$, respectively. This means that all points in $\mathcal{P}$ are now contained in $X$. This immediately gives the following corollary, which proves a part of Theorem \ref{thm:lin-dep} for the landing vectors corresponding to points in the interior of $J_0$:

\begin{corollary}
We have that $\alpha_j^0 = \alpha^{0,+}(j,1) + \alpha^{0,-}(j,1)$ for all $1 \le j \le N_0-1$ in conclusion \eqref{eq:lin-dep-equality2} of Theorem \ref{thm:lin-dep}.
\end{corollary}

\begin{proof}
By Corollary \ref{cor:alpha_x=0}, we know that for all $1 \le j \le N_0-1$, there exist only three vectors $v \in \mathcal{V}$ such that the $\bm{f}$-coefficient at ind($\beta^0(j)$) is non-zero. These are exactly $L_j^0, C^{0,+}(j,1)$ and $C^{0,-}(j,1)$. As the $\bm{f}$-coefficients at ind($\beta^0(j)$) in equation \eqref{eq:lin-dep} have to cancel out, we get that:
\[
\alpha_j^0 = \alpha^{0,+}(j,1) + \alpha^{0,-}(j,1)
\]
for all $1 \le j \le N_0-1$.
\end{proof}
We may now assume that the points and vectors associated to the landing vectors $L_j^0$, for all $1 \le j \le N_0-1$, are not contained in $\mathcal{P}$ and $\mathcal{V}$, respectively. Thus $\mathcal{P}$ now contains only signed points. We can make this assumption because we will be able to deduce the remaining equalities from the conclusions \eqref{eq:lin-dep-equality1} and \eqref{eq:lin-dep-equality2} of Theorem \ref{thm:lin-dep} without using these landing vectors. 

The maximality of $\mathcal{I}^M$ also gives the following corollary:

\begin{corollary}
\label{cor:a-subord-int-in-X}
All maximal intervals contained in $X$ that do not contain discontinuities in their interiors are $\mathcal{A}$-subordinate intervals.
\end{corollary}
\begin{proof}
We claim that every element $I^M_s$ of the partition, $I^M_s \cap X$ contains at most one interval that contains points in $\mathcal{P}$. Indeed, if there were at least two such intervals in this intersection, then we could choose a pair of points from two such neighbouring intervals as the pair of points from Lemma \ref{lem:dist-pair}, which would mean we can further refine $\mathcal{I}^M$, and this is a contradiction with maximality of $\mathcal{I}^M_s$. Since $\mathcal{I}^M$ is a refinement of $\mathcal{I}^0$ and since in Proposition \ref{prop:refine} we always split along a gap interval $G$ not containing any points in $\mathcal{P}$ in its interior, a maximal interval $H$ contained in $X$ that does not contain discontinuities in its interior and has non-empty intersection with $\mathcal{P}$ is contained in a single element $I^M_s$ of the partition $\mathcal{I}^M$. By Corollary \ref{cor:alpha_x=0}, all of the points in $\mathcal{P} \cap I^M_s$ are contained in a single such interval $H$. Since $I^M_s$ is an $\mathcal{A}$-subordinate interval, this gives that $H$ is also an $\mathcal{A}$-subordinate interval. Since intervals contained in $X$ that do not contain points in $\mathcal{P}$ are trivially $\mathcal{A}$-subordinate, the proof follows.
\end{proof}

\subsection{Refining partitions into $\mathcal{A}$-subordinate intervals}

Our next goal is to partition the intervals of $X$ into smaller $\mathcal{A}$-subordinate intervals. These smaller intervals can then be further partitioned, and we will show that we may continue with this procedure until all intervals of the form $T^k(J_j^0)$, for $1 \le j \le N_0-1$ and $l_j^0 \le k < r_j^0$, and $T^k(J^i)$, for $1 \le i \le n$ and $k \ge 0$, are elements of this partition. In order to do this, we will show how a partition of an interval into smaller $\mathcal{A}$-subordinate intervals may be further refined by `pushing forward' the partition of the preimage of this interval.

Let $H = [a,b) \neq J_0$ be a component interval of $X$. Assume that $H$ satisfies the following two properties:
\begin{enumerate}
\label{H-properties}
\item There is a finite number of points $a = h_0 < h_1 < \dots < h_k < h_{k+1} = b$ such that the intervals $H_0 = [h_0, h_1), \dots, H_k = [h_k,h_{k+1})$ are $\mathcal{A}$-subordinate, and such the intervals $(h_q^+,h_{q+1}^-)$ contain no discontinuities for all $0 \le q \le k$; 
\item The preimage of $H$ in $X$ can be partitioned into $\mathcal{A}$-subordinate intervals $G_0 = [g_0,g_1), G_1=[g_2,g_3), \dots, G_l=[g_{2l},g_{2l+1})$, such that the intervals $(g_{2r}^+, g_{2r+1}^-)$ contain no discontinuities for all $0 \le r \le l$. We allow for $g_{2r+1}=g_{2r+2}$, i.e. the intervals could touch at their boundary points.     
\end{enumerate} 
We will refer to these intervals as $H$-intervals and $G$-intervals, respectively. We distinguish between the following two types of points in $H$:
\begin{itemize}
    \item $h_0^+,h_1^{\pm},\dots,h_{k+1}^-$ will be called \textit{Type-1 points};
    \item $T(g_0^+), T(g_1^-) \dots, T(g_{2l+1}^-)$ will be called \textit{Type-2 points}.
\end{itemize}
We will refer to any point that is either a Type-$1$ or a Type-$2$ point as \textit{special}. Note that a point can be both a Type-$1$ and a Type-$2$ point. For example, this is the case for boundary points points $a^+$ and $b^-$.

Let $H_0' = [h_{0}',h_{1}'), H_1'=[h_{1}',h_{2}'), \dots, H_m'=[h_{m}',h_{m+1}')$ be the maximal intervals in $H$, such that the intervals $(h_{t}'^+,h_{t+1}'^-)$ contain no special points for all $0 \le t \le m$. In other words, the boundary of each $H_t'$ consists of two neighboring, but not touching, special points. We will refer to these intervals as $H'$-intervals. These intervals form a partition of $H$, which is a refinement of the partition into $H_0, H_1, \dots, H_k$.

Recall that because of Corollary \ref{cor:alpha_x=0} all of the points in $\mathcal{P}$ are contained in $X$. Since $H \neq J_0$, every point $p \in \mathcal{P} \cap H$ not equal to some $\beta^{0,\pm}(j,1)$, for $1 \le j \le N_0-1$, has a unique preimage in $X \cap \mathcal{P}$. Moreover, $p$ is contained in the same $P(v)$ as its unique preimage in $X \cap \mathcal{P}$ if and only if $p$ is not a discontinuity of $T$. This fact will be repeatedly used in the proof below. It will also be convenient to introduce the following notation: for a $\beta \in X$ that is not equal to some $\beta^{0,\pm}(j,1)$, for $1 \le j \le N_0-1$, denote by $v^{-1}_{\beta}$ the unique vector in $\mathcal{V}$ such that $T(v_{last}) = \beta$. Note that this vector is now unique because of Corollary \ref{cor:alpha_x=0}. For a discontinuity $\beta^{0,\pm}(j,1)$, with $1 \le j \le N_0-1$, we define $v^{-1}_{\beta^{0,\pm}(j,1)} = 0$. Recall also that $v_{\beta}$ is the unique vector in $\mathcal{V}$ such that $v_{first} = \beta$.

\begin{proposition}[Push-forward of a partition]
\label{prop:push-fwd}
Let $H$ satisfy the two assumptions in \ref{H-properties}. Then the intervals $H_0', \dots, H_m'$ are also $\mathcal{A}$-subordinate intervals. Moreover, for every discontinuity contained in $H$ that is also a Type-$2$ point, we have that:
\[
-\alpha(v^{-1}_{\beta}) + \alpha(v_{\beta}) = 0.
\]
\end{proposition}

\begin{proof}
The proof goes by induction, starting with the leftmost interval $H_0'$ and going to the right. The induction has two hypotheses:
\begin{itemize}
    \item $H_z'$ is an $\mathcal{A}$-subordinate interval for all $z \le t$;
    \item For each discontinuity $\beta$ of $T$ contained in these intervals that is also a Type-$2$ point, we have that:
    \[
    -\alpha(v^{-1}_{\beta}) + \alpha(v_{\beta}) = 0.
    \]
\end{itemize}
We now prove the base of the induction for $t=0$. Note that by assumption, none of the $H$-intervals contain discontinuities in their interior. The proof depends on the type of the boundary point $h_1'$ of $H_0'$. There are $3$ different cases.

\underline{Case 1: $h_1'$ is only a Type-$1$ point}
\vspace{1mm}

Then $H_0' = H_0$, so the first hypothesis follows. Assume first that $a$ is not a discontinuity of $T$. In this case, even if $h_1'$ is a discontinuity of $T$, it is not of Type-$2$, so the second hypothesis follows. If $a^+ = \beta^+$, then since $H$ is assumed to be a component of $X$, we know that $\beta^- \notin X$. Thus the only vectors in $\mathcal{V}$ for which the $\bm{f}$-coefficient at ind($\beta^+$) is non-zero are $v_{\beta^+}$ and $v^{-1}_{\beta^+}$. Thus we have that:
\[
-\alpha(v^{-1}_{\beta^+}) + \alpha(v_{\beta^+}) = 0,
\]
so the second hypothesis follows in this case as well.

\underline{Case 2: $h_1'$ is only a Type-$2$ point}
\vspace{1mm}

Note that this means that $h_1'$ is not a discontinuity of $T$, as discontinuities are Type-$1$ points by the assumption that $H$-intervals do not contain discontinuities in their interiors. Moreover, we have that $H_0' = T(G_r)$, for some $0 \le r \le l$. Assume that $a$ is not a discontinuity of $T$. Then the second hypothesis follows, as there are no discontinuities in $H_0'$. Moreover, all of the points in $H'_0 \cap \mathcal{P}$ belong to the same $P(v)$ as their preimages in $G_r$ and thus $N(H_0', v) = N(G_r, v)$ for all $v \in \mathcal{V}$. Thus $H'_0$ is an $\mathcal{A}$-subordinate interval, because $G_r$ is. Same as in the previous case, if $a^+ = \beta^+$, then:
\[
\alpha(v^{-1}_{\beta^+}) - \alpha(v_{\beta^+}) = 0,
\]
so the second hypothesis follows.

\underline{Case 3: $h_1'$ is both a Type-$1$ and a Type-$2$ point}
\vspace{1mm}

This means that $H_0' = H_0 = T(G_r)$, for some $0 \le r \le l$, so the first hypothesis follows. Assume that $a$ is not a discontinuity of $T$. If the right boundary point $h_1'$ of $H'_0$ is not a discontinuity of $T$, the second hypothesis trivially follows. If $h_1'^- = \betas^-$ is a discontinuity of $T$, then we have that $N(H_0', v) = N(G_r, v)$ for all $v \in \mathcal{V}$, except for $v_{\betas}$ and $v^{-1}_{\betas}$. Thus we have that:
\[
0 = \sum_{v \in \mathcal{V}} \alpha(v) N(G_r, v) - \sum_{v \in \mathcal{V}} \alpha(v) N(H'_0, v) = \alpha(v^{-1}_{\betas^-}) - \alpha(v_{\betas^-}),
\]
so the second hypothesis follows. If we assume $a^+ = \beta^+$, then as in Case $1$, we get that:
\[
\alpha(v^{-1}_{\beta^+}) - \alpha(v_{\beta^+}) = 0.
\]
If $h_1'$ is not a discontinuity of $T$, the second hypothesis follows, so we may assume $h_1'^-=\betas^-$. Then we have that:
\[
0 = \sum_{v \in \mathcal{V}} \alpha(v) N(G_r, v) - \sum_{v \in \mathcal{V}} \alpha(v) N(H'_0, v) = \alpha(v^{-1}_{\beta^+}) - \alpha(v_{\beta^+}) + \alpha(v^{-1}_{\betas^-}) - \alpha(v_{\betas^-}),
\]
which implies:
\[
\alpha(v^{-1}_{\betas^-}) - \alpha(v_{\betas^-}) = 0,
\]
so the second hypothesis follows.

Assuming the two induction hypothesis, we now want to prove that $H_{t+1}' = [h_{t+1}',h_{t+2}')$ is also an $\mathcal{A}$-subordinate interval, for some $t \ge 0$. The proof now depends on the type of both $h_{t+1}'^+$ and $h_{t+2}'^-$, so there are $4$ cases.

\underline{Case 1: Both are only Type-$1$ points}
\vspace{1mm}

Then $H_{t+1}' = H_q$ for some $q$, so we are done, because in this case even if the boundary points are discontinuities, they are only of Type-$1$, so we need not check the second inductive hypothesis. 

\underline{Case 2: Both are only Type-$2$ points}
\vspace{1mm}

Then $H_{t+1}' = T(G_r)$ for some $r$, and neither of the boundary points is a discontinuity of $T$. Thus the first inductive hypothesis follows from the fact that $N(H_0', v) = N(G_r, v)$ for all $v \in \mathcal{V}$, and the second because there are no discontinuities in $H'_{t+1}$.

\vspace{1mm}

\underline{Case 3: $h_{t+1}'^+$ is a Type-$1$ point and $h_{t+2}'^-$ is a Type-$2$ point}

\vspace{1mm}

Then $H_{t+1}'$ is contained in the image of the interval $G_r$ for which $T(g_{2r+1}^-) = h_{t+2}'^-$. Let $H'_{u}, \dots, H'_{t}$ be all of the other $H'$-intervals contained in $T(G_r)$ and note that they form a partition of $T(G_r)$. Thus we have that:

\[
\sum_{z=u}^{t+1} N(H'_z, v) = N(G_r,v)
\]
for all $v$ except those for which either $v_{first} = \beta$ or $T(v_{last}) = \beta$, where $\beta$ is a discontinuity contained in $T(G_r)$. Thus we have that:

\begin{align}
\label{eq:h-t+1}
\begin{split}
\sum_{v \in \mathcal{V}} \alpha(v) N(H'_{t+1}, v) &= \sum_{v \in \mathcal{V}} \alpha(v) N(G_r, v) - \sum_{z=u}^t \left( \sum_{v \in \mathcal{V}} \alpha(v) N(H'_z, v) \right) \\ &+ \sum_{\beta \in T(G_r)} \left( -\alpha(v_{\beta}) + \alpha(v^{-1}_{\beta}) \right) \\
&= \sum_{\beta \in T(G_r)} \left( -\alpha(v_{\beta}) + \alpha(v^{-1}_{\beta}) \right).
\end{split}
\end{align}
Here we sum over all signed discontinuities $\beta \in T(G_r)$. We now need to consider all of the discontinuities contained in $T(G_r)$. The analysis is different for discontinuities on the boundary and those in the interior. For a discontinuity $\beta$ contained in the interior of $T(G_r)$, we know that:

\begin{equation}
\label{eq:H'-int-disc}
-\alpha(v_{\beta^+}) + \alpha(v^{-1}_{\beta^+}) -\alpha(v_{\beta^-}) + \alpha(v^{-1}_{\beta^-}) = 0,
\end{equation}
by the assumption of Theorem \ref{thm:lin-dep} and by Corollary \ref{cor:alpha_x=0}. If the left boundary point $h'^+_{t+1}$ of $T(G_r)$ is a discontinuity $\betas^+$, then we know that:
\begin{equation}
\label{eq:H'-bdry-disc}
-\alpha(v_{\betas^-}) + \alpha(v^{-1}_{\betas^-}) = 0,
\end{equation}
by the second inductive hypothesis. If the right endpoint $h'^-_{t+2}$ of $T(G_r)$ is a discontinuity $\betass^-$, then it is a Type-$1$ point as well, so we know that $H'_{t+1}$ is an $H$-interval by the first assumption about the $H$-partition, and thus an $\mathcal{A}$-subordinate interval. Putting equations \eqref{eq:h-t+1}, \eqref{eq:H'-int-disc} and \eqref{eq:H'-bdry-disc} together gives:

\begin{align*}
0 = &\sum_{v \in \mathcal{V}} \alpha(v) N(H'_{t+1}, v) \\
\overset{\eqref{eq:h-t+1}}{=} &\sum_{\beta \in T(G_r)} \left( -\alpha(v_{\beta}) + \alpha(v^{-1}_{\beta}) \right) \\
\overset{\eqref{eq:H'-int-disc}\&\eqref{eq:H'-bdry-disc}}{=} &-\alpha(v_{\betas^+}) + \alpha(v^{-1}_{\betas^+}), 
\end{align*}
This proves the second inductive hypothesis, so the inductive step is done in this case. Thus we may assume that $h_{t+2}'$ is not a discontinuity. Then by equations \eqref{eq:h-t+1}, \eqref{eq:H'-int-disc} and \eqref{eq:H'-bdry-disc} we have that:

\begin{align*}
0 \overset{\eqref{eq:H'-int-disc} \& \eqref{eq:H'-bdry-disc}}{=} &\sum_{\beta \in T(G_r)} \left( -\alpha(v_{\beta}) + \alpha(v^{-1}_{\beta}) \right) \\
\overset{\eqref{eq:h-t+1}}{=} &\sum_{v \in \mathcal{V}} \alpha(v) N(H'_{t+1}, v), 
\end{align*}
so $H'_{t+1}$ is an $\mathcal{A}$-subordinate interval in this case as well. Finally, if the left boundary point $h'^+_{t+1}$ of $T(G_r)$ is not a discontinuity, then by equations \eqref{eq:h-t+1} and \eqref{eq:H'-int-disc} we have that:
\begin{align*}
0 \overset{\eqref{eq:H'-int-disc}}{=} &\sum_{\beta \in T(G_r)} \left( -\alpha(v_{\beta}) + \alpha(v^{-1}_{\beta}) \right) \\
\overset{\eqref{eq:h-t+1}}{=} &\sum_{v \in \mathcal{V}} \alpha(v) N(H'_{t+1}, v),
\end{align*}
which finishes the proof.

\underline{Case 4: $h_{t+1}'$ is a Type-$2$ point and $h_{t+2}'$ is a Type-$1$ point}
\vspace{1mm}

Then similarly as in the previous case, we know that $H'_{t+1}$ is contained in a larger $H$-interval $H_q$, which is partitioned into $H'$-intervals $H'_u, \dots, H'_t$, which are $\mathcal{A}$-subordinate by the inductive hypothesis, and $H'_{t+1}$. Similarly as in the previous case, we have that:

\begin{align*}
\sum_{v \in \mathcal{V}} \alpha(v) N(H'_{t+1}, v) &= \sum_{v \in \mathcal{V}} \alpha(v) N(H_q, v) - \sum_{z=u}^t \left( \sum_{v \in \mathcal{V}} \alpha(v) N(H'_z, v) \right) \\
&+ \sum_{\beta \in H_q} \left( -\alpha(v_{\beta}) + \alpha(v^{-1}_{\beta}) \right) \\
&= \sum_{\beta \in H_q} \left( -\alpha(v_{\beta}) + \alpha(v^{-1}_{\beta}) \right).
\end{align*}
From here we simply need to consider whether the boundary points of $H_q$ are discontinuities and the proof follows analogously as in the previous case.
\end{proof}

It is interesting to note that the analogous statement for the `pull back' of a partition does not hold. A simple counterexample occurs when the interval $J_0$ does not contain a discontinuity, but its image $T(J_0)$ contains a single discontinuity $\beta$ in its interior. In this case $T(J_0)$ is partitioned into two $\mathcal{A}$-subordinate intervals $T(J_1)$ and $T(J_2)$, where $T(J_1)$ and $T(J_2)$ touch at $\beta$. Thus if we consider $H = T(J_0)$, the pull-back version of Proposition \ref{prop:push-fwd} would give us that $J_1$ and $J_2$ are $\mathcal{A}$-subordinate. This would mean that $\alpha_j^0 = \alpha^{0,+}(j,1)$ and $\alpha^{0,-}(j,1) = 0$, where $a_j$ is the discontinuity in $J-0$ that lands on $\beta$. By Theorem \ref{thm:lin-dep}, we know that this does not necessarily hold, so in general we can not `pull back' a partition.

\subsection{Proof of Theorem \ref{thm:lin-dep}}

In this subsection, we finish the proof of Theorem \ref{thm:lin-dep}. We do this by continuing with the process of refining partitions started in Subsection \ref{subsec:refining-part}. This will be done by refining the maximal partition $\mathcal{I}^M$ from Definition \ref{defn:max-part}. This refinement is obtained by further partitioning the maximal intervals in $X$ that do not contain discontinuities in their interiors into smaller intervals.

As we will now care only about the partitions of the component intervals of $X$, we introduce a new label for such partitions. We will say that a partition $\mathcal{I}_*$ of the components of $X$ into $\mathcal{A}$-subordinate intervals is an \textit{$\mathcal{A}X$-partition}. Such a partition is defined by a finite set of points that are contained in the interiors of intervals of $X$, which we call \textit{the break points} of $\mathcal{I}_*$. By Corollary \ref{cor:a-subord-int-in-X}, we know that the partition of $X$ into maximal intervals that do not contain discontinuities in their interiors is an $\mathcal{A}X$-partition. As we do not need to consider $\mathcal{I}^M$ as a partition of $I$ anymore, we will denote this $\mathcal{A}X$-partition by $\mathcal{I}^M$ as well. Thus the set of break points $\mathcal{B}_0$ of $\mathcal{I}^M$ is the union of the boundary points of all maximal intervals contained in $X$ that do not contain a discontinuity in their interiors. This set is clearly the union of all discontinuities of $T$ and all the boundary points of $X$. We will show how to inductively refine this partition by applying Proposition \ref{prop:push-fwd}.

More precisely, we will show by induction that for each $m \ge 0$, there is an $\mathcal{A}X$-partition $\mathcal{I}^{M+m}$ such that its set of break points $\mathcal{B}_m$ is equal to the union of the following two types of points:

\begin{enumerate}
    \item Points $z \in X$ such that there exists a $q \le m$ and a discontinuity $\beta \in X$ such that $z = T^q(\beta)$ and $T^r(\beta) \notin J_0$ for all $0 < r \le q$;
    \item Points $z \in X$ such that there exists a $q \le m$ and a boundary point $x$ of $X$ such that $z = T^q(x)$, and $T^r(x) \notin J_0$ for all $0 < r \le q$.
\end{enumerate}
These two properties clearly imply that $\mathcal{B}_{m+1} \supset \mathcal{B}_m$, so $\mathcal{I}^{M+m+1}$ is a refinement of $\mathcal{I}^{M+m}$.

\begin{proof}[Proof of Theorem \ref{thm:lin-dep}]
The base of induction is given by Lemma \ref{lem:dist-pair} and Proposition \ref{prop:refine}. Assuming we have a level $m$ partition $\mathcal{I}^{M+m}$ such that its set of break points $\mathcal{B}_m$ is equal to the union of the two types of points described above, we show how to refine it into a level $m+1$ partition $\mathcal{I}^{M+m+1}$. We describe this refinement in terms of the new break points that get added to $\mathcal{B}_m$ by applying Proposition \ref{prop:push-fwd}.

Let $H$ be any component of $X$ not equal to $J_0$. By the inductive assumption, $\mathcal{I}^{M+m}$ induces a partition of $H$ into $\mathcal{A}$-subordinate intervals. We may take these $\mathcal{A}$-subordinate intervals to be the $H$-intervals from Proposition \ref{prop:push-fwd}. As $H$ is a component interval of $X$, we know that its $T$-preimage in $X$ consists of maximal intervals that do not contain discontinuities in their interiors, which are $\mathcal{A}$-subordinate by Corollary \ref{cor:a-subord-int-in-X}. As $\mathcal{I}^{M+m}$ is a refinement of $\mathcal{I}^M$, these preimage intervals are further partitioned into $\mathcal{A}$-subordinate intervals, and we may take these intervals to be the $G$-intervals from Proposition \ref{prop:push-fwd}. Thus by Proposition \ref{prop:push-fwd} we know that the partition into $H$-intervals can be refined into a partition into $\mathcal{A}$-subordinate $H'$-intervals. The set of boundary points of these $H'$-intervals is equal to the union of the images of break points in $\mathcal{B}_m$ contained in the $G$-intervals and the break points in $\mathcal{B}_m \cap H$. This holds for any $H \neq J_0$, and thus for each break point $p \in \mathcal{B}_m$, we get that $T(p)$ is a break point of $\mathcal{B}_{m+1}$, provided $T(p) \notin J_0$. Thus every point of the first and second type described above is contained in $\mathcal{B}_{m+1}$. Thus the partition $\mathcal{I}^{M+m+1}$ obtained by applying Proposition \ref{prop:push-fwd} has the desired properties, which completes the induction step.

Thus for every sufficiently large $m$ every point of either of these two types is a break point of $\mathcal{B}_{m}$:

\begin{enumerate}
    \item Points $z \in X$ such that there exists a $q$ and a discontinuity $\beta \in X$ such that $z = T^q(\beta)$ and $T^r(\beta) \notin J_0$ for all $0 < r \le q$;
    \item Points $z \in X$ such that there exists a $q$ and a boundary point $x$ of $X$ such that $z = T^q(x)$ and $T^r(x) \notin J_0$ for all $0 < r \le q$;
\end{enumerate}
Let $\beta^{0,+}(j,k)$, for $0 \le j \le N_0-1$ and $k < m^{0,+}_j$, be a discontinuity in the orbit of $J_0$, and let $P(v)$ be the orbit associate to the vector $v$ such that $v_{first} = \beta^{0,+}(j,k)$. Note that because of $k < m^{0,+}_j$, we know that $T(v_{last}) \in \mathcal{P}$. For $m$ large enough, we therefore know that $v_{last}$ is a break point of $\mathcal{B}_m$. Thus if we apply Proposition \ref{prop:push-fwd} to the component interval $H$ of $X$ containing $T(v_{last})$, we get that:
\begin{equation}
\label{eq:alpha-eqcc}
\alpha^{0,+}(j,k) = \alpha^{0,+}(j,k+1),    
\end{equation}
if $k+1 < m^{0,+}_j$, or
\begin{equation}
\label{eq:alpha-eqcr}
\alpha^{0,+}(j,m^{0,+}_j-1) = \alpha^{0,+}(j,m^{0,+}_j),
\end{equation}
if $k+1 = m^{0,+}_j$. This is because $T(v_{last}) = \beta^{0,+}(j,k+1)$ is a Type-$2$ point. It is clear that an analogous argument applies to discontinuities of the form $\beta^{0,-}(j,k)$, for $1 \le j \le N_0$ and $k < m^{0,+}_j$, $\beta^{i,+}(0,k)$, for $1 \le i \le n$ and $k \le m^{i,+}_{0}$, and $\beta^{i,-}(1,k)$, for $1 \le i \le n$ and $k \le m^{i,-}_{1}$. Thus combining equations \eqref{eq:alpha-eqcc} and \eqref{eq:alpha-eqcr} for all of these discontinuities gives:
\begin{align}
\label{eq:alpha-eq-ccr}
\begin{split}
&\alpha^{i,+}(0,1) = \dots = \alpha^{i,+}(0,m_0^{i,+}) \\
&\alpha^{i,-}(1,1) = \dots = \alpha^{i,-}(1,m_{1}^{i,-}); \\
&\alpha^{0,+}(j,1) = \dots = \alpha^{0,+}(j,m_j^{0,+}-1) = a_j^{0,+}; \\
&\alpha^{0,-}(j+1,1) = \dots = \alpha^{0,-}(j+1,m_{j+1}^{0,-}-1) = \alpha_{j+1}^{0,-}; \\
&\alpha_0^0 = \alpha^{0,+}(0,1); \\
&\alpha_{N_0}^0 = \alpha^{0,-}(N_0,1), \\
\end{split}
\end{align}
for all $1 \le i \le n$ and $0 \le j \le N_0-1$, respectively.

Next, for every $0 \le j \le N_0$, we know that both $a^{0,+}_j$ and $a^{0,-}_{j+1}$ have to land on discontinuities before the interval $[a^{0,+}_j,a^{0,-}_{j+1}$ return to $J_0$. Thus there is a time $q$, such that the points $T^q(a^{0,+}_j)$ and $T^q(a^{0,-}_{j+1})$ are both break points of $\mathcal{B}_m$, for a sufficiently large $m$. Thus the interval $[T^q(a^{0,+}_j),T^q(a^{0,-}_{j+1}))$ is $\mathcal{A}$-subordinate, so we get that
\begin{equation}
\label{eq:alpha-eq+-}
\alpha(v) = -\alpha(w),
\end{equation}
where $v$ and $w$ are the vectors such that $T^q(a^{0,+}_j) \in P(v)$ and $T^q(a^{0,-}_{j+1}) \in P(w)$, respectively. Moreover, $v$ is either a critical connection vector or a return vector related to the orbit of $a^{0,+}_j$, and $w$ is either a critical connection vector or a return vector related to the orbit of $a^{0,-}_{j+1}$. Again, an analogous argument applies to discontinuities of the form $\beta^{0,-}(j,k)$, for $1 \le j \le N_0$ and $k < m^{0,+}_j$, $\beta^{i,+}(0,k)$, for $1 \le i \le n$ and $k \le m^{i,+}_{0}$, and $\beta^{i,-}(1,k)$, for $1 \le i \le n$ and $k \le m^{i,-}_{1}$. Thus by combining \eqref{eq:alpha-eq-ccr} and \eqref{eq:alpha-eq+-} for all of these discontinuities we get both of the conclusions \eqref{eq:lin-dep-equality1} and \eqref{eq:lin-dep-equality2} from the statement of Theorem \ref{thm:lin-dep}.
\end{proof}

\section{Stability is equivalent to ACC and Matching}
\label{sec:theoremB} 

\subsection{Definitions and examples} 

Let us recall the definition of a stable $\ITM$ from the introduction:

\begin{definition}[Stable maps]
\label{defn:stable} 
We say that a map $T$ is \textit{stable} if there a neighbourhood $\mathcal{U}$ of $T$ in $\ITM(r)$ such that:

\begin{enumerate}
\item The mapping $\overline{X}$ assigning to each $\ITM$ in $\mathcal{U}$ its non-wandering set is continuous with respect to the Hausdorff topology. Moreover, for each $\tilde{T} \in \mathcal{U}$, $\overline{X}(\Tilde{T})$ is homeomorphic to $\overline{X}(T)$.
\item The number of discontinuities in $I \setminus \overline{X}(\Tilde{T})$ is constant in $\mathcal{U}$.
\end{enumerate}
\end{definition} 

A corollary of Theorem \ref{thm:ep-dense-param} is that stable maps are of finite type:

\begin{corollary}
\label{cor:stable-fin-type}
Stable maps are of finite type.
\end{corollary}

\begin{proof}
By Theorem \ref{thm:ep-dense-param} we know that eventually periodic maps are dense in $\ITM(r)$. Thus they are also dense in the neighbourhood $\mathcal{U}$ of a stable map $T$ on which the definition of stability \ref{defn:stable} holds. This means that there is at least one map in $\mathcal{U}$ for which the non-wandering set is equal to a finite union of intervals since eventually periodic maps are of finite type by Corollary \ref{cor:ep-fin-type}. Thus an infinite map can not be contained in $\mathcal{U}$, since the non-wandering set of an infinite type map must contain a Cantor set by Lemma \ref{lem:nonwandering}, which is not homeomorphic to a finite union of intervals. 
\end{proof}

Note that this definition is softer than the usual notions of stability in 1-dimensional dynamics, like $\Omega$-stability (see \cite{MR1239171}) and $J$-stability (see \cite{MR1312365}). The difference is that we do not require any dynamical conjugacy between the maps in the neighbourhood $\mathcal{U}$, just that several properties remain the same: finite type, number of discontinuities and the shape of the non-wandering set.

Recall that we use the shorthand $\Tilde{X}$ for the non-wandering set of a perturbation $\Tilde{T}$ of $T$. The next two lemmas follow immediately from the fact that the non-wandering set $\Tilde{X}$ of a sufficiently small perturbation $\Tilde{T}$ of $T$ is close to and homeomorphic to the non-wandering set $X$ of $T$.

\begin{lemma}
\label{lem:conseq-stab}
For a sufficiently small perturbation $\Tilde{T}$ of a stable map $T$, we have that $\overline{X}(\Tilde{T})$ consists of the same number of intervals as $X$ and these intervals are pairwise close to each other in sense of Hausdorff. This means that each component $J$ of $X$ has a well-defined continuation $\Tilde{J} \subset \Tilde{X}$, and in particular that its boundary points $x$ and $y$ also have continuations $\Tilde{x}$ and $\Tilde{y}$. Moreover, the continuation $\Tilde{J}$ of an interval $J$ contains the continuations of the discontinuities contained in $J$. \qed
\end{lemma}

\begin{lemma}
\label{lem:loc-disc}
Each sufficiently small perturbation $\Tilde{T}$ of a stable map $T$ has the same number of discontinuities in each of the following sets:
\begin{enumerate}
    \item Interior of the non-wandering set: $\text{int}(X)$.
    \item Interior of the complement of the non-wandering set: $I \setminus \overline{X}$.
    \item Boundary of the non-wandering set: $\partial X$.
\end{enumerate}
Moreover, each continuation $\Tilde{J}$ of an interval $J$ of $X$ contains the continuations of discontinuities in $J$. \qed
\end{lemma}

Let us now give a non-trivial example of a stable map, shown in figure \ref{fig:stab-map} below.

\begin{example}[Stable map]
\label{ex:stable-map} 
Assume that $r=3$ and that the interval $J = [T^2(\beta_2^+),T(\beta_2^-)]$ 
contains $\beta_2$ in its interior. Let $J^-=[T^2(\beta_2^+),\beta_2^-]$ and $J^+=[\beta_2^+,T(\beta_2^-)]$.
Assume the following about the orbits of $\beta_2^+$ and $\beta_2^-$:

\begin{itemize}
    \item $T(\beta_2^+)\in I_1$ and $T^2(\beta_2^+) \in I_2$;
    \item $T(\beta_2^-)\in I_3$ and $T^2(\beta_2^-)\in I_1$;
\end{itemize} 
Then:

\begin{itemize}
    \item $T^3(\beta_2^+) = \beta_2^+  + \gamma_3 + \gamma_1 + \gamma_2$;
    \item $T^3(\beta_2^-) = \beta_2^-  + \gamma_2 + \gamma_3 + \gamma_1$,
\end{itemize}
so $T^3(\beta_2^+)\sim T^3(\beta_2^-)$. If we additionally assume that $T^3(\beta_2^+), T^3(\beta_2^-) \in J$, then, for a suitable choice of parameters as in the figure below, $T$ is stable and we have that $X(T)$ is equal to the orbit of $J$. Indeed, the itineraries of $\beta_2^+$ and $\beta_2^-$ remain the same for all nearby maps $\Tilde{T}$, which means that the interval $\Tilde{J} = [\Tilde{T}^2(\Tilde{\beta}_2^+),\Tilde{T}(\Tilde{\beta}_2^-))$ is a continuation of $J$. 
\end{example} 

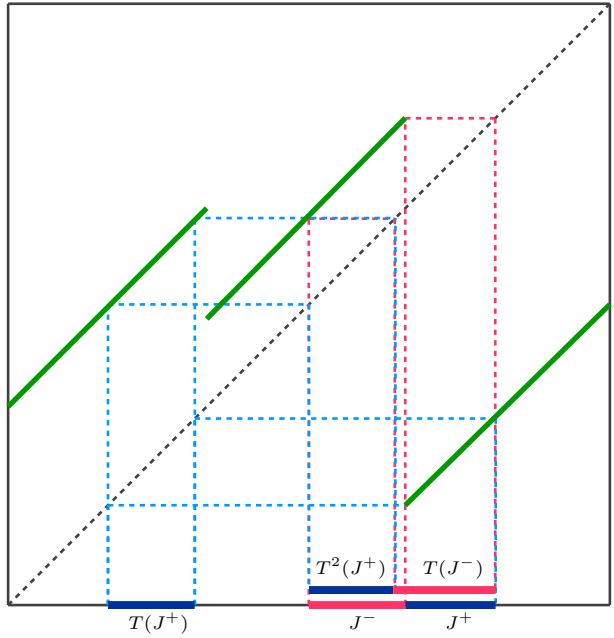
\begin{figure}[h]
\label{fig:stab-map}
\begin{center}
\definecolor{qqzzff}{rgb}{0.,0.6,1.}
\definecolor{qqttzz}{rgb}{0.,0.2,0.6}
\definecolor{ffttww}{rgb}{1.,0.2,0.4}
\definecolor{uququq}{rgb}{0.25,0.25,0.25}
\definecolor{qqzzqq}{rgb}{0.,0.6,0.}
\begin{tikzpicture}[line join=round,x=8.0cm,y=8.0cm]
\clip(-0.1,-0.1) rectangle (1.1,1.1);
\draw [line width=1.0pt,color=uququq] (0.,0.)-- (1.,0.);
\draw [line width=1.0pt,color=uququq] (0.,1.)-- (0.,0.);
\draw [line width=1.0pt,dash pattern=on 2pt off 2pt,color=uququq] (0.,0.)-- (1.,1.);
\draw [line width=1.0pt,color=uququq] (1.,0.)-- (1.,1.);
\draw [line width=1.0pt,color=uququq] (1.,1.)-- (0.,1.);
\draw [line width=1.0pt,dash pattern=on 2pt off 2pt,color=ffttww] (0.4996686055531796,0.)-- (0.49966860555317955,0.6425257484103225);
\draw [line width=1.0pt,dash pattern=on 2pt off 2pt,color=ffttww] (0.49966860555317955,0.6425257484103225)-- (0.6425257484103225,0.6425257484103225);
\draw [line width=1.0pt,dash pattern=on 2pt off 2pt,color=ffttww] (0.66,0.)-- (0.66,0.8095238095238095);
\draw [line width=1.0pt,dash pattern=on 2pt off 2pt,color=ffttww] (0.66,0.8095238095238095)-- (0.8095238095238095,0.8095238095238095);
\draw [line width=1.0pt,dash pattern=on 2pt off 2pt,color=ffttww] (0.8095238095238095,0.8095238095238095)-- (0.8095238095238095,0.024594454816063052);
\draw [line width=1.0pt,dash pattern=on 2pt off 2pt,color=qqzzff] (0.166,0.)-- (0.166,0.166);
\draw [line width=1.0pt,dash pattern=on 2pt off 2pt,color=qqzzff] (0.166,0.166)-- (0.66,0.166);
\draw [line width=1.0pt,dash pattern=on 2pt off 2pt,color=qqzzff] (0.3102479928758915,0.)-- (0.3102479928758915,0.3102479928758915);
\draw [line width=1.0pt,dash pattern=on 2pt off 2pt,color=qqzzff] (0.3102479928758915,0.3102479928758915)-- (0.8102479928758916,0.3102479928758915);
\draw [line width=1.0pt,dash pattern=on 2pt off 2pt,color=qqzzff] (0.8102479928758916,0.3102479928758915)-- (0.8102479928758916,0.);
\draw [line width=1.0pt,dash pattern=on 2pt off 2pt,color=ffttww] (0.6425257484103225,0.6425257484103225)-- (0.6425257484103224,0.024594454816063052);
\draw [line width=1.0pt,dash pattern=on 2pt off 2pt,color=qqzzff] (0.166,0.)-- (0.166,0.5);
\draw [line width=1.0pt,dash pattern=on 2pt off 2pt,color=qqzzff] (0.3102479928758915,0.)-- (0.3102479928758915,0.6435813262092248);
\draw [line width=1.0pt,dash pattern=on 2pt off 2pt,color=qqzzff] (0.3102479928758915,0.6435813262092248)-- (0.6435813262092248,0.6435813262092248);
\draw [line width=1.0pt,dash pattern=on 2pt off 2pt,color=qqzzff] (0.6435813262092248,0.024594454816063052)-- (0.6435813262092248,0.6435813262092248);
\draw [line width=1.0pt,dash pattern=on 2pt off 2pt,color=qqzzff] (0.166,0.5)-- (0.5,0.5);
\draw [line width=1.0pt,dash pattern=on 2pt off 2pt,color=qqzzff] (0.5,0.5)-- (0.5,0.024594454816063052);
\draw [line cap=square,line width=3pt,color=qqttzz] (0.66,0.)-- (0.81,0.);
\draw [line cap=square,line width=3pt,color=qqttzz] (0.5,0.025)-- (0.64,0.025);
\draw [line cap=square,line width=3pt,color=qqttzz] (0.166,0.)-- (0.31,0.);
\draw [line cap=square,line width=3pt,color=ffttww] (0.64,0.025)-- (0.81,0.025);
\draw [line cap=square,line width=3pt,color=ffttww] (0.5,0.)-- (0.66,0.);
{\scriptsize
\draw (0.59,0.0) node[anchor=north] {$J^-$};
\draw (0.75,0.0) node[anchor=north] {$J^+$};
\draw (0.25,0.0) node[anchor=north] {$T(J^+)$};
\draw (0.57,0.09) node[anchor=north] {$T^2(J^+)$};
\draw (0.74,0.09) node[anchor=north] {$T(J^-)$};
}
\draw [line width=2pt,color=qqzzqq] (0.66,0.166)-- (1.,0.5);
\draw [line width=2pt,color=qqzzqq] (0.,0.33)-- (0.33,0.66);
\draw [line width=2pt,color=qqzzqq] (0.33,0.476)-- (0.66,0.81);
\end{tikzpicture}
\caption{{A particular map as in Example ~\ref{ex:stable-map} with $\beta_0=0$, $\beta_1=1/3$, $\beta_2=2/3$, $\beta_3=1$, $\gamma_1=1/3$, $\gamma_2=1/7$, $\gamma_3=-1/2$. The intervals $J^-$ and $T(J^-)$ are shown in pink and the intervals $J^+$, $T(J^+)$, $T^2(J^+)$ are drawn in blue.}}
\label{fig3}
\end{center}
\end{figure}

The definition of stability mainly concerns the topology and the location of the non-wandering set. Thus the dynamical meaning of stability is unclear from the definition, and the dynamical information retained by making a small perturbation of a stable map is relatively soft. We would therefore like to characterize stability in terms of more concrete dynamical properties. Intuitively speaking, three things can happen for an arbitrarily small perturbation $\Tilde{T}$ of a map $T$ that can break stability:

\begin{enumerate}
    \item A definite part of $X$ gets removed, i.e. $\tilde{X}$ is smaller (either by size or by the number of intervals) than $X$ by a definite amount, so the mapping $\overline{X}$ is not lower semi-continuous.
    \item A definite new part gets added to $X$, i.e. $\Tilde{X}$ is bigger (either by size or by the number of intervals) than $X$ by a definite amount, so the mapping $\overline{X}$ is not upper semi-continuous.
    \item Discontinuities that were in $X$ are not in $\Tilde{X}$. 
\end{enumerate}
We are therefore looking for dynamical properties that guarantee that none of these three things can happen for an arbitrarily small perturbation of $T$.
To guarantee that the first thing does not happen, we can ask for the dynamics of the full orbit of every interval $J$ of $X$ to remain the same after perturbation (as in the example above). For the third, we can simply require that there are no discontinuities in the boundary of $X$ (assuming the first and the second things do not happen). The second one is the most elusive, and we will need to discuss it a bit before giving the definitions of the required dynamical properties.

The following example illustrates the simplest way in which the set $X$ can become larger.

\begin{example}
\label{ex:ghost-preimage}
Assume that there are two discontinuities $\betas$ and $\betass$ for which there exist $k_1 > 0$ and $k_2 > 0$ such that:
\[
T^{k_1}(\betas^+) = \betass^+ \text{ and } T^{k_2}(\betass^-) = \betas^-.
\]
Moreover, assume that none of the discontinuities $\betas^+, \betas^-, \betass^+$ and $\betass^-$ are in $X$ and that there exists a perturbation $\Tilde{T}$ of $T$ such that:
\begin{itemize}
    \item The itinerary of $\betas^+$ remains the same up to time $k_1$ and at time $k_1$ it lands $\epsilon$-distance to the left of $\betass^-$;
    \item The itinerary of $\betass^-$ remains the same up to time $k_2$ and at time $k_2$ it lands $\epsilon$-distance to the right of $\betas^+$.
\end{itemize}
Note that there exists an interval of points to the right of $\betas^+$ all of which have the same itinerary as $\betas^+$ up to time $k_1$. The same holds for $\betass^-$ and its itinerary up to time $k_2$. In particular, these intervals of definite size map forward continuously up to times $k_1$ and $k_2$, respectively. By our assumption on the perturbation $\Tilde{T}$, for a sufficiently small $\epsilon$ the interval $[\betas^+,\betas^+ + \epsilon)$ maps forward continuously up to time $k_1 + k_2$. Thus it is periodic with period $k_1 + k_2$ and is thus contained in $\Tilde{X}$. As $\betas^+, \betas^-, \betass^+$ and $\betass^-$ were by assumption not in $X$ and therefore a definite distance away from it, we see that $\overline{X}(\Tilde{T})$ and $\overline{X}(T)$ cannot be close. Thus $X$ can not be a stable map.
\end{example}

This example is the simplest case of a more general phenomenon called \textit{ghost preimages}. We give the definition for a $+$-type discontinuity, with the one for $-$-type being analogous.

\begin{definition}[Ghost preimage]
Let $\betas^+$ a discontinuity of $T$. A discontinuity $\betass^-$ that lands on $\betas^-$ is called a \textit{ghost preimage} of $\betas^+$.
\end{definition}

Thus a ghost preimage of a $+$-type discontinuity $\betas^+$ is a $-$-type discontinuity $\betass^-$ that lands on $\betas^-$. The definition can easily be extended to a $-$-type discontinuity. The motivation for the name is the following: $\betass^-$ is \textit{almost} a preimage of $\betas^+$, i.e. by an arbitrarily small perturbation we can make the iterate $T^k(\betass^-)$ land to the right of $\betas^+$, thus creating an actual preimage of $\betas^+$.


In Example \ref{ex:ghost-preimage} we have that $\betas^+$ is a ghost preimage of $\betass^-$ and that $\betass^-$ is a ghost preimage of $\betas^+$ as well. This is the property that allows for the enlargement of $X$, and it motivates the following definition:

\begin{definition}[Ghost tree]
\label{def:ghost-tree}
Let $\beta$ be a discontinuity of $T$. The \textit{ghost tree} $\mathcal{GT}(\beta)$ of $\beta$ is defined inductively in the following way. Set $\beta$ to be the root of the tree, i.e. the set of level $0$ vertices of the tree. Assume that we have defined all of the vertices of level $\le n$ and all of the edges between them. Then the level $n+1$ vertices of the tree correspond to the set of all ghost preimages (if such exist) of the discontinuities corresponding to level $n$ vertices of the tree. The new edges are those between discontinuities and their ghost preimages, i.e. a new direct edge $\betas \to \betass$ is added if and only if $\betas$ is a level $n+1$ vertex, $\betass$ is a level $n$ vertex and $\betas$ is a ghost preimage of $\betass$. 
\end{definition}

Thus $\mathcal{GT}(\beta)$ is a directed tree, with well-defined levels, and edges exist only between vertices of consecutive levels. If $\betas \to \betass$, then we say that $\betass$ is a \textit{parent} of $\betas$ and that $\betas$ is a \textit{child} of $\betass$. The types of discontinuities alternate between consecutive levels and the tree can have finitely or infinitely many levels. The following simple lemma gives a criterion for when a ghost tree is infinite:

\begin{lemma}[Infinite ghost trees]
\label{lem:inf-gt}
The ghost tree $\mathcal{GT}(\beta)$ of some discontinuity $\beta$ is infinite if $\beta$ appears as a vertex in $\mathcal{GT}(\beta)$ at some level $n > 0$.
\end{lemma}

Note that this is exactly the case in Example \ref{ex:ghost-preimage}.

\begin{proof}
If $\beta$ appears again as a vertex on some level $n > 0$, then all of the levels that appeared before the second appearance of $\beta$ have to repeat. There is at least one level between these appearances, as the discontinuity type has to change from level to level. Thus the tree must be infinite.
\end{proof}

We are now ready to state the two dynamical properties that characterize stability. All of our statements work for all components $J$ of $X$, except in the special case when no point in the interior of $J$ lands on a discontinuity and the return map $R_J$ to this component is the identity. We will call such components of $X$ \textit{dynamically trivial}, and components not of this form \textit{dynamically non-trivial}.

\begin{definition}[Absence of Critical Connections (ACC)]
\label{defn:acc}
We say that a finite-type $T$ satisfies the \textit{ACC} condition if the following three conditions hold: 
\begin{enumerate}
    \item (A1) For every component $J$ of $X$ and each point $a \in J$, we have that the orbit of $a$ up to and including the return time to $J$ contains at most one critical point of $T$;
    \item (A2) For every dynamically non-trivial component $J$ of $X$, we have that none of the boundary points of $J$ land on discontinuities up to and not including the return time to $J$;
    \item (A3) For every discontinuity $\beta \notin X$, its ghost tree $\mathcal{GT}(\beta)$ does not contain $\beta$.
\end{enumerate}
\end{definition}

The reason for the name is that if any of these properties is violated, then the defining parameters $(\gamma \, \beta)$ of $T$ satisfy a non-trivial linear equation which we call a \textit{critical connection}. Note that these are particular critical connections that we want to avoid and that we do not need to avoid all of them. An easy consequence of A1 and A2 is that there are no discontinuities in the boundary of $X$. Indeed, the preimage in $X$ such a discontinuity is either a discontinuity, which violates A1, or a boundary point of a dynamically non-trivial interval of $X$, which violates A2.

Note that if A1 and A2 hold, then every vertex of the ghost tree $\mathcal{GT}(\beta)$ for a $\beta \notin X$ corresponds to a discontinuity that is not contained in $X$. Indeed, by the above, we have that there are no discontinuities in the boundary. Thus if $\beta_2^-$ is a ghost preimage of $\beta_1^+$, and $\beta_2^-$ is contained in $X$, this means that $\beta_1^+$ is also contained in $X$. Thus every ghost preimage of a $\beta \notin X$ must also not be in $X$. Thus the claim follows by induction. 

\begin{definition}[Matching]
We say that a finite-type $T$ satisfies the \textit{Matching} condition if for every dynamically non-trivial interval $J$ of $X$, we have that exactly one point $a \in J$ lands on a critical point before returning to $J$, and this point $a$ is in the interior of $J$.
\end{definition}

The reason we call this property Matching is that if $a$ is the single point from above, then we must have that:

\begin{enumerate}
    \item $J = [R_J(a^+), R_J(a^-))$;
    \item $R^2_j(a^+) \sim R^2_J(a^-)$,
\end{enumerate}
and thus the second iterates under $R_J$ of $a^+$ and $a^-$ must `match'. An analogous property also called Matching has been discussed in slightly different contexts (see \cite{MR3893724}, \cite{MR3597033}, \cite{MR2422375}). In the case when no point in the interior of $J$ lands on a discontinuity, we know that $J$ is a maximal periodic interval and thus both boundary points of $J$ land on discontinuities. If A1 and A2 hold for $T$, then the boundary points of such intervals only land on discontinuities and are not actually themselves discontinuities of $T$, since there are no discontinuities in the boundary of $X$, by the discussion below Definition \ref{defn:acc}.

With the definitions out of the way, we can now recall the statement of Theorem B:

\begin{theorem}[Characterization of Stability]
\label{thm:stability=accm}
A finite-type interval translation map $T$ is stable if and only if it satisfies the ACC and Matching properties.
\end{theorem}

The proof of Theorem \ref{thm:stability=accm} is spread over the next three subsections. The outline of the proof is as follows. For the `only if' direction, we will show that a stable map $T$ that already satisfies A1 and A2 must also satisfy Matching. We will then show that any nearby map must also satisfy A1 and A2, which shows that any stable map must satisfy A1 and A2. If a stable map does not satisfy A3, we will be able to produce an explicit perturbation of the map that increases the non-wandering set by a definite size, which gives a contradiction. Putting all of this together shows that stability implies ACC and Matching. For the other direction, we will show that the ACC and Matching properties are very rigid, and that therefore any map close to such a map must also satisfy these properties, giving stability.

\subsection{Stability, A1 and A2 imply Matching}

Recall that $k_s(x)$ is the number of entries of the orbit of $x$ to the interval $I_s$ up to, but not including, time $k$:
\[
k_s(x) \coloneqq \# \{ T^j(x) \in I_s \text{ for } 0 \le j < n \}.
\]
With this definition, we have the following explicit formula for any iterate of $x$:

\[
T^k(x) = x + \sum_{s=1}^r k_s(x) \gamma_s.
\]
We now recall the notation for the dynamics of a return map $R_J$ to $J$, from the beginning of Section 4. The notation is slightly changed, so that it better fits the notation of this section. Let $J = [x,y)$ be a dynamically non-trivial interval of $X$ and let $a_1, a_2, \dots a_{N-1}$ be the points in $J$ which land on discontinuities before returning to $J$ and let $l_1, l_2, \dots, l_N$ be the corresponding landing times on critical points. More precisely, let $l_i \ge 0$ be the smallest $l_i$ such that $T^{l_i}(a_i) = \beta_{J,i} \in \mathcal{C}$. By the A1 assumption, $l_i$ is unique. Let $J_1 = [x,a_1^-), J_2 = [a_1^+, a_2^-), \dots, J_N = [a_{N-1}^+,y)$, and let $k_1, k_2, \dots, k_{N}$ be the return times of these intervals, i.e. $R_J(J_i) = T^{k_i}(J_i)$. Note that we must have $l_i < k_i$ and $l_i < k_{i+1}$. We will also use the labels $a_0^+ \coloneqq x$ and $a_{N+1}^- = y$, but these points do not land on discontinuities by A2.

Let us note that we consider all $a$'s in $J$ which land on critical points. For simplicity, we choose to use the term `discontinuity of $R_J$' as a shorthand for `the points in $J$ which land on discontinuities of $T$ before returning to $J$', even though the return map might not actually be discontinuous at these points. Likewise, we will refer to the intervals $J_1, \dots, J_{N}$ as the `continuity intervals of $R_J$'.

It is convenient to associate to each $J$ a permutation $\sigma_J$ which corresponds to the order on $J$ in which the subintervals return to $J$. Thus if $d=2$ and the order in which the intervals return is: $R_J(J_3) \; R_J(J_2) \; R_J(J_1)$, then $\sigma_J = (3 2 1)$. 

Let $\tau_J = (i_1, i_2, \dots, i_{N})$ be the inverse of $\sigma_J$, so that the order of the images under the return map is $R_J(J_{i_1}) \; R_J(J_{i_2}) \dots R_J(J_{i_{N}})$. Then the following equations must be satisfied:

\[
\begin{array}{cc}
\label{eq:RJeq}
     &  R_J(a_{i_1}^-) \sim R_J(a_{i_2-1}^+) \\
     &  R_J(a_{i_2}^-) \sim R_J(a_{i_3-1}^+) \\
     & \dots \\
     & R_J(a_{i_{N-1}}^-) \sim R_J(a_{i_{N}-1}^+).
\end{array}
\]
These equations are central to the stability of the interval $J$, i.e. to the property that $J$ always has a continuation $\Tilde{J}$ for any sufficiently small perturbation. To show this, we need to analyse what happens when we perturb a map $T$ that is stable and satisfies A1 and A2. By $\Tilde{T}$ we will denote a perturbation of $T$. As before, by $\Tilde{Z}$ we will denote the continuation of any object of interest $Z$, e.g. $\Tilde{\beta}^+$, $\Tilde{X}$, $\Tilde{T}$ and similar.

The main thing that properties A1 and A2 allow is control over finite time itineraries for any sufficiently small perturbation. To make this precise, consider any point $x \in I$ that does not land on a critical point before time $n$. Let $s(x,i)$, for $0 < i < n$, be index such that  the iterate $T^i(x)$ is contained in $I_{s(x,i)}$. Associated to each iterate $T^i(x)$ we therefore have the following two vectors:

\begin{align*}
v^{i,+} &:= \left(\sum_{1 \le s \le r} i_s(x) \, \bm{e}_s, \,  -\bm{f}_{s(x,i)-1}\right); \\
v^{i,-} &:= \left(\sum_{1 \le s \le r} i_s(x) \, \bm{e}_s, \, -\bm{f}_{s(x,i)}\right).
\end{align*}

Thus we have the following formula for the distance between an iterate of $x$ and the critical points in $I_{s(x,i)}$:

\begin{align*}
T^i(x) - \beta_{s(x,i)-1}^+ = x + \langle v^{i,+}, (\gamma \, \beta) \rangle; \\
\beta_{s(x,i)}^- - T^i(x) = -\langle v^{i,+} , (\gamma \, \beta) \rangle - x.
\end{align*}

By the assumption on time $n$, all of the quantities above are positive for $0 < i < n$. Thus they remain positive for all sufficiently small perturbation $\Tilde{T}$ of $T$, because they clearly depend continuously on $(\gamma \, \beta)$. This means that the itinerary of $x$ up to time $n$ remains the same for all sufficiently small perturbations. Note that this in fact also holds for a small interval of definite size around this point. This follows immediately from the fact that if some orbit up to time $n$ of some point $x$ does not land on a discontinuity, then the same holds true for an interval of definite size (depending on $n$) around $x$.

We would like to have the same conclusions for itineraries of critical points. For this we have to slightly change the argument because the points themselves change under perturbation. Assume that some discontinuity $\beta$ does not land on a discontinuity before time $n$. Then analogously as above we define $s(\beta,i)$, for $0 < i < n$, to be index such that  the iterate $T^i(\beta)$ is contained in $I_{s(\beta,i)}$. Again, we have the following vectors:

\begin{align*}
v^{i,+}(\beta) &:= \left(\sum_{1 \le s \le r} i_s(\beta) \, \bm{e}_s, \, \bm{f}_{\text{ind}(\beta)} -\bm{f}_{s(x,i)-1}\right); \\
v^{i,-}(\beta) &:= \left(\sum_{1 \le s \le r} i_s(\beta) \, \bm{e}_s, \, \bm{f}_{\text{ind}(\beta)} -\bm{f}_{s(x,i)}\right),
\end{align*}

and the following formula for the distance from the critical points in $I_{s(x,i)}$:

\begin{align*}
T^i(\beta) - \beta_{s(x,i)-1}^+ = \langle v^{i,+}, (\gamma \, \beta) \rangle; \\
\beta_{s(x,i)}^- - T^i(\beta) = -\langle v^{i,+} , (\gamma \, \beta) \rangle .
\end{align*}

Thus by the same argument as before, we have that the itinerary of $\beta$ up to time $n$ does not change for all sufficiently small perturbations.

We can analogously do this for the backward itinerary of any point $x$ or discontinuity $\beta$, so we get that they also persist under sufficiently small perturbations. Moreover, we can use ACC2 in the same way to prove that the itinerary of every boundary point of $X$ also remains the same after a sufficiently small perturbation and that this also holds on a neighbourhood of definite size around these points.

We will now put together these observations to prove that the structure of the entire return map $R_J$ to any interval $J$ of $X$ remains stable for a sufficiently small perturbation. More precisely, we will prove that the number of points that land on discontinuities remains the same, that the images of $J_1, \dots, J_N$ under $R_J$ remain in the same order and that their return times remain the same. We first give a lemma that eliminates some annoying cases in the proof of Proposition \ref{prop:stab-return-map}.

\begin{lemma}
\label{lem:not-fixed}
Let $T$ be a stable map that also satisfies property A2. Then we must have that have that $R_J(J_1) \neq J_1$ and $R_J(J_{N}) \neq J_{N}$, for any dynamically non-trivial interval $J$ of $X$. In other words, the return map cannot fix the boundary sub-intervals. 
\end{lemma}

\begin{proof}
We give the proof for $J_1$, with the one for $J_{N}$ being analogous. Assume the opposite, that $R_J(J_1) = J_1$. Then we have that $R_J(x) = x$, i.e. $x$ is periodic. Since $x = T^m(\beta^+)$, where $\beta \in X$ is a discontinuity of the starting map $T$, and $T\vert_X$ is a bijection, by pulling back with $(T\vert_X)^{-1}$, we see that $x$ must land on $\beta^+$ before it returns to $J$. This is a contradiction with A2. 
\end{proof}

\begin{proposition}[Stability of the return map]
\label{prop:stab-return-map}
Let $T$ be a stable map that also satisfies properties A1 and A2. Let $\Tilde{T}$ be a sufficiently small perturbation of $T$. Let $J = [x, y)$, where $x = T^{m_1}(\beta_{J,\tau(1)-1}^+)$ and $\ = T^{m_2}(\beta_{J,\tau(N)}^-)$, be a dynamically non-trivial interval of $X$ and let $\Tilde{J}$ be its continuation. Then the following properties hold:
\begin{enumerate}
    \item $\Tilde{J} = [\Tilde{T}^{m_1}(\Tilde{\beta}_{J,\tau(1)-1}^+),\Tilde{T}^{m_2}(\Tilde{\beta}_{J,\tau(N)}^-))$;
    \item The only points in $\Tilde{J}$ that land on critical points before they return are $\Tilde{a}_1, \dots, \Tilde{a}_{N-1}$, where $\Tilde{a}_i := \Tilde{T}^{-l_i}(\Tilde{\beta}_{J,i})$ for $1 \le i \le N-1$;
    \item Each $\Tilde{a}_i^{+}$, for $0 \le i \le N-1$, and $\Tilde{a}_i^{-}$, for $1 \le N$, has the same itinerary up to time $k_i$ as before the perturbation;
    \item The return time of each interval $[\Tilde{a}_i^+, \Tilde{a}_{i+1}^-)$ is $k_{i+1}$ (where we use the convention that $\Tilde{a}_0^+ = \Tilde{x}$ and $\Tilde{a}_{d+1}^- = \Tilde{y}$);
    \item The order in $\Tilde{J}$ in which the intervals return is given by the same permutation $\tau = (j_1, j_2, \dots, j_{d+1})$ as for $J$.
\end{enumerate}
\end{proposition}

\begin{proof}
We know that the itinerary of $\Tilde{a}_i^{\pm}$ contained in the interior of $\Tilde{J}$ remains unchanged until time $k_i$, when it lands back into $\Tilde{J}$. This comes from the stability of the backward itinerary of $\beta_{J,i}$ until time $l_i$ and the stability of the forward orbit of $\beta_{J,i}$ up to time $k_i$. That the same holds for $\Tilde{a}_0^+$ and $\Tilde{a}_N^-$ follows from the stability of the forward orbit of the boundary points of $X$. Thus the third property follows.

For the second property, assume the contrary: there exists some other point $z \in \Tilde{J}$ that lands on a discontinuity before returning to $\Tilde{J}$. Let $\Tilde{J}_i = [\Tilde{a}_{i-1}^+,\Tilde{a}_{i}^+)$ be the interval that contains $z$. In the case when both $\Tilde{a}_{i-1}^+$ and $\Tilde{a}_{i}^+$ return to $\Tilde{J}$ at time $k_i$, we know that $z$ must land on a discontinuity before time $k_i$. This is clearly impossible, as the itineraries of $\Tilde{a}_{i-1}^+$ and $\Tilde{a}_{i}^+$ must therefore be different, which contradicts the third property. Thus we know that $i = \tau(1)$ or $i = \tau(N)$. We may assume the former, with the proof in the other case being analogous. This means that $\Tilde{a}_{i-1}^+$ and $z$ land outside of $\Tilde{J}$ at time $k_i$. The entire orbit of $\Tilde{J}$ must be contained in $\Tilde{X}$, so in particular the interval $[T^{k_i}(\Tilde{a}_{i-1}^+),\Tilde{a}_0^+]$ must also be contained in $X$. This contradicts the maximality of $\Tilde{J}$, so the second property follows.

In fact, the proof of the second property also implies the fourth. Indeed, if we assume that some interval $[\Tilde{a}_{i-1}^+,\Tilde{a}_{i}^+)$ does not return to $\Tilde{J}$ at time $k_i$, we get a contradiction with the maximality of $\Tilde{J}$. The fifth property follows from this as well.

The first property follows from Lemma \ref{lem:not-fixed}: $\Tilde{a}_0^+$ must be the image under $\Tilde{R_J}$ of some $\Tilde{a}_i^+$ with $i > 0$. This must be the same $i$ as for $T$. The proof for $\Tilde{a}_N^-$ is analogous, so the first property follows.

\end{proof}

Let us state as a corollary the fact that this lemma immediately implies that the equations for the return map still hold after perturbation:

\begin{corollary}
\label{cor:stab-eq}
Let $T$ and $J$ be as in Proposition \ref{prop:stab-return-map}. Then for a sufficiently small perturbation of $T$, all of the following equations still hold:
\[
\begin{array}{cc}
\label{eqn:pert-RJeq}
     &  \Tilde{R}_J(\Tilde{a}_{i_1}^-) \sim \Tilde{R}_J(\Tilde{a}_{i_2-1}^+) \\
     &  \Tilde{R}_J(\Tilde{a}_{i_2}^-) \sim \Tilde{R}_J(\Tilde{a}_{i_3-1}^+) \\
     & \dots \\
     & \Tilde{R}_J(\Tilde{a}_{i_{N-1}}^-) \sim \Tilde{R}_J(\Tilde{a}_{i_{N}-1}^+).
\end{array}
\]
\end{corollary}

With this proposition, we are ready to prove that any dynamically non-trivial interval $J$ of $X$ contains exactly one discontinuity.

\begin{proposition}
\label{prop:exactly-1}
Let $T$ be a stable map that satisfies the ACC condition and let $J$ be a dynamically non-trivial interval of $X$. Then the return map $R_J$ to $J$ contains exactly one discontinuity.
\end{proposition}

\begin{proof}
Assume first that the return map has $3$ of more discontinuities. Then there exists an equation:
\[
\Tilde{R}_J(\Tilde{a}_{i_j}^-) \sim \Tilde{R}_J(\Tilde{a}_{i_{j+1}-1}^+)
\]
such that neither $\Tilde{a}_{i_j}^-$ nor $\Tilde{a}_{i_{j+1}-1}^+$ are boundary points of $\Tilde{J}$. Thus we have that :
\[
\Tilde{a}_{i_j}^- + \sum_{s=1}^r (k_{i_j})_s(a_{i_j}^-) \Tilde{\gamma}_s = \Tilde{a}_{i_{j+1}-1}^+ + \sum_{s=1}^r (k_{i_{j+1}})_s(a_{i_{j+1}-1}^+) \Tilde{\gamma}_s
\]
for any sufficiently small perturbation $\delta$. Let us first assume that the discontinuities $\Tilde{a}_{i_j}^-$ and $\Tilde{a}_{i_{j+1}-1}^+$ are different, i.e. that $i_j \neq i_{j+1}-1$.

Now, let the perturbation $\Tilde{T}$ be such that we only change $\beta_{J,i_j}$ to $\Tilde{\beta}_{J,i_j} \coloneq \beta_{i_j} + \epsilon$, where $\epsilon$ is sufficiently small so that Proposition \ref{prop:stab-return-map} applies. By Corollary \ref{cor:stab-eq}, we have that:

\begin{align*}
\Tilde{a}_{i_j}^- + \sum_{s=1}^r (k_{i_j})_s(a_{i_j}^-) \Tilde{\gamma}_s &= \Tilde{a}_{i_{j+1}-1}^+ + \sum_s (k_{i_{j+1}})_s(a_{i_{j+1}-1}^+) \Tilde{\gamma}_s \\
\implies \epsilon &= 0,
\end{align*}
as the only value that changes is $a_{i_j}^-$, since $\Tilde{a}_{i_j}^- := \Tilde{T}^{-l_{i_j}}(\Tilde{\beta}_{J,i_j}^-)$ by definition. This is clearly a contradiction, so we may assume $i_j = i_{j+1}-1$. Then if we make a perturbation that only changes some $\gamma_t$ to $\Tilde{\gamma}_t := \gamma_t + \epsilon$, we get:

\begin{align*}
&\epsilon\left(-(l_{i_j})_t(a_{i_j}^-) + (k_{i_j} - l_{i_j})_t(\beta_{J,i_j}^-)\right) = \\
&\epsilon\left(-(l_{i_{j+1}-1})_t(a_{i_{j+1}-1}^+) + (k_{i_{j+1}} - l_{i_{j+1}-1})_t(\beta_{J,i_{j+1}-1}^+)\right),
\end{align*}
by Corollary \ref{cor:stab-eq}. Since the above holds for all $\epsilon$ and $t$, we get:
\[
(k_{i_j} - l_{i_j})_t(\beta_{J,i_j}^-) = (k_{i_{j+1}} - l_{i_{j+1}})_t(\beta_{J,i_{j+1}-1}^+),
\]
since $(l_{i_j})_t(a_{i_j}^-) = (l_{i_{j+1}-1})_t(a_{i_{j+1}-1}^+)$, as these points have the same itinerary up to time landing time $l_{i_j} = l_{i_{j+1}-1}$ to $\beta_{J,i_{j}}$. Since this holds for all $t$, we must in fact have that the return vectors $v_{i_j}$ and $v_{i_j+1}$, defined as:

\begin{align*}
v_{i_j} &:= \left(\sum_{1 \le s \le r} (k_{i_j})_s(a_{i_j-1}^+) \, \bm{e}_s, \,  0\right) \\
v_{i_j+1} &:= \left(\sum_{1 \le s \le r} (k_{i_j}+1)_s(a_{i_j}^+) \, \bm{e}_s, \,  0\right)
\end{align*}
of $J_{i_j}$ and $J_{i_j+1}$ are equal. This is a contradiction with Corollary \ref{cor:lin-indep}. Indeed, we have that $v_{i_j}$ is equal to the sum of all vectors in Corollary \ref{cor:lin-indep} related to the orbit up to return time of $a_{i_j-1}^+$ and that $v_{i_j+1}$ is equal to the sum of all vector related to the orbit of $a_{i_j}^+$. Thus if they are equal, we get a non-trivial dependence between the vectors in Corollary \ref{cor:lin-indep} by putting the coefficient $1$ in front of vector in Corollary \ref{cor:lin-indep} related to the orbit up to return time of $a_{i_j-1}^+$, coefficient $-1$ in front of every vector related to the orbit of $a_{i_j}^+$ and $0$ in front of every other vector. This is a contradiction, so this case is impossible.

Assume now that $J$ contains exactly two discontinuities. By Lemma \ref{lem:not-fixed}, we know that we must have $\sigma_J = (3 2 1)$. The above equations now reduce to:

\begin{align*}
&a_1^+ + \sum_{s=1}^r (k_2)_s(a_1^+)\gamma_s = y + \sum_{s=1}^r (k_3)_s(y)\gamma_s \\
&a_2^- + \sum_{s=1}^r (k_2)_s(a_2^-)\gamma_s = x + \sum_{s=1}^r (k_1)_s(x)\gamma_s,
\end{align*}
where $x = T^{k_3-l_2}(\beta_{J,2}^+)$ and $y = T^{k_1-l_1}(\beta_{J,1}^-)$. Thus by doing the same perturbations as above, we get that:

\begin{align*}
-(l_1)_t(a_1^+) + (k_2-l_1)_t(\beta_{J,1}^+) &= (k_3 + k_1 - l_1)_t(\beta_{J,1}^-) \\
-(l_2)_t(a_2^-) + (k_2-l_2)_t(\beta_{J,2}^-) &= (k_1 + k_3 - l_2)_t(\beta_{J,2}^+),
\end{align*}
for all $1 \le t \le r$. If we define the vectors $v_1, v_2$ and $v_3$ as above, this implies $v_1 + v_3 = v_2$. In the same way as above, this contradicts Corollary \ref{cor:lin-indep}, as we again get a non-trivial linear dependence between vectors from this corollary.

Thus we get that it is impossible for $J$ to contain more than one discontinuity. As it is by assumption dynamically non-trivial, it must therefore contain exactly one discontinuity.

\end{proof}

Proposition \ref{prop:exactly-1} tells us that every dynamically non-trivial interval $J$ must contain exactly one point that lands on a discontinuity before returning $J$, which implies Matching. Thus we have the following result:

\begin{theorem}[Stability + A1 + A2 $\implies$ Matching]
\label{thm:s+A1+A2=M}
A stable map that satisfies properties A1 and A2 must also satisfy the Matching property. \qed
\end{theorem}

\subsection{Stability implies ACC}

Because of Theorem \ref{thm:s+A1+A2=M}, we now only need to prove that stability implies ACC to complete the first implication of Theorem \ref{thm:stability=accm}. Our proof will be indirect: we will first perturb the map so that it satisfies ACC, and therefore Matching, and use this to conclude the same thing about the starting map. 

We now recall the definition of rational independence:

\begin{definition}[Rational independence]
Let $\{ s_1, s_2, \dots, s_n \}$ be a finite set of real numbers. These numbers are said to be rationally independent if the following holds:
\[
\left( \sum_i q_i s_i = 0, \; \text{with} \; q_i \in \mathbb{Q} \right) \implies \left( q_i = 0 \; \forall i \right).
\]
\end{definition}

This condition on the set $\{ \beta_1, \dots, \beta_r, \gamma_1, \dots, \gamma_r \}$ is stronger than ACC. Indeed, if ACC does not hold, i.e. if one of A1, A2 or A3 is violated, then we have an equation of the following form:
\[
\beta + \sum_{s=1}^r n_s \gamma_s = \betas,
\]
which implies that the the numbers $\{ \beta_1, \dots, \beta_r, \gamma_1, \dots, \gamma_r \}$ are rationally dependent.

Rational independence, and therefore ACC, can be achieved by an arbitrarily small perturbation. This follows from the next lemma applied to the set  $\{ \beta_1, \dots, \beta_r, \gamma_1, \dots, \gamma_r \}$.

\begin{lemma}
\label{rat-ind}
Let $\{ s_1, s_2, \dots, s_n \}$ be a finite set of real numbers. Then we can make an arbitrarily small perturbation to each of the numbers so that they become rationally independent.
\end{lemma}

\begin{proof}
The proof is by induction. Assume that the numbers $s_1, \dots, s_k$ are rationally independent. We need to show that we can add $s_{k+1}$ do this set by an arbitrarily small perturbation. Indeed, the set of all numbers which are rationally dependent with $s_1, \dots, s_k$ is equal to the set of all linear combinations with rational coefficients of these numbers. This set is only countably infinite, so there is a real number arbitrarily close to $s_{k+1}$ which is rationally independent from $s_1, \dots, s_k$.
\end{proof}

Before going into the proof of Theorem \ref{thm:s-acc}, we give two lemmas about the maps that satisfy properties A1, A2 and Matching.

\begin{lemma}
\label{lem:indep-orb}
Let $T$ be a finite type map that satisfies properties A1, A2 and Matching. Then the orbit of any point $z \in X$ enters at most one interval of $X$ that contains a discontinuity of $T$.
\end{lemma}

\begin{proof}
Assume the contrary, that for some point $z$ we have that $T^{k_1}(z) \in J_1 = [x_1,y_1) \ni \beta$ and $T^{k_2}(z) \in J_2 = [x_2,y_2) \ni \betas$, with $k_1 < k_2$. Without loss of generality, we may assume that $T^{k_1}(z)$ lands into $J_2$ before returning to $J_1$. Indeed, we can choose $k_1$ to be the last time $z$ is in $J_1$ before landing to $J_2$. Moreover, we may assume that $T^{k_2}(z)$ lands to $J_1$ before returning to $J_2$, by choosing $k_2$ as the last time $T^{k_1}(z)$ is in $J_2$ before returning to $J_1$. We may assume $T^{k_1}(z) \in [\beta^+, y_1)$, with the other case being analogous. By Matching, no point except $\beta^+$ in the interior of this interval lands on a discontinuity before returning to $J_1$, $\beta^+$ does not land because of A1 and $y$ does not land because of A2. Thus it must at time $k_2 - k_1$, when it is contained in $J_2$, be compactly contained in one of the intervals $[x_2, \betas^-), [\betas^+, y_2)$, which we call $J'$. As $J_2$ is dynamically non-trivial, by Matching and ACC1 no point in $J'$ can land on discontinuity before returning to $J_2$. By our choice of $k_2$, the interval $T^{k_2-k_1}([\beta^+, y_1))$ must land to $J_1$ before it returns to $J_2$. At this time $\beta^+$ must land on $x$. But the points from $J'$ to the left of $T^{k_2-k_1}(\beta^+)$ must land outside of $J$. Since they must be contained in $X$, this contradicts the maximality of $J$.
\end{proof}

\begin{lemma}
\label{lem:triv-int-struct}
Assume that a finite type map $T$ satisfies properties A1, A2 and Matching. Then every trivial interval $J = [x,y)$ of $X$ has the property that $x$ lands only on the discontinuity $\beta^+$ and $y$ lands only on the discontinuity $\beta^-$, for some $\beta \in X$.
\end{lemma}

\begin{proof}
Assume first that $x$ lands on $\betas^+$ at the same time $y$ lands on $\betass^-$. This means that $\betas^+$ and $\betass^-$ are contained in the same interval of $X$. Since by A1 and A2 there are no discontinuities in the boundary of $X$, we know that the return map to this interval of $X$ has at least two discontinuities of its return map. This is a contradiction with Matching, so this case is impossible.

Let $T^{k_1}(x) = \betas^+$ and $T^{k_2}(y) = \betass^-$, and assume without loss of generality that $k_1 < k_2$. By A1 and A2, the interval $J_1$ of $X$ containing $T^{k_1}(J)$ must be dynamically non-trivial, which means by Matching that no point in $J_1$, except $\betas$, lands on a discontinuity before returning to $J_1$. But $\betas^+$ and $\betas^-$ also do not land on any other discontinuities because of A1. As we know that $y$ must land on a discontinuity, this forces $\betass^- = \betas^-$.
\end{proof}

We are now ready to prove the following theorem, which alongside Theorem \ref{thm:s+A1+A2=M} proves the first implication of Theorem \ref{thm:stability=accm}:

\begin{theorem}
\label{thm:s-acc}
A stable map $T$ satisfies the ACC condition.
\end{theorem}

\begin{proof}
Let us first derive some simpler properties of $T$. Let $\Tilde{T}$ be a sufficiently small perturbation of $T$ which satisfies properties A1 and A2, and for which Lemma \ref{lem:loc-disc}. Then $\Tilde{T}$ also satisfies the Matching condition, by Theorem \ref{thm:s+A1+A2=M}. This means that there are no discontinuities in the boundary of $\Tilde{X}$, so the same must be true for $X$, by Lemma \ref{lem:loc-disc}. Moreover, each dynamically non-trivial interval $\Tilde{J}$ of $\Tilde{X}$ contains at most one discontinuity in its interior. By Lemma \ref{lem:loc-disc}, the same must be true for each corresponding interval $J$ of $X$. For a trivial interval $J$, Lemma \ref{lem:triv-int-struct} shows that the boundary points of $J$ land only on the $+$ and $-$ of a single discontinuity $\beta \in X$.

Assume that $T$ does not satisfy A1: some interval $J = [x,y)$ contains a point $a$ which lands on a discontinuity more than once before returning to $J$. By the last part of the paragraph above, we know that $J$ is dynamically non-trivial. By the first part, we know that $a$ is in the interior of $J$ and that $a^+$ and $a^-$ are the only points in $J$ that land on a discontinuity. We may without loss of generality assume that $a^+$ lands on two discontinuities, with the other case being analogous. Thus there are times $k_1 < k_2$ such that $T^{k_1}(a^+) = \betas^+ \in J_1$ and $T^{k_2}(a^+) = \betass^+ \in J_2$. Let $z$ be a point in the interval $[a^+,y)$ which does not land on a critical point before return time and assume that the above perturbation is small enough so that $z$ is still contained in $\Tilde{J}$ return time and has the same itinerary up to return time to $\Tilde{J}$. By Lemma \ref{lem:indep-orb}, this means that the $\Tilde{T}$-orbit of $z$ up to return time to $\Tilde{J}$ lands into at most one interval of $\Tilde{X}$ that contains a discontinuity. Thus we have the same conclusion for the $T$-orbit of $z$ up to return time, as the $T$ and $\Tilde{T}$ orbits of $z$ pass through corresponding intervals of $X$ and $\Tilde{X}$, respectively. Moreover, this has to also hold for every point in the continuity interval $[a,y)$ of $R_J$ containing $z$. Thus $a^+$ does not land in both $J_1$ and $J_2$, which is a contradiction. Thus $T$ must satisfy A1.  

Next, assume that $T$ violates A2: there is a point $x$ in the boundary of a dynamically non-trivial interval $J = [x,y)$ of $X$ that lands on a discontinuity before return time (the case for $y$ is analogous). Let $z$ be a point in the interior of $J$, to the left of $a$, the only discontinuity of $R_J$. In particular, the orbit of $z$ up to return time does not land on any critical point. We may thus assume that the perturbation is small enough so that $z$ has the same itinerary up to return time with respect to the intervals of $X$ and Lemma \ref{lem:indep-orb} holds for $z$ and $\Tilde{T}$. Thus the $T$-orbit of $z$ has the same properties: it enters at most one interval of $X$ that contains a discontinuity of $T$. Therefore this has to hold until return time for the continuity interval $[x,a)$ of $R_J$ that contains $z$. Thus it is impossible that both $x$ and $a$ land on a discontinuity, which is a contradiction. Thus $T$ must satisfy A2 as well.

Finally, assume that $T$ violates A3. We will show that we can perturb $T$ in such a way that the $X$ increases by definite size, contradicting the stability of $T$. Indeed, if A3 does not hold, then there is some $\beta$ that appears in its ghost tree $\mathcal{GT}(\beta)$. Without loss of generality, we may assume that it is of $+$-type and equal to $\beta_{i_1}^+$. Thus we have a finite path in the ghost tree: $\beta_{i_1}^+ \leftarrow \beta_{i_2}^- \leftarrow \beta_{i_3}^+ \leftarrow \dots \leftarrow \beta_{i_{2k+1}}^+ = \beta_{i_1}^+$. By A1 and A2, none of the discontinuities in this path are contained in $X$. For even $j$ with $2 \le j \le 2k$, let $k_j$ be the time at which $\beta_{i_j}^+)$ lands at $\beta_{{i_{j-1}}}^+$ and let $v_j$ be the landing vector:
\[
v_j := \left(\sum_{1 \le s \le r} (k_j)_s(\beta) \, \bm{e}_s, \, \bm{f}_{\text{ind}(\beta_{i_j}^+)} -\bm{f}_{\text{ind}(\beta_{i_{j-1}}^+)}\right).
\]
Define $k_j$ and $v_j$ analogously for odd $i$ with $3 \le j \le 2k+1$. If the orbit of $\beta_{i_j}^+$ lands on $n_i > 0$ other discontinuities before time $k_j$, we split $v_j$ into $n_j+1$ vectors that correspond to consecutive landings in this orbit. By Corollary \ref{cor:lin-indep}, all of these vectors are linearly independent, as they correspond to discontinuities in $\mathcal{C}_{\neg X}$. Thus for sufficiently small $\epsilon$, there exist a sufficiently small perturbation $\delta$ (in the sense that $\Tilde{T}$ is still in $\mathcal{U}$) so that we have:

\begin{align*}
\langle v_j, \delta \rangle = -\epsilon &\text{ for odd $i$ }; \\
\langle v_j, \delta \rangle = +\epsilon &\text{ for even $i$ },
\end{align*}

and so that $\beta_{i_j}^+$ lands on the same discontinuities before time $k_j$ as before, for all $1 \le j \le 2k+1$.

This makes $\beta_{i_1}^+$ periodic with period $\sum_{j=1}^{2k+1} k_j$, and is therefore contained in a periodic interval contained in $\Tilde{X}$. As $\beta_{i_1}^+$ was a definite distance away from $X$, since there are no discontinuities in the boundary of $X$, this means that for a sufficiently small perturbation the set $\Tilde{X}$ contains more discontinuities than $X$. This contradicts the stability of $T$, so A3 follows.
\end{proof}

\subsection{ACC + Matching implies Stability}

We are now going to prove the other direction in Theorem \ref{thm:stability=accm}, which finishes the proof.

\begin{theorem}
\label{thm:acc+m-stability}
Let $T$ be a finite type map that satisfies the ACC and Matching properties. Then $T$ is a stable map.
\end{theorem}

\begin{proof}

We will show that there for a sufficiently small perturbation of $T$, $\Tilde{X}$ is close to and homeomorphic to $X$, and contains the same discontinuities as $X$. By definition, this shows that $T$ is stable.

By A1 and A2 we know that that there are no discontinuities in the boundary of $X$. Moreover, for any $T$ that satisfies A1, A2 and Matching, every dynamically non-trivial interval $J$ has the property that only a single point $a$ in the interior of $J$ lands on a discontinuity and that the return map to $J$ is a rotation. This gives $X = \bigsqcup_{\beta \in X} O(J_\beta)$, where $J_{\beta}$ is the dynamically non-trivial interval of $X$ containing $\beta$. These orbits are disjoint and there is a definite distance between them because of Lemma \ref{lem:indep-orb}.

We have that every dynamically non-trivial interval $J$ is of the form $[T^{k_2}(a^+), T^{k_1}(a^-))$, where $R_J = T^{k_1}$ on $J_1 = [x,a^-)$ and $R_J = T^{k_2}$ on $J_2 = [a^+,y)$. Moreover, for every such interval $J$ of $X$, we have that there is an interval $J' = [x', y')$ containing $J$ such that:
\begin{itemize}
    \item No point in $[x',\beta^-)$ lands on discontinuity up to time $k_1$ and no point in $[\beta^+, y')$ lands on a discontinuity up to time $k_2$.
    \item The intervals $[x',x)$ and $[y,y')$ have definite size.
\end{itemize}
Indeed, this follows from A2. We may assume that $J'$ is sufficiently small so that for a sufficiently small perturbation of $T$, the itineraries of points in $[x',\Tilde{a}^-)$ up to time $k_1$ and $[\Tilde{a}^+, y')$ remain the same as for $T$. Moreover, we may assume that the perturbation is small enough so that the itineraries of $a^+$ and $a^-$ up to time $k_1 + k_2$ remain the same as for $T$. Thus the interval $[\Tilde{T}^{k_2}(\Tilde{a}^+), \Tilde{T}^{k_1}(\Tilde{a}^-))$ has a well-defined return map on it that is a rotation, so it is in particular contained in $\Tilde{X}$. We will show that it must be an interval of $\Tilde{X}$. For a sufficiently small perturbation, it is compactly contained in $J'$. This means that the points in $[x',\Tilde{T}^{k_2}(\Tilde{a}^+))$ land into $[\Tilde{T}^{k_2}(\Tilde{a}^+), \Tilde{T}^{k_1}(\Tilde{a}^-))$ at time $k_1$ and points in $[\Tilde{T}^{k_1}(\Tilde{a}^-),y')$ land into $[\Tilde{T}^{k_2}(\Tilde{a}^+), \Tilde{T}^{k_1}(\Tilde{a}^-))$ after time $k_2$, which means that they are not contained in $\Tilde{X}$. Thus $[\Tilde{T}^{k_2}(\Tilde{a}^+), \Tilde{T}^{k_1}(\Tilde{a}^-))$ is an interval of $\Tilde{X}$, so it is in fact equal to $\Tilde{J}$.

For every trivial interval $J$ of $X$, we know that there is a single dynamically non-trivial interval $J_{\beta}$ of $X$ that contains a discontinuity $\beta$ of $T$ into which $J$ lands, by Lemma \ref{lem:indep-orb} and Lemma \ref{lem:triv-int-struct}. Thus it must be contained in the orbit of $J_{\beta}$. Let $J_1$ and $J_2$ be the two continuity intervals of $J_{\beta}$. As $J$ is trivial, it must in fact be equal to an iterate of either $J_1$ or $J_2$. We may assume the former, with the other case being analogous. Assume $J = T^{k_1}(J_1)$. By the discussion for dynamically non-trivial intervals, this iterate moves continuously for sufficiently small perturbations. Thus $\Tilde{T}^{k_1}(\Tilde{J}_1)$ is an interval contained in $\Tilde{X}$. We now show that it must be maximal. Indeed, there is again an interval $J' \supset J$ that is of a definite size larger than $J$ that has the same itinerary as $J$ up to the time $J$ lands into $J_{\beta}$. We may assume $J'$ is sufficiently small so that it lands into the interior of $J_{\beta}'$ and that this continues to hold for all sufficiently small perturbations. As above, this means that the points in $J' \setminus \Tilde{T}^{k_1}(\Tilde{J}_1)$ are not contained in $\Tilde{X}$, which show that $\Tilde{T}^{k_1}(\Tilde{J}_1)$ is an interval of $\Tilde{X}$.

Thus we know that every interval $J$ of $X$ has a well-defined continuation $\Tilde{J}$ for all sufficiently small perturbations. We will show that $\Tilde{X}$ must in equal $\bigsqcup_{\beta \in X} \Tilde{O}(\Tilde{J}_\beta)$, which completes the proof. Let $X'$ be the union of all intervals $J'$.

Let us first show that this is the case if the itineraries of critical points outside of $X$ do not change up to their landing time into the interior $J'$. Indeed, since for any sufficiently small perturbation the points in $J' \setminus \Tilde{J}$ eventually land into $\Tilde{J}$, and thus into $\Tilde{X}$, if none of the itineraries change, then all of the critical points outside of $X$ still land into the interior of $X'$. Moreover, if these itineraries do not change, then they also do not change for intervals of definite size containing these discontinuities. We may assume these intervals are small enough so that they also land into the interior of $X'$ after a sufficiently small perturbation. Thus all of these intervals are not contained in $\Tilde{X}$, so all of the discontinuities that were not in $X$ remain outside of $\Tilde{X}$. By the Orbit Classification Lemma \ref{lem:orb-class}, we know that every point in $z \in I \setminus X'$ that is not eventually periodic must either land on a discontinuity or accumulate on a discontinuity. If it lands on a discontinuity, it is clearly not in $\Tilde{X}$. It either accumulates onto a discontinuity inside of $X'$, or onto a discontinuity outside of $X'$, and in this case, it must eventually land into the interval of definite size that lands into the interior of $X'$. Thus in both cases, it eventually lands into $\bigsqcup_{\beta \in X} \Tilde{O}(\Tilde{J}_\beta)$, which means that it is not in $\Tilde{X}$. If $z \in I \setminus X'$ is eventually periodic, it must land into a periodic interval of $\Tilde{T}$. The boundary points of such an interval must be the landings of discontinuities contained in $\Tilde{X}$, so $z$ must also land into $\bigsqcup_{\beta \in X} \Tilde{O}(\Tilde{J}_\beta)$. Thus every point must eventually land into $\bigsqcup_{\beta \in X} \Tilde{O}(\Tilde{J}_\beta)$, which means $\Tilde{X} = \bigsqcup_{\beta \in X} \Tilde{O}(\Tilde{J}_\beta)$. By the characterization of finite type maps \ref{thm:fin-type-char}, we know that $\Tilde{T}$ is of finite type, so the claim follows.

Thus it must hold that for any sufficiently small perturbation, there is a critical point outside of $X$ for which the itinerary changes. By the paragraph above, we may assume that the perturbation is small enough so that if the itinerary of a critical point does not change, then there is an interval of definite size that contains it and lands into the interior of $X'$ at the same time as before the perturbation. Moreover, we may assume that the perturbation is small enough so that the itinerary of every discontinuity outside of $X$ up to some finite time in which it does not land on another discontinuity remains the same as before the perturbation. Thus if some $\beta \notin X$ does not land on a discontinuity before the time it lands into the interior of $X'$, we may assume it still does so.

Now, let $\Tilde{T}$ be a sufficiently small perturbation as in the preceding paragraph. Let $\Tilde{\beta_{i_1}^+}$ (we may assume it is of $+$-type, with the other case being analogous) be a discontinuity that:
\begin{itemize}
    \item It changes itinerary before landing into the interior of $X'$;
    \item It does not eventually land into $X'$.
\end{itemize}
Such a discontinuity exists because if every discontinuity that changes itinerary still lands into $X'$, then there are again intervals of definite size containing such discontinuities that also land into $X'$, so in the same way as in the paragraph above we must have that every point $z \in I \setminus X'$ eventually lands into $X'$, which gives $\Tilde{X} = \bigsqcup_{\beta \in X} \Tilde{O}(\Tilde{J}_\beta)$.

We will show now that there is a finite path $\beta_{i_1}^+ \leftarrow \beta_{i_2}^- \leftarrow \dots \leftarrow \betas$ in the ghost tree of $\beta_1^+$ such that $\betas$ already appears earlier in the path. This means that $\betas$ must be in its ghost tree $\mathcal{GT}(\betas)$, which contradicts A3.

Let $n_1$ be the first time at which the itinerary of $\beta_{i_1}^+$ changes. Since the perturbation is small enough, we know that $\beta_{i_1}^+$ must have landed on a discontinuity $\beta_{i_2}^+$ at time $n_1$, and now it lands to the left of it. We now claim that there is a discontinuity $\beta_{i_3}^-$ such that $\beta_{i_2}^-$ lands on $\beta_{i_3}^-$ at some time $n_2$ and such that $\Tilde{T}^{n_1 + n_2}(\Tilde{\beta}_{i_1}^+)$ lands to the right of $\Tilde{\beta}_{i_3}^+$.

Indeed, if $\beta_{i_2}^-$ does not change itinerary before landing into $X'$, the same holds for a definite size interval to the left of $\Tilde{\beta}_{i_2}^-$ and thus for $\Tilde{T}^{n_1}(\Tilde{\beta}_{i_1}^+)$, which is a contradiction with our choice of $\beta_{i_1}^+$. Let $m_{2,1}$ be the first time at which $\beta_{i_2}^-$ changes itinerary. In the same way above, we know that we must have $T^{m_{2,1}}(\beta_{i_2}^-) = \beta_{{i_2},1}^-$ for some discontinuity $\beta_{2,1}^-$ not contained in $X$ and that $\beta_{i_2}^-$ now lands to the right of $\beta_{i_2,1}^+$ after perturbation. If $\Tilde{T}^{n_1 + m{2,1}}(\Tilde{\beta}_{i_1}^+)$ is to the left of $\Tilde{\beta}_{i_2,1}^-$, then we may repeat this argument for $\beta_{i_2,1}^-$, and get that it must land at some $\beta_{i_2,2}^-$ at time $m_{2,2}$, and it has to land to the right of $\Tilde{\beta}_{i_2,2}^+$ after the perturbation. Continuing by induction, we get that there is a minimal finite $t_2$ such that $\Tilde{T}^{n_1 + m_{2,1} + \dots + m_{2,t_2}}(\beta_{i_1}^+)$ lands to the right of $\beta_{i_2,t_2}^+$ and that we have $\beta_{i_2,1}^- \to \beta_{i_2,2}^- \to \dots \to \beta_{i_2,t_2}^-$. There has to exist such a finite $t_2$ because there are finitely many discontinuities, so each iterate $\Tilde{T}^{m_{2,i}}(\Tilde{\beta}_{i_2,i}^i)$ moves to the right by at least $\epsilon > 0$. Thus at each step $i$ of the induction $\beta_{i_2,i}^- - \Tilde{T}^{n_1 + m_{2,1} + \dots + m_{2,i}}(\beta_{i_1}^+)$ has to be smaller than $\beta_{i_2,i-1}^- - \Tilde{T}^{n_1 + m_{2,1} + \dots + m_{2,i-1}}(\beta_{i_1}^+)$ by at least $\epsilon$. We can therefore set $n_2 := m_{2,1} + \dots + m_{2,t_2}$ and $\beta_{i_3}^+ := \beta_{i_2,t_2}^+$.

This argument can now be repeated for $\beta_{i_3}^+$, so by induction we get an infinite path $\beta_{i_1}^+ \leftarrow \beta_{i_2}^- \leftarrow \beta_{i_3}^+ \leftarrow \dots$ in the ghost tree of $\beta_{i_1}^+$. Since there are only finitely many discontinuities, one of them has to repeat, so we get the required finite sequence.

\end{proof}

\section{Eventually periodic maps can be approximated by stable maps}
\label{proofoftheoremC} 

\subsection{Unstable number and correspondence}

Recall the statement of Theorem C:

\begin{theorem}
\label{thm:ep->acc+m}
Let $T$ be an eventually periodic map. Then arbitrarily close to $T$ in the parameter space, there is a stable map $\Tilde{T}$.  
\end{theorem}

The idea of the proof is to inductively remove the critical connections that prevent the map $T$ from being stable. Our main tools will be Theorem \ref{thm:lin-dep} and Corollary \ref{cor:lin-indep}. In this subsection, we formalise this idea and prove a few preliminary lemmas, before going into the proof of Theorem \ref{thm:ep->acc+m} in the next subsection. We first introduce a few concepts related to the dynamics of the discontinuities of an eventually periodic map $T$.

\begin{definition}[Cycle of discontinuities]
Let $\beta$ be a discontinuity of an eventually periodic map $T$ contained in $X$. Then the cycle of $\beta$ is defined as:
\[
C(\beta) = O(\beta) \cap \mathcal{C},
\]
i.e. the set of all discontinuities in the orbit of $\beta$.
\end{definition}

We say that a cycle is non-trivial if $C(\beta) \neq \{\beta\}$. In this case $\beta$ lands on at least one other discontinuity of $T$.

Assuming $T$ is eventually periodic, we know that for all $\betas, \betass \in X$ we have that $\betas \in C(\betass)$ if and only if $\betass \in C(\betas)$. Thus the set of all discontinuities contained in $X$ is partitioned into cycles, which we consider as equivalence classes of the relation $\betas \simeq \betass$ if and only if $\betas \in O(\betass)$. Let $\mathcal{Z}$ be the set of all equivalence classes, i.e. the set of all cycles.

A stable eventually periodic map has the property that the size of each cycle is exactly one. Additionally, it has no discontinuities in the boundary of $X$. With this in mind, we define a number which tells us how far an eventually periodic map is from having this property:

\begin{definition}[Unstable number of $T$]
Let $T$ be an eventually periodic map. Then the \textit{unstable number} of $T$, denote by $U(T)$m is defined as:
\[
U(T) \coloneqq \left( \sum_{C \in \mathcal{Z}} (|C|-1) \right) + |\partial X \cap \mathcal{C}|
\]
\end{definition}

For a periodic discontinuity $\beta$, we will denote by $P(\beta)$ the maximal periodic interval containing $\beta$, in the sense of Definition \ref{defn:max-per-int}. For brevity, we will usually omit stating maximality when discussing periodic intervals, but it is always assumed.

\begin{definition}[Correspondence]
\label{def:corr}
Let $T$ be an eventually periodic map. We say that $T$ has the \textit{correspondence property} if for each discontinuity $\beta^+ \in X$, we have that $\beta^-$ eventually lands into the maximal periodic interval $P(\beta^+)$ and if analogously for every discontinuity $\betas^- \in X$, $\betas^+$ eventually lands into $P(\betas^-)$.
\end{definition}

Since the orbits of $\beta^-$ and $\beta^+$ both end up in the same periodic cycle of intervals, these points have `corresponding dynamics'. This is a fairly strong property. One of the things it implies is the following restriction on the structure of every interval $J$ of $X$:

\begin{lemma}
\label{lemma:corr-rj-structure}    
Let $T$ be an eventually periodic map that has the correspondence property. Then for every interval $J$ of $X$, there exists a discontinuity $\beta^+ \in X$ such that $J$ is contained in the union of iterates of $P(\beta^+)$.
\end{lemma}

Note that since $T$ is eventually periodic, each interval of $X$ already has the property of being `tiled' by periodic intervals. If we additionally assume correspondence, the lemma above says that there exists a $\beta^+ \in X$ such that all of these periodic intervals are iterates of $P(\beta^+)$.

\begin{proof}
Assume the contrary, that there is an interval $J$ of $X$ containing iterates of two periodic intervals with disjoint orbits. Let $P_1$ and $P_2$ be two touching periodic intervals in $J$ with this property. By maximality, the point at which they touch must eventually land on a discontinuity $\beta$. Thus $\beta^-$ and $\beta^+$ belong to periodic intervals with disjoint orbits, which contradicts the correspondence property.
\end{proof}

A map that has unstable number zero and satisfies the correspondence property is almost stable:

\begin{lemma}
\label{lemma:corr+u=a1a2m}
Let $T$ be an eventually periodic map. Assume that $U(T) = 0$ and that $T$ satisfies the correspondence property. Then $T$ also satisfies the A1, A2 and Matching properties.
\end{lemma}

\begin{proof}
Since $U(T) = 0$, no point in $X$ can land on two discontinuities, so A1 follows.

By the Lemma \ref{lemma:corr-rj-structure}, we know that any interval $J$ of $X$ is tiled by intervals that are all iterates of a single maximal periodic interval $P(\beta^+)$. Since $U(T) = 0$, $\beta^+$ does not land on any other discontinuity. Since $P(\beta^+)$ is a maximal periodic interval, there exists a discontinuity $\betas^- \in X$ such that the right boundary point of $P(\beta^+)$ lands on $\betas^-$ at some time $k$. Additionally, we have that $P(\betas^-) = T^k(P(\beta^+))$.

Assume that the left boundary point of a dynamically non-trivial interval $J$ lands on a discontinuity before returning to $J$. The case for the right boundary point is analogous. We know that this discontinuity must be $\beta^+$. Since the interval is dynamically non-trivial, a point in the interior of $J$ must also land on a discontinuity of $T$ before returning to $J$. This discontinuity has to be either $\beta$ or $\betas$. In the first case, we get a contradiction with the fact that orbits of the components of a return map must be disjoint up to return time. In the second case, we get that $\beta^+$ lands on $\betas^+$, which is a contradiction with $U(T) = 0$. Thus both cases are impossible, so A2 follows.

If a dynamically non-trivial interval $J$ does not satisfy the Matching property, then there are two points in the interior of $J$ that land on discontinuity before returning to $J$. These discontinuities can again be only $\beta$ or $\betas$, and we get a contradiction in the same way as for A2. Thus Matching follows as well.
\end{proof}

Our goal now is to produce a finite sequence of sufficiently small perturbations each of which decreases the unstable number of $T$, and also makes $T$ satisfy the correspondence property. By induction, this will result in a map that has unstable number zero and satisfies the correspondence property. We then do a final perturbation to get a map that also satisfies A3, which by Lemma \ref{lemma:corr+u=a1a2m} means it satisfies ACC and Matching. 

The main idea of the proof is to remove the discontinuities that contribute to the unstable number of $T$, i.e. are in a cycle inside of $X$ or in the boundary of $X$, or violate the correspondence property, while preserving the dynamical properties of points not in the orbit of these discontinuities. 

This result allows us to show that an arbitrarily small perturbation of $T$ satisfies the correspondence property and does not increase the unstable number:

\begin{lemma}[Perturbation to correspondence]
\label{lem:pert-corr}
Let $T$ be an eventually periodic map. There exists an arbitrarily small perturbation $\Tilde{T}$ of $T$ that satisfies the correspondence property and $U(\Tilde{T}) \le U(T)$.
\end{lemma}

Similar to the statements of Theorem \ref{thm:lin-dep} and Corollary \ref{cor:lin-indep}, for brevity, we will omit stating the ranges of all indices explicitly.

\begin{proof}
Assume that $T$ does not satisfy the correspondence property at some discontinuity $\beta$. Without loss of generality, assume that $\beta^+ \in X$, with the other case being analogous. Let $n = |\mathcal{Z}|$ be the number of critical cycles in $X$. Let $J = J_1$ be the periodic interval $P(\beta^+)$. For each of the remaining $n-1$ cycles, let $\beta'$ be any discontinuity from this cycle and let $JP(\beta')$ be the corresponding maximal periodic interval. We label these intervals as $J_2, \dots, J_n$, in any order. By construction, these are maximal intervals with disjoint orbits, so they satisfy the assumption of Corollary \ref{cor:lin-indep} (note that we may clearly omit the interval $J_0$ from the statement of Corollary \ref{cor:lin-indep} and still get the same conclusion for the vectors related to $J_1, \dots, J_n$). By linear independence, there exists a arbitrarily small perturbation $\delta$ of parameters $(\gamma \, \beta)$ such that:
\begin{align*}
&\langle C^{i,+}(0,k), \delta \rangle = \langle C^{1,-}(1,k), \delta \rangle = 0 \text{ for i } > 1, \\
& \langle C_{\beta}, \delta \rangle = 0 \text{ for } \beta \in \mathcal{C}_{\neg X}, \\
& \langle C^{1,+}(0,k), \delta \rangle = 0 \text{ for } k < m_0^{1,+}, \\
& \langle C^{1,-}(1,k), \delta \rangle = 0 \text{ for } k < m_1^{1,-}, \\
& \langle C^{1,+}(0,m_i^+), \delta \rangle = -\epsilon,
\end{align*}
where $\epsilon > 0$ is arbitrarily small. In the first line, we are using $C^{i,-}(1,m^{i,-}_1) = \sum_{k=1}^{m^{i,+}_0} C^{i,+}(0,k) - \sum_{k=1}^{m^{i,-}_1-1} C^{i,+}(1,k)$. If $\epsilon$ is sufficiently small, the first two equations above imply that the dynamics of any critical point whose orbit is disjoint from the orbit of $J_1$ remain unchanged. Indeed, the first equation gives that every such point in $X$ is still periodic and that its periodic interval moves continuously with $\delta$. The second shows that any such point in $\mathcal{C}_1$ still lands on a periodic discontinuity. Finally, we may assume that $\delta$ is small enough so that every such point in $\mathcal{C}_2$ still lands into the interior of a periodic interval, as these intervals move continuously with $\delta$. In particular, for any such point, it is in $\Tilde{X}$ after perturbation if and only if it was in $X$ before the perturbation. 

Note that the points in $\mathcal{C}_1$ that are not contained in $X$ still land on discontinuities in $O(J_1)$ and that the points in $\mathcal{C}_2$ still land into the interior of the orbit of $J_1$ for a sufficiently small perturbation.

By the third and fourth equations, the dynamics of the interval $J_1$ also do not change up until the time it returns to itself, when, by the fifth equation, the interval $I_{\epsilon}(\beta^+) = [\beta^+, \beta^+ + \epsilon)$ now lands to the left of $\beta^-$. By assumption, $\beta^-$ does not land into $J_1$. Since $T$ is eventually periodic, $\beta^-$ must eventually into a periodic interval, whose orbit must be disjoint from the orbit of $J_1$. By the above, the dynamics of this periodic interval and of $\beta^-$ remain unchanged after the perturbation. Thus, for $\epsilon$ sufficiently small, the interval $I_{\epsilon}(\beta^+)$ eventually lands into a periodic interval whose orbit does not intersect $I_{\epsilon}(\beta^+)$. Thus $I_{\epsilon}(\beta^+)$ is not contained in $\Tilde{X}$ anymore. By our choice of $\delta$, every point in $J_1$ eventually lands into $I_{\epsilon}(\beta^+)$, the entire orbit of $J_1$ is also not in $\Tilde{X}$. Thus we have that the entire critical cycle $C(\beta^+)$ is not in $\Tilde{X}$ anymore. Thus every discontinuity that was not contained in $X$ and landed into $J_1$ is also not contained in $\Tilde{X}$. As the dynamics of all other critical points do not change, we have that $U(\Tilde{T}) \le U(T)$. We also know that $\Tilde{T}$ is eventually periodic, since all critical points that land into the orbit of $J_1$ eventually land into a periodic interval, while all other critical points remain eventually periodic.

By induction, after a finite number of arbitrarily small perturbations, we have that $\Tilde{T}$ satisfies the correspondence property and that $U(\Tilde{T}) \le U(T)$.
\end{proof}

\subsection{Proof of Theorem \ref{thm:ep->acc+m}}

As mentioned before, our strategy is to remove from $X$ the discontinuities that are contained in non-trivial critical cycles. The perturbation we make depends on the location of the critical value corresponding to the critical point we want to remove. More precisely, we want the critical value to not be the first or the last one in an interval $J$ of $X$. If the number of continuity of $J$ is at least $4$ such a critical value can easily be found. The cases with less continuity intervals need to be treated separately. In all cases, the perturbation we choose is simple and explicit, but a lot of things need to be checked to see that it works. Similar to the statements of Theorem \ref{thm:lin-dep} and Corollary \ref{cor:lin-indep}, we will assume that the boundary points of $J_0$ land on discontinuities of $T$, and do the proof under this assumption. The analysis in the other case is simpler.

\begin{proof}[Proof of Theorem \ref{thm:ep->acc+m}]
By Lemma \ref{lem:pert-corr}, we may assume that $T$ satisfies the correspondence property. Assume first that there is a $\beta^+ \in X$ that lands on at least one other discontinuity, with the proof of $\beta^-$ being analogous. Let $J_0$ be the interval of $X$ containing such a $\beta^+$ with the maximal number of branches, and let $N_0$ be the number of branches of the return map to $J_0$. If $N_0 \ge 2$, we may assume that $\beta^+$ is in the interior of $J_0$. Indeed, if it is in the boundary, we can replace $\beta^+$ by the first critical point in the orbit of $J_0$ that is contained in the interior of an iterate of $J_0$. Recall that $\sigma$ denotes the permutation corresponding to the order in which the intervals of $J_0$ return to $J$, and that $\tau := \sigma^{-1}$.

\underline{Case 1 : $N_0 > 3$}

Let $v_2$ be the second critical value of the return map $R_{J_0}$ to $J$, with respect to the order inside of $J_0$ and let $J^0_{\tau(2)}$ and $J^0_{\tau(3)}$ be the two continuity intervals of the return map such that the images of their boundary points touch at $v_2$. Let $a_p^0$, with $1 \le p \le N_0-1$, be the first critical point of $R_{J_0}$ in the orbit of $v_2$, which exists because $J_0$ is, by Lemma \ref{lemma:corr-rj-structure}, equal to the union of iterates of a single periodic interval. Let $P \ge 0$ be minimal time such that $R_{J_0}^P(v_2) = a_p^0$. Let $a$ and $b$ be the integers:
\begin{itemize}
    \item $a:= \# \{ R_{J_0}^t(v_2) \in J_{\tau(2)}^0; 0\le t < P \}$;
    \item $b := \# \{ R_{J_0}^t(v_2) \in J_{\tau(3)}^0; 0\le t < P \}$.
\end{itemize}
Let $\epsilon_1, \epsilon_2 > 0$ be sufficiently small such that for $b > 0$ we have that:
\[
\frac{a}{b+1} < \frac{\epsilon_2}{\epsilon_1} < \frac{a+1}{b},
\]
or if $b = 0$, such that:
\[
a \epsilon_1 < \epsilon_2.
\]
In both cases, $\epsilon_1$ and $\epsilon_2$ satisfy:
\[
-\epsilon_2 < -a \epsilon_1 + b \epsilon_2 < \epsilon_1.
\]
Note that $\epsilon_1$ and $\epsilon_2$ can clearly be chosen arbitrarily small. Moreover, we may choose $\epsilon_1$ and $\epsilon_2$ to be of the form $r_1 |P(\beta^+)|$, $r_2 |P(\beta^+)|$, where $r_1$ and $r_2$ are small rational numbers.

Let $n$ be the number of critical cycles in $X$ not contained in the orbit of $J_0$. For each such cycle, let $\beta'$ be one of the discontinuities contained in it and let $P(\beta')$ be the corresponding maximal periodic interval. We label these intervals as $J_1, \dots, J_n$. The intervals $J_0, J_1, \dots, J_n$ therefore satisfy the assumptions of Corollary \ref{cor:lin-indep}. Thus we may make a perturbation $\delta$ of the parameters $(\gamma \, \beta)$ such that:
\begin{align*}
&\langle C^{i,+}(0,k), \delta \rangle = \langle C^{i,-}(1,k), \delta \rangle = 0 \text{ for } i \ge 1, \\
& \langle C_{\beta}, \delta \rangle = 0 \text{ for } \beta \in \mathcal{C}_{\neg X}, \\
& \langle L_{j}^0, \delta \rangle = \langle C^{0,+}(j,k), \delta \rangle= \langle C^{0,-}(j,k), \delta \rangle = 0, \\
& \langle R_{\tau(2)-1}^{0,+}, \delta \rangle = -\epsilon_1, \\
& \langle R_{\tau(3)-1}^{0,+}, \delta \rangle = \epsilon_2, \\
& \langle R_{j}^{0,+}, \delta \rangle = 0, \text{ for } j \neq \tau(2)-1, \tau(3)-1.
\end{align*}
The analysis is now similar as in Lemma \ref{lem:pert-corr}: the only dynamical change is that the image of the interval $J_{\tau(2)}^0$ at the time of returning to $J_0$ shifts to the left by $\epsilon_1$ (fourth equation) and the image of $J_{\tau(3)}^0$ shifts to the right by $\epsilon_2$ (fifth equation). Indeed, the first two equations guarantee that the dynamics of critical points with orbits disjoint from $J_0$ have the same dynamics as before. The third equation says that for each $0 \le j \le N_0$, the same point $a_j^0$ in $J_0$ lands on the discontinuity $\Tilde{\beta}_0(j)$. Moreover, the $\Tilde{T}$-orbit of $a_j^{0,\pm}$ up to the time $r_j^{0,\pm}$ contains same set of discontinuities as before the perturbation. Finally, the last three equations mean that each interval $[a_j^{0,+},a_{j+1}^{0,-})$, for $0 \le j < N_0$, returns to $J_0$ at the same time $r_j^+$ as before. Thus the return map $\Tilde{R}_{J_0}$ to $J_0$ is well-defined and has the same intervals of continuity. The images under this return map of every interval except $J_{\tau(2)}^0$ and $J_{\tau(3)}^0$ are also the same. In particular, their itineraries remain the same for sufficiently small $\epsilon$. The resulting perturbation for the return map with $5$ branches and the associated permutation $\sigma = (5 3 2 1 4)$ is shown in Figure \ref{fig:perturbation of a return map}.

\begin{figure}
    \centering
    \includegraphics[width=\linewidth]{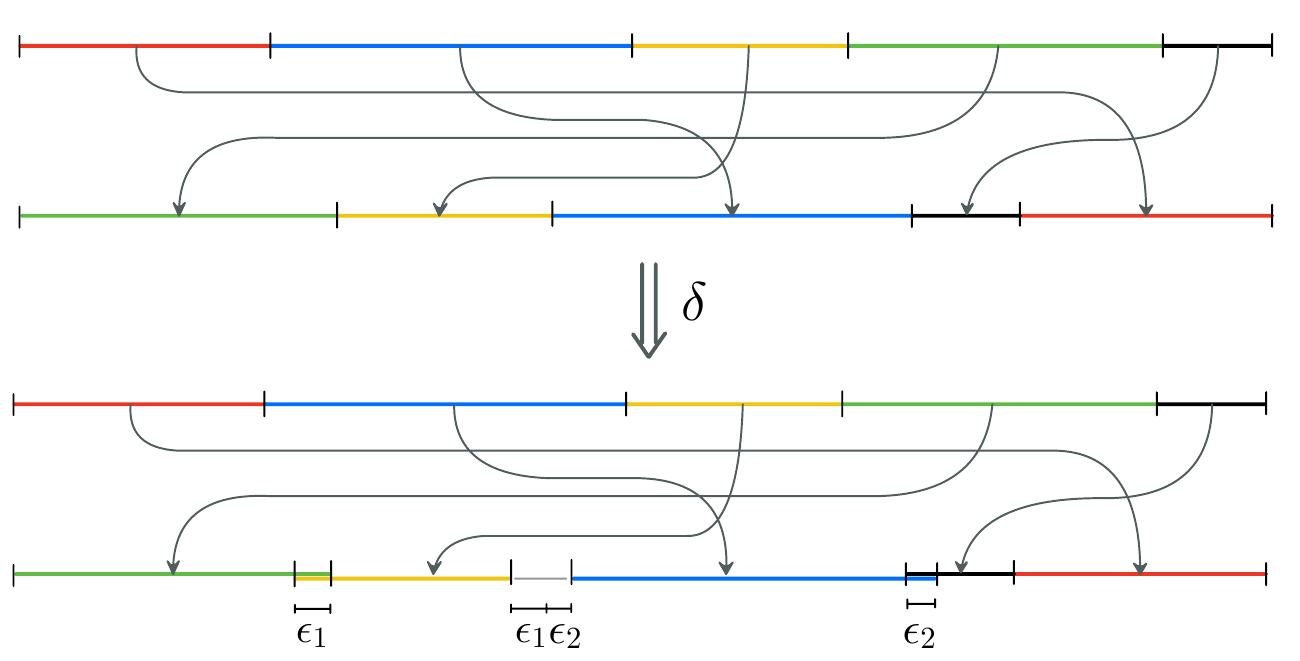}
    \caption{The result of applying $\delta$ to the particular $J_0$ for which the return map has $5$ branches.}
    \label{fig:perturbation of a return map}
\end{figure}

Because $\Tilde{R}_{J_0}$ is well-defined, we have that $\Tilde{X} \cap \Tilde{O}(J_0) = \Tilde{O}(\Tilde{X} \cap J_0)$ and $\Tilde{X} \cap J_0$ = $\bigcap_{n=1}^{\infty} \Tilde{R}^n_{J_0}(J_0)$. The interval $[v_2-\epsilon_1,v_2+\epsilon_2)$ is not in $\Tilde{X}$, since it is not in the image of $\Tilde{R}_{J_0}$ and $\Tilde{X} \cap J_0$ = $\bigcap_{n=1}^{\infty} \Tilde{R}^n_{J_0}(J_0)$. On the other hand, the intervals $[v_1-\epsilon_1, v_1)$ and $[v_3,v_3+\epsilon_2)$ now both have two $\Tilde{R}_{J_0}$-preimages, and every other point in $J_0$ only has a single $\Tilde{R}_{J_0}$-preimage. As $a_p$ is by assumption the first critical point in the orbit of $v_2$, we know that $v_2$ must land on it before it lands on $v_1$ or $v_3$, as they are critical values. Thus for sufficiently small $\epsilon_1$ and $\epsilon_2$, we know that $[v_1-\epsilon_1, v_1)$ and $[v_3,v_3+\epsilon_2)$ are not contained in the $\Tilde{T}$-iterates up to and including time $P$ of $[v_2-\epsilon_1,v_2+\epsilon_2)$. Thus all of these iterates of $[v_2-\epsilon_1,v_2+\epsilon_2)$ are also not contained in $\Tilde{X}$.

The image $\Tilde{R}^P_{J_0}([v_2-\epsilon_1,v_2+\epsilon_2))$ is equal to the image $R_{J_0}([v_2-\epsilon_1,v_2+\epsilon_2))$ shifted by $-a \epsilon_1 + b \epsilon_2$, by the definition of $a$ and $b$. Thus by our choice of $\epsilon_1$ and $\epsilon_2$, we have that $a_p^0$ is still contained in this image. Thus $a_p^0$ is also not contained in $\Tilde{X}$. Consider the set of $\Tilde{T}$ iterates of the intervals $J_j^0$ up to the times when they return to $J_0$. By our choice of $\delta$, we know all of these intervals are mutually disjoint, with the only possible exceptions being:
\begin{enumerate}
    \item Iterates of $J_{\tau(2)}$ for the times after $a^{0,+}_{\tau(2)-1}$ lands on $\beta^{0,+}(\tau(2)-1,m_{\tau(2)-1}^{0,+})$ and iterates of $J_{\tau(1)}$ for the times after $a^{0,-}_{\tau(1)}$ lands on $\beta^{0,-}(\tau(1),m_{\tau(1)}^{0,-})$;
    \item Iterates of $J_{\tau(3)}$ after the time $a^{0,-}_{\tau(3)}$ for the times after $\beta^{0,-}(\tau(3),m_{\tau(3)}^{0,-})$ and iterates of $J_{\tau(4)}$ for the times after $a^{0,+}_{\tau(4)-1}$ lands on $\beta^{0,+}.(\tau(4)-1,m_{\tau(4)-1}^{0,+})$.
\end{enumerate}
Since there are only two possible exceptions to $\Tilde{R}_{J_0}$ iterates up to return time of the $J_i$'s to not be disjoint, we know that we have that every point in the orbit of $a_p^0$ up to the landing on $\beta_0(p)$, as well as all of the discontinuities in the orbit up to return time of $\beta_0^+(p)$ and $\beta_0^-(p)$ to $J_0$ do not have any additional iterated $\Tilde{T}$-preimages contained in $J_0$. Since $\Tilde{X} \cap \Tilde{O}(J) = \Tilde{O}(\Tilde{X} \cap J)$, this means that all of these points are not contained in $\Tilde{X}$.

Assuming that $\Tilde{T}$ is eventually periodic, this shows that $U(\Tilde{T}) < U(T)$. The proof is the same as in Lemma \ref{lem:pert-corr}. The discontinuities in $\mathcal{C}_1$ and $\mathcal{C}_2$ that landed into $J_0$ still do so for a sufficiently small $\delta$, so they are still not contained in $\Tilde{X}$, as we can choose $\delta$ small enough so that $\Tilde{O}(J_0)$ remains a positive distance away from these discontinuities. Moreover, every critical cycle not contained in the $T$-orbit of $J_0$ remains unchanged, and all of the discontinuities that land into intervals corresponding to these cycles are still not in $\Tilde{X}$. Thus the contribution of all of these critical points to $U(\Tilde{T})$ remains the same. The set of discontinuities contained in $\Tilde{O}(J_0) \cap \Tilde{X}$ is smaller by at least one and they possibly split into several cycles instead of forming a single one. Thus the contribution of these discontinuities to $U(\Tilde{T})$ is strictly smaller than before, so $U(\Tilde{T}) < U(T)$ follows.

Finally, we see that $\Tilde{R}_{J_0}$ is conjugate, via the affine map that sends $J_0$ to $[0,1)$, to an $\ITM$ on $N_0$ intervals for which all parameters are rational. Indeed, we know by Lemma \ref{lemma:corr-rj-structure} that $J_0$ is equal to the union of iterates of $P(\beta^+)$, so the length of every $J_j^0$ is an integer multiple of $|P(\beta^+)|$, and by our choice of $\epsilon_1, \epsilon_2$, every translation factor associated to every branch of $\Tilde{R}_{J_0}$ is also a rational multiple of $|P(\beta^+)|$. This means that every point in $J_0$ is eventually periodic. Moreover, every critical point that lands into $\Tilde{O}(J_0)$ is also eventually periodic. Since the dynamics of other critical points do not change, we see that $\Tilde{T}$ remains eventually periodic.

\underline{Case 2 : $N_0 = 3$}

This case is different because we only have two critical values of $R_{J_0}$ and thus we are unable to make the same perturbation as in Case 1 and still have that $\Tilde{R}_{J_0} \subset J_0$ (with $\Tilde{R}_{J_0}$ defined as in Case 1). The perturbation $\delta$ in this case depends on the permutation $\sigma$ associated to $R_{J_0}$. We give the proof for $\sigma = (3 2 1)$, with the proof in every other case being analogous.

Without loss of generality, we may assume that the first critical value of $R_{J_0}$ in the backward $R_{J_0}$-orbit of $\beta$ is $v_1$, with the other case being analogous. Let $P$ be such that $R_{J_0}^P(v_1) = \beta$. Thus we have that $v_1^- = R_{J_0}(y^{0,-})$ and $v_1^+ = R_{J_0}(a_1^{0,+})$. In this case, let:

\begin{itemize}
    \item $a:= \# \{ R_J^t(v_1) \in J_3; 0\le t < P \}$;
    \item $b := \# \{ R_J^t(v_1) \in J_2; 0\le t < P \}$.
\end{itemize}
Once again, let $\epsilon_1, \epsilon_2 > 0$ be sufficiently small and of the same form as before, so that for $b > 0$:
\[
\frac{a}{b+1} < \frac{\epsilon_2}{\epsilon_1} < \frac{a+1}{b},
\]
or if $b = 0$:
\[
a \epsilon_1 < \epsilon_2.
\]
In both cases, we again have that:
\[
-\epsilon_2 < -a \epsilon_1 + b \epsilon_2 < \epsilon_1.
\]
We can now choose the analogous intervals $J_1, \dots, J_n$ as in Case 1, and thus find a perturbation $\delta$ such that:
\begin{align*}
&\langle C^{i,+}(0,k), \delta \rangle = \langle C^{i,-}(1,k), \delta \rangle = 0 \text{ for } i \ge 1, \\
&\langle C_{\beta}, \delta \rangle = 0 \text{ for } \beta \in \mathcal{C}_{\neg X}, \\
& \langle L_{j}^0, \delta \rangle = \langle C^{0,+}(j,k), \delta \rangle= \langle C^{0,-}(j,k), \delta \rangle = 0, \\
& \langle R_{0}^{0,+}, \delta \rangle = -\epsilon_1, \langle R_{1}^{0,+}, \delta \rangle = \epsilon_2, \langle R_{2}^{0,+}, \delta \rangle = 0.
\end{align*}
Thus we once again have that $\Tilde{R}_{J_0}(J_0) \subset J_0$, and the rest of the analysis is the same as in Case 1. The dynamics outside of the orbit of $J_0$ do not change and $\Tilde{R}_{J_0}$ is conjugate to an $\ITM$ with rational parameters, so $\Tilde{T}$ is eventually periodic. To prove that $U(\Tilde{T}) < U(T)$, it is enough to show that $\Tilde{\beta} \notin \Tilde{X}$. The intervals $[v_1, v_1+\epsilon_2)$ and $[a_3^{0,-}-\epsilon_1, a_3^{0,-})$ are not in the image $\Tilde{R}_{J_0}(J_0)$ and therefore not in $\Tilde{X}$. Since $\langle R_{2}^{0,+}, \delta \rangle = 0$, we thus have that the only $\Tilde{R}_{J_0}$-preimage of $[v_1-\epsilon_1, v_1)$ is $[a_3^{0,-}-\epsilon_1, a_3^{0,-})$, and therefore $[v_1-\epsilon_1, v_1)$ is also not in $\Tilde{X}$. Thus the image $\Tilde{R}_{J_0}^P([v_1-\epsilon_1, v_1+\epsilon_2])$ is also not in $\Tilde{X}$. It is equal to the image $R_{J_0}^P([v_1-\epsilon_1, v_1+\epsilon_2])$ shifted by exactly $-a \epsilon_1 + b \epsilon_2$, and therefore still contains $\beta$. Thus $\beta \notin \Tilde{X}$.

\underline{Case 3 : $N_0 = 2$}

In the first two cases, we were able to make a perturbation that removed from $X$ a discontinuity that was the first one in the orbit of some $a_j^0 \in J_0$. Such a perturbation is not possible in this case, so we will instead remove the discontinuities that come later in the orbits of $a_0^{0,+}, a_1^{0,\pm}$ and $a_2^{0,-}$.

Again, we may choose $\delta$ so that the dynamics do not change outside the orbit of $J_0$, but for $J_0$ we have:
\begin{align*}
& \langle L_{0}^0, \delta \rangle = \epsilon, \langle L_{1}^0, \delta \rangle = 0, \langle L_{2}^0, \delta \rangle = -\epsilon, \\
& \langle C^{0,+}(1,1), \delta \rangle = \epsilon, \, \langle C^{0,-}(1,1), \delta \rangle = -\epsilon, \\
& \langle C_J^{0,+}(1,k), \delta \rangle = 0, \, \langle C_J^{0,-}(1,k), \delta \rangle = 0, \text{ for all } k > 1, \\
& \langle R_{0}^{0,+}, \delta \rangle = \epsilon, \langle R_{1}^{0,+}, \delta \rangle = -\epsilon,
\end{align*}
where $\epsilon > 0$ is arbitrarily small. The first equation means that the $\Tilde{T}$-iterates on $a_0^{0,+}$ and $a_2^{0,-}$ not land on the first critical points in their $T$-orbits. The second and third equation give that all of the critical connections in the $T$-orbit of $J_0$, except the landing of $\beta^{0,-}(1,1)$ to $\beta^{0,-}(1,2)$ and of $\beta^{0,+}(1,1)$ to $\beta^{0,+}(1,2)$, are preserved under the iterates of $\Tilde{T}$. The fourth equation means that the return map $\Tilde{R}_{J_0}$ to $J_0$ remains unchanged. This means that the interval $I^+_{\epsilon} = [\beta^{0,+}(1,2), \beta^{0,+}(1,2) + \epsilon)$ maps forward continuously with $\Tilde{T}$-iterates of $J_2^0$ and lands to the left of $J_1^0$ at the time when $J_2^0$ returns to $J_0$ under the iterates of $\Tilde{T}$. Analogously, the interval $I^-_{\epsilon} = [\beta^{0,-}(1,2) - \epsilon, \beta^{0,-}(1,2))$ lands to the right of $J_2^0$ at the time $J_1^0$ returns to $J_0$ under the iterates of $\Tilde{T}$. Thus $I^+_{\epsilon}$ follows $J_1^0$ continuously under the iterates of $\Tilde{T}$ until the time $J_1^0$ returns to $J_0$, and at this time $I^+_{\epsilon}$ lands into $J_0$. As the $\Tilde{T}$-orbit of $J_0$ is disjoint from $I^+_{\epsilon}$, we conclude that $I^+_{\epsilon}$ is not in $\Tilde{X}$. Analogously, we have that $I^-_{\epsilon}$ is also not in $\Tilde{X}$. Thus every discontinuity in the orbit of $a_0^{0,+},a_2^{0,-}, \beta^{0,+}(1)$ and $\beta^{0,-}(1)$ is not contained in $\Tilde{X}$ anymore, except for $\beta$. As $\Tilde{R}_{J_0}$ is the same as before perturbation, all of these discontinuities are eventually periodic and $\beta$ is still periodic. Thus $\Tilde{T}$ is still eventually periodic. Moreover, the critical cycles of $\beta^+$ and $\beta^-$ now have size one, so $U(\Tilde{T}) < U(T)$ as well.

\underline{Case 4 : $N_0 = 1$}

In this case $\beta^+$ is in the boundary of $J_0 = P(\beta^+)$. Since $T$ satisfies the correspondence property, we know that $\beta^-$ must eventually land into $J$. Thus we make a similar perturbation as in Case 3: we keep all of the dynamics outside of the orbit of $J_0$ as they are, and for $J_0$ we choose a $\delta$ such that:
\begin{align*}
& \langle L_{0}, \delta \rangle = 0, \langle L_{1}, \delta \rangle = -\epsilon, \\
& \langle C^{0,+}(0,1), \delta \rangle = \epsilon, \\
& \langle C^{0,+}(0,k), \delta \rangle = 0, \, \langle C^{0,-}(1,k), \delta \rangle = 0, \text{ for all } k > 1, \\
& \langle R_{0}^{0,+}, \delta \rangle = -\epsilon,
\end{align*}
where $\epsilon > 0$ is arbitrarily small. Thus we avoid the first critical connection, for both $a_0^{0,+} = \beta^+$ and $a_1^{0,-}$, and keep all of the others in the orbit of $a_0^{0,+}$. The return map to $J_0$ again does not change. Similarly as in Case 3, we again have that the interval $I^+_{\epsilon} = [\beta^{0,+}(0,2), \beta^{0,+}(0,2) + \epsilon)$ follows $J_0$ continuously under the iterates of $\Tilde{T}$ until $J_0$ returns to itself, when it lands to the left of $\beta^-$. By assumption, $\beta^-$ eventually lands into $J_0$, and thus, for sufficiently small $\epsilon$, $J^+_{\epsilon}$ eventually lands into $J$ as well. Since the $\Tilde{T}$-orbit of $J$ is disjoint from it, we know that $J^+_{\epsilon}$ is not in $X$. Same as in Case 3, we thus know that $\Tilde{T}$ is eventually periodic and $U(\Tilde{T}) < U(T)$.

Since we have now covered all cases, we may assume by induction that we have a map that satisfies the correspondence property and for which no critical point in $X$ lands on a different critical point. Assume now that there exists some discontinuity $\beta^+$ in the boundary of some interval $J$ of $X$, with the case for $\beta^-$ being analogous. Since $T$ satisfies the correspondence property, we know that $J_0 = P(\beta^+)$, because otherwise $\beta^+$ would have to land on some other discontinuity. We also know that $\beta^-$ must eventually land into $J$. Let $t_1$ be the time at which $\beta^-$ lands into $J_0$ and let $l$ be the distance between $\beta^+$ and $T^{t_1}(\beta^-)$. Let $t_2$ be the minimal period of $\beta^+$. We then make a perturbation $\delta$ that keeps all dynamics outside of the orbit of $J$ as they are and shifts the return map to $J_0$ to the left by $\epsilon$, i.e. $\Tilde{T}^{t_2}(\beta^+) = -\epsilon$. Note that this means that $\Tilde{T}^{t_1}(\beta^-) = T^{t_1}(\beta^-)$. For a sufficiently small $\epsilon$, the interval $[\beta-\epsilon, \beta)$ does not contain any $+$-type discontinuity before it lands into $J_0$, even after perturbation. Thus the return map to the interval $J' = [\beta-\epsilon, \beta+l)$ is now well defined and equal $\Tilde{T}^{t_1}$ on $J_1' = [\beta-\epsilon, \beta)$ and $\Tilde{T}^{t_2}$ on $J_2' = [\beta, \beta + l)$. Taking $\epsilon$ to be a rational multiple of $l$ means that $\beta^+$ remains eventually periodic. Thus, $\Tilde{T}$ is still eventually periodic and there is one less discontinuity in the boundary of $\Tilde{X}$. We have added $\beta^-$ to $\Tilde{X}$, but by construction, it is the only discontinuity in its cycle, and therefore $U(\Tilde{T}) < U(T)$.

Thus by induction, we get a map that satisfies the correspondence property and has unstable number zero. By Lemma \ref{lemma:corr+u=a1a2m} we know that it also satisfies properties A1, A2 and Matching. Assume now that it does not satisfy A3. This in particular means that $C_{\neg X} \neq \emptyset$. Then we can make a perturbation $\delta$ that does not change the dynamics of all critical points contained in $X$, and for which:

\begin{align*}
& \langle C_{\beta}, \delta \rangle = -\epsilon, \text{ for all } \beta \in C_{\neg X} \cap \mathcal{C}^-, \\
& \langle C_{\beta}, \delta \rangle = \epsilon, \text{ for all } \beta \in C_{\neg X} \cap \mathcal{C}^+,
\end{align*}
where $\epsilon > 0$ is small enough so that the itineraries up to the landing time into $X$ of discontinuities in $C_{\neg X}$ do not change. Thus there are no critical connections outside of $X$ anymore, so A3 must hold. Since the dynamics of all other critical points do not change, A1, A2 and Matching still hold. Thus the resulting map is stable.
\end{proof}

\section{Almost every finite type map is stable}
\label{sec:ae-fin-stable}

In this section, we show that the Boshernitzan--Kornfeld Conjecture implies the Irrational Rotations Conjecture. In fact, this results follows immediately the Characterization of Stability Theorem \ref{thm:stability=accm} and the following theorem:

\begin{theorem}
\label{thm:fin-tip-ae-rot}
The set of all infinite type maps has full measure in $ITM(r) \setminus S(r)$.    
\end{theorem}

\begin{proof}[Proof of Boshernitzan--Kornfeld Conjecture $\implies$ Irrational Rotations Conjecture]
Since infinite type maps have zero measure in $\ITM(r)$, by Theorem \ref{thm:fin-tip-ae-rot} the stable maps have full measure in $\ITM(r)$. By Theorem \ref{thm:stability=accm}, we know that stable maps correspond to maps for which the return map to every interval is either a rotation or the identity. The set of stable maps with rationally independent coefficients forms a full measure subset of stable maps, and for these maps, every return map corresponds to an irrational circle rotation. Thus the maps corresponding to irrational circle rotations form a full measure subset of $\ITM(r)$. 
\end{proof}

We now turn to the proof of Theorem \ref{thm:fin-tip-ae-rot}. Until now, we only needed to consider first return maps on intervals contained in $X$, for which it is known that they are bijective and have finitely many branches, and thus correspond to $\IET$s. It is clear that this notion can be generalized to an arbitrary interval $J \subset X$.

\begin{definition}
\label{defn:generelised-rJ}
Let $T$ be an $\ITM$ and assume that there is an interval $J \subset I$ such that each point of $I$ returns to $J$ under iterates of $T$. Then $J$ is partitioned into maximal intervals on which the itinerary up to the return time to $J$ is constant. The map $R_J$ defined piecewise on these intervals as the iterate of $T$ for which this interval returns to $J$ is called the return map to $J$.
\end{definition}

Compared to the definition of the return map to a component interval of $X$, the generalized return map does not need to be bijective. Thus the images of its branches are not necessarily disjoint. The next lemma tells us when a generalized return map has a continuation for all sufficiently small perturbations.

\begin{lemma}
\label{lem:J-continuity}
Let $T$ be an $\ITM$ and let $J \subset I$ be an interval such that the generalized return map $R_J$ is well-defined. Assume that:
\begin{itemize}
    \item For each point $a$ in the interior of $J$, we have that the orbit of $a$ contains at most one critical point of $T$;
    \item The boundary points $x,y$ of $J$ do not land on discontinuities for all time.
\end{itemize}
Then for any sufficiently small perturbation of $T$, there exists an interval $\Tilde{J}$ close to $J$ such that the generalized return map $R_{\Tilde{J}}$ to $\Tilde{J}$ is well-defined.
\end{lemma}

\begin{proof}
Let $J_1, \dots, J_n$ be the component intervals of $R_J$, let $r_1, \dots, r_n$ be their return times to $J$, and let $a_1, \dots, a_{n-1}$ be the points in the interior of $J$ that land on discontinuities. By the second assumption, we know that $J_1$ does not land back on itself when it returns to $J$. Indeed, this would mean that $x$ is a periodic point. Since the points to the left of $x$ are not contained in $X$, this means that $x$ must be the left boundary point of a maximal periodic interval, so it must eventually land on discontinuity, contradicting the second assumption. By an analogous argument, we know that $J_n$ also does not land on itself when it returns to $J$, so both $J_1$ and $J_n$ land into the interior of $J$ when they return to $J$. This means that there is an interval $\hat{J}$ containing $J$, which is larger by a definite size, and such that a generalized return map $R_{\hat{J}}$ on $\hat{J}$ is well-defined. By the first assumption, for every sufficiently small perturbation of $T$, the points $a_1, \dots, a_{n-1}$ have continuations $\Tilde{a}_1, \dots, \Tilde{a}_{n-1}$ which land on the same discontinuities as before and have the same landing times. This means that every interval $\Tilde{J}_i = [\Tilde{a}_{i-1},\Tilde{a}_{i})$ has the same itineraries up to time $r_i$. Let $\tau$ be the permutation associated to $J$, so that the intervals $J_{\tau(1)}$ and $J_{\tau(n)}$ get mapped to the front and back of $J$, respectively. For a sufficiently small perturbation $\Tilde{T}^{r_{\tau(1)}}(\Tilde{J}_{\tau(1)})$ and $\Tilde{T}^{r_{\tau(n)}}(\Tilde{J}_{\tau(n)})$ are both contained in $\hat{J}$. Thus the interval $[\Tilde{T}^{r_{\tau(1)}}(a^+_{\tau(1)-1}),\Tilde{T}^{r_{\tau(n)}}(a^-_{\tau(n)}))$ has a well-defined generalized return map and is close to $J$.
\end{proof}

We now show a lemma about the global stability of the critical set and critical orbits in $[T_0]$. Recall that $\mathcal{C}_{X}(\Tilde{T})$ denotes the set of all critical points of $\Tilde{T}$ contained in $X(\Tilde{T})$. For a $\Tilde{\beta} \in \mathcal{C}_{X}(\Tilde{T})$, let $J(\Tilde{\beta})$ be the component interval of $X(\Tilde{T})$ containing $\Tilde{\beta}$. Let $r(\Tilde{\beta}^-)$ and $r(\Tilde{\beta}^+)$ be the return times of $\Tilde{\beta}^-$ and $\Tilde{\beta}^+$ to $J(\Tilde{\beta})$, respectively.

\begin{lemma}
\label{lem:glob-crit-stab}
Let $[T_0]$ be the stable region. Then:
\begin{itemize}
    \item For any $\Tilde{T} \in [T_0]$ the set $\mathcal{C}_{X}(\Tilde{T})$ is equal to the set of continuations of the discontinuities contained in $\mathcal{C}_{X}(T_0)$.
    \item If the $\Tilde{\beta}$ is the continuation of a discontinuity $\beta \in \mathcal{C}_{X}(T_0)$, then $r(\Tilde{\beta}^-) = r(\beta^-)$ and $r(\Tilde{\beta}^+) = r(\beta^+)$.
    \item The $\Tilde{T}$-itineraries of $\Tilde{\beta}^-$ and $\beta^+$ are equal to the $T_0$-itineraries of $\beta^-$ and $\beta^+$ up to and including the time $r(\beta^-)+r(\beta^+)$.
\end{itemize}
\end{lemma}

More briefly, we will say that $\mathcal{C}_{X}(T_0)$, $r(\beta^-), r(\beta^+)$ and the itineraries of $\beta^+,\beta^-$ up to and including the time $r(\beta^-)+r(\beta^+)$ are constant in $[T_0]$. Note that for a stable map $\Tilde{T}$, all of these are locally constant in a small neighbourhood of $\Tilde{T}$ in parameter space. The proof is in the first part of the proof of Theorem \ref{thm:acc+m-stability}. Lemma \ref{lem:glob-crit-stab} tells us that this is true globally inside the stable region as well.

\begin{proof}
Since $[T_0]$ is connected and open, it is path-connected. Let $\eta:[0,1] \to [T_0]$ be a path starting at $T_0$ and ending at $\Tilde{T}$. Since $\eta(t)$ is a stable map, for every $t \in [0,1]$, there is a small neighbourhood $\mathcal{U}(t)$ of $\eta(t)$ such that the critical set, return times and itineraries of critical points up to the times described above are constant. By compactness, the cover of $\eta([0,1])$ by the neighbourhoods $\mathcal{U}(t)$ has a finite subcover $\mathcal{U}(t_1), \dots, \mathcal{U}(t_p)$, where $t_1 < \dots < t_p$. Since these neighbourhoods are open, $\mathcal{U}(t_i) \cap \mathcal{U}(t_{i+1})$ is non-empty, for all $1 \le i < p$. Thus the critical set, return times, and itineraries are constant on the union of these sets as well. By induction, it must be constant on the entire path $\eta([0,1])$, and thus constant on the entire stable region $[T_0]$.
\end{proof}

The crucial reason why Theorem \ref{thm:fin-tip-ae-rot} holds is the following theorem:

\begin{theorem}
\label{thm:cc-in-bdry}
Let $T$ be a finite map in the boundary of stable region. Then $T$ has a critical connection.
\end{theorem}

This means that the parameters defining such a map must satisfy a linear equation with integer coefficients. Thus such maps live on a countable union of codimension-$1$ subspaces and are therefore of zero measure.

\begin{proof}
Let $[T_0]$ be the stable region such that $T \in \partial [T_0]$. Let $\{T_n\}_{n \in \N}$ be a sequence of maps in $[T_0]$ converging to $T$. We may assume that $T$ satisfies the ACC property because otherwise $T$ has a critical connection. We will show that this means that the sequence of non-wandering sets $X(T_n)$ converges in the Hausdorff topology and that the resulting set is a subset of $X(T)$. This immediately follows if we prove that the constancy from Lemma \ref{lem:glob-crit-stab} extends to $T$ in the following sense: the continuations of discontinuities in $\mathcal{C}_X(T_0)$ are contained in $\mathcal{C}_{X}(T)$ and their return times and itineraries are the same as for $T_0$.

For the sake of contradiction, assume that there is a discontinuity $\beta_{T_0} \in \mathcal{C}_{X}(T_0)$ such that its continuation $\beta_T$ has a different itinerary up to the return time of $\beta_{T_0}$ to its component interval of $X(T_0)$. We may without loss of generality assume that $\beta_T$ is of $+$-type, with the other case being analogous. Let $k$ be the smallest time for which the itineraries are different. This means that $T_0^k(\beta_{T_0}) \in I_s$ and $T^k(\beta_T) \in I_{s'}$, where by continuity $s' \in \{s-1,s+1\}$. Since $\beta_T$ is of $+$-type, we must in fact have that $s' = s+1$. Indeed, by ACC, none of the images $T_n^k(\beta_{T_n})$ land on $\beta_s^+(T_n)$, where $\beta_{T_n}$ are the continuations of $\beta_{T_0}$. In particular, they are all to the right of $\beta_s^+$. Thus passing to the limit, we have that $T^n(\beta_T)$ can't be to the left of $\beta_s(T)$, so $s' = s+1$. By a similar argument as before, we have that all of the iterates $(\beta_{T_n})$ must be strictly to the left of $\beta_{s+1}^-(T_n)$. Thus if the itinerary changes, in the limit we must have that $T_n^k(\beta_{T_n})$ lands on $\beta_{s+1}^+(T)$. This is a contradiction with our assumption that there are no critical connections, so $\beta_T$ has the same itinerary up to the return time of $\beta_{T_0}$ to its component interval of $X(T_0)$.

A similar argument then shows that each iterates $T^{r(\beta_T^+)}$ has the same itinerary up to time $r(\beta_T^-)$ as $\beta_T^-$. Indeed, again by looking at the smallest time when the itinerary changes, we get a critical connection for $T$. This clearly shows that the required constancy extends to $T$ in a continuous way, meaning that the orbits of intervals $J_{\beta}$, where $J_{\beta}$ is the interval containing $\beta \in \mathcal{C}_{X}(T_0)$, move continuously from $[T_0]$ to $T$. Finally, we will show that the orbits of different intervals $J_{\beta}$ must remain separated, i.e. they do not touch. Indeed, if two components of different orbits touched, we would have a critical connection. This is because these different components must have different itineraries. Thus they must at some point land on different sides of some discontinuity $\beta_s$. By considering the first such time, we must have that the point at which these components touch lands on $\beta_s$, thus creating a critical connection, since the boundary points of these components are iterates of critical points. Thus $X(T_0)$ has a homeomorphic continuation inside $X(T)$ with the same dynamics as $X(T_0)$.

Since $T$ is not stable, there exists a component interval $J$ of $X(T)$ for which the return map violates Matching, and thus has at least three branches. By our previous discussion, the orbit of $J$ has a definite distance from the continuation of $X(T_0)$. Indeed, it is clear that is disjoint from it and does not touch the the continuation of $X(T_0)$ because the touching point would have to land on discontinuity to separate the orbits of the touching components of $X(T_0)$ and $O(J)$, since they must have different itineraries. Since the ACC condition holds for $T$, the interval $J$ satisfies all of the assumptions of Lemma \ref{lem:J-continuity}. Thus by Lemma \ref{lem:J-continuity}, for any sufficiently small perturbation $\Tilde{T}$ of $T$ there is an interval $\Tilde{J}$ close to $J$ which has a well-defined generalized return map $R_{\Tilde{J}}$. Since $T$ is in the boundary of $J$, we may assume that $\Tilde{T} \in [T_0]$. But then $\Tilde{J}$ 
contain points in $\Tilde{X}$ which are definite distance from the continuation of $X(T_0)$. This is clearly a contradiction, so $T$ does not satisfy ACC and must therefore have a critical connection.
\end{proof}

As a corollary, we get a sufficient condition for a finite type map to not be in the closure of any stable region:

\begin{corollary}
\label{cor:fin-not-in-bdry}
Let $T$ be a finite type map that does not satisfy Matching and has no critical connections. Then $T$ is not in the closure of any stable region.
\end{corollary}

\begin{proof}
Since $T$ does not satisfy Matching, it is not contained in any stable region. Since it does not have any critical connection, by Theorem \ref{thm:cc-in-bdry} it is also not contained in the boundary in any stable region. Thus it is not contained in the closure of any stable region. 
\end{proof}

Theorem \ref{thm:cc-in-bdry} allows us to establish the following criterion for a map to be of infinite type:

\begin{lemma}
\label{lem:inf-type-characterization}
Let $T$ be a map that is not a finite type stable map and assume that the defining parameters $(\gamma_1, \dots, \gamma_r, \beta_1, \dots, \beta_{r-1})$ are rationally independent. Then $T$ is an infinite type map.  
\end{lemma}

To prove this criterion, we need the following lemma about the number of equations a map needs to satisfy in order for a return map to an interval of $X$ to have $n \ge 2$ branches.

\begin{lemma}
\label{lem:rj-equation-number}
Let $T$ be a map such that there is an interval $J$ contained in $X$ on for which the return map $R_J$ has exactly $n \ge 2$ continuity intervals. Then there are at least $n-2$ linearly independent vectors $v_1, \dots, v_{n-2}$ with integer coefficients such that defining parameters $(\gamma \, \beta)$ of $T$ satisfy $\langle v_i, (\gamma \, \beta) \rangle = 0$, for all $1 \le i \le n-2$.
\end{lemma}

\begin{proof}
Let $\sigma$ be the permutation associated to $R_J$ and let $\tau$ be its inverse. We have the following string of $n-1$ equalities that must hold:
\begin{align*}
R_J(a^-_{\tau(1)}) &= R_J(a^+_{\tau(2)-1}) \\
R_J(a^-_{\tau(2)}) &= R_J(a^+_{\tau(3)-1}) \\
&\dots \\
R_J(a^-_{\tau(n-1)}) &= R_J(a^+_{\tau(n)-1}),
\end{align*}
where the $a$'s are the discontinuities of $R_J$ and the boundary points of $J$. There are two cases. The first is when the boundary points of $J$ land on discontinuities. Then we have that $R_J(a^{\pm}_{i}) = \langle R^{J,\pm}_{i}, (\gamma \, \beta) \rangle$, where $R^{J,\pm}_{i}$ is the corresponding return vector, so all the equalities above are equivalent to:
\begin{align}
\label{eq:vector-eq}
\begin{split}
\langle R^{J,-}_{\tau(1)} - R^{J,+}_{\tau(2)-1}, (\gamma \, \beta) \rangle &= 0 \\
\langle R^{J,-}_{\tau(2)} - R^{J,+}_{\tau(3)-1}, (\gamma \, \beta) \rangle &= 0 \\
&\dots \\
\langle R^{J,-}_{\tau(n-1)} - R^{J,+}_{\tau(n)-1}, (\gamma \, \beta) \rangle &= 0
\end{split}
\end{align}
If there is a linear dependence between the $n-1$ vectors $R^{J,-}_{\tau(1)} - R^{J,+}_{\tau(2)-1}, \dots, R^{J,-}_{\tau(n-1)} - R^{J,+}_{\tau(n)-1}$ with coefficients $\alpha_1, \dots, \alpha_{n-1}$, we can extend it to the linear dependence of all vectors associated to the return map $J$, by setting the coefficients for all other vectors to $0$. Then it easily follows from Theorem \ref{thm:lin-dep} that all of the coefficients of this linear dependence must be zero, so $\alpha_1 = \dots = \alpha_{n-1} = 0$. This is due to the simple fact that return vector $R^{J,+}_{\tau(1)-1}$ does not appear among the vectors in \eqref{eq:vector-eq}, so the coefficient next to it is $0$, and this propagates in the conclusion of Theorem \ref{thm:lin-dep} to all coefficients being equal to zero. Thus in this case we get $n-1$ linearly independent vectors.

If either $a_0^+$ or $a_n^-$ does not land on a discontinuity, then we must replace in \eqref{eq:vector-eq} the vector $R_0^{J,+}$ by $R_{\tau(1)-1}^{J,+} + R_0^{J,+}$, or the vector $R_n^{J,-}$ by $R_{\tau(n)}^{J,-} + R_n^{J,-}$. This is because in this case $R_J(a^{+}_{0}) = a_0^+ + \langle R^{J,+}_{0}, (\gamma \, \beta) \rangle$ and $a_0^+ = \langle R_{\tau(1)-1}^{J,+}, (\gamma \, \beta) \rangle$, and similarly for $a_n^-$. If either of these vectors lands on a discontinuity, then the same argument as above shows that we get $n-1$ linearly independent vectors, because either $R_{\tau(1)-1}^{J,+}$ or $R_{\tau(n)}^{J,-}$ will have coefficient zero. Thus we may assume that both boundary points do not land on discontinuities. But then we may simply take the subset of $n-2$ vectors that does not include the vector related to the return of $a_0^+$, and the same argument as before shows that these vectors are linearly independent.
\end{proof}

The last part of the proof where we reduce the set from $n-1$ to $n-2$ vectors might appear unnecessary, but it is simple to check that for cases $n=3$ and $n=2$, the set of $n-1$ vectors we get in the last part of the proof is linearly dependent.

\begin{proof}[Proof of Lemma \ref{lem:inf-type-characterization}]
Since we have assumed that the defining parameters are rationally independent, we know that $T$ has no critical connections so it satisfies ACC. If $T$ did not satisfy Matching, there would have to be a component interval of $X$ with at least three branches of the return map. By Lemma \ref{lem:rj-equation-number} this would mean that there is at least one non-zero vector $v$ with integer coefficients such that $\langle v, (\gamma \, \beta) \rangle$ is equal to $0$. This contradicts our assumption on rational independence, so $T$ satisfies Matching. Thus if $T$ were of finite type, it would have to be stable, but we have assumed that this is not the case. Thus $T$ has to be of infinite type.
\end{proof}

Finally, Lemma \ref{lem:inf-type-characterization} allows us to prove Theorem \ref{thm:fin-tip-ae-rot}:

\begin{proof}[Proof of Theorem \ref{thm:fin-tip-ae-rot}]
Let $\mathrm{INF}(r)$ be the set of all infinite type $\ITM$s on $r$ intervals. The set of all rationally independent parameters $\mathcal{R}(r)$ has full measure in $ITM(r)$. By Lemma \ref{lem:rj-equation-number}, we know that any map associated with such parameters has to be either a stable finite type map or an infinite type map. Thus $S(r) \cup \mathrm{INF}(r)$ has full measure in $ITM(r)$, so $\mathrm{INF}(r)$ has full measure in $ITM(r) \setminus S(r)$.
\end{proof}

\section{Parameter space and future work}
\label{sec:future} 

\subsection{The Bruin--Troubetzkoy family}
\label{subsec:bt-family}
In this subsection, we are going to apply our results to the parameter space of the special family from \cite{MR2013352} and show that the open coloured triangles in Figure~\ref{fig:bt-fin} (duplicated below for easier reference) form an open and dense subset of the parameter space.

\vspace{3mm}
\begin{center}
\label{fig:bt-special-family}
\includegraphics[width=0.5\textwidth]{itm_correct_parameter_space.png}
\end{center}

In \cite{MR2013352}, the authors consider a two-parameter family $\mathcal{T}_{a,b}$, with $0 \le b \le a \le 1$, consisting of maps $T_{a,b}$ defined as follows:
\[
T_{a,b}(x)= \left\{  \begin{array}{ll} x+a  &\mbox{ for } x\in [0,1-a)\\ 
x+b  &\mbox{ for } x\in [1-a,1-b) \\ 
x +b -1 & \mbox{ for } x\in [1-b, 1)  \end{array} \right. 
\]

In \cite{MR2013352} it was shown that the set of all finite type maps in $\mathcal{T}_{a,b}$ is equal to the set of all maps that under the iterates of a certain renormalisation operator get mapped into a certain region. The set of closed yellow triangles touching the bottom side of the big triangle in the figure above corresponds to the set of all maps that get mapped into this region after one step. It is easy to show that the maps on the boundaries of these triangles are not stable (in this family) because they violate ACC and that the maps in the interiors of these triangles are stable. Moreover, every other coloured triangle in the picture gets mapped homeomorphically onto one of these yellow triangles under the iterates of the renormalisation operator. Thus the same conclusion holds for any triangle: the maps in the interior are stable and the maps on the boundary are not. Thus the set of all stable finite type maps is equal to the union of interiors of these triangles. In the remainder of this subsection, we are going to prove that the stable maps form a dense subset of this family, which therefore shows that the set of (open) triangles forms a dense and open subset of the parameter space, as suggested by the figure. The authors of \cite{MR2013352} do not explicitly state this result, but it can probably be derived from their results. 

To apply our results to the family $\mathcal{T}_{a,b}$, we will use the following trick. We may assume that $\beta_0^+$ and $\beta_r^-$ are parameters, and not the fixed points $0^+$ and $1^-$. The boundary points of the interval are parameters of this family but they are not discontinuities of a map. Thus we allow for the interval of the definition of an $\ITM$ to change, which enlarges the space $\ITM(r)$ to $\widehat{\ITM}(r)$ of $2r+1$ parameters. The boundary conditions for this space are:
\begin{align*}
\beta^+_0 &\le \beta_{s-1}^+ + \gamma_s; \\
\beta^-_r &\ge \beta_s^- + \gamma_s, \\
\beta_{s+1} &\ge \beta_{s}
\end{align*}
for all $0 \le s \le r-1$. We will in fact only need to consider a small neighbourhood of $\ITM(r)$ in this space, so all of the parameters we consider are uniformly bounded. The reason for making this extension is that when we make perturbations in the proofs of Theorems \ref{thm:stability=accm} and \ref{thm:ep->acc+m}, we need to be able to change all of the parameters defining a map $T$, and this is impossible if the parameters $0$ and $1$ are fixed.

Thus we can consider an extension $\widehat{\mathcal{T}}_{a, b}$ of the family above in this extended space $\widehat{\ITM}(3)$. The defining equations of this family are:
\begin{align*}
\beta_1^- + \gamma_1 &= \beta_3^- \\
\beta_2^- + \gamma_2 &= \beta_3^- \\
\beta_2^+ + \gamma_3 &= \beta_0^+.
\end{align*}

\begin{theorem}[Density of stable maps in the Bruin--Troubetzkoy family]
Stable maps are dense in $\mathcal{T}_{a,b}$.
\end{theorem}

\begin{proof}

The following is the list of changes in the proof of our $4$ main Theorems: \ref{thm:ep-dense-param}, \ref{thm:lin-dep}, \ref{thm:stability=accm} and \ref{thm:ep->acc+m}.

\begin{enumerate}
\label{special-family-changes}
    \item 
    \label{item:BT1}
    Theorem \ref{thm:ep-dense-param} - Density of eventually periodic maps in $\mathcal{T}_{a,b}$: We already know that eventually periodic maps are dense in parametric families because Theorem \ref{thm:ep-dense-param} was proven for rational families, and $\mathcal{T}_{a,b}$ is clearly a rational family. Thus every map in the interior of this parameter space can be approximated by an eventually periodic map in this space, so the eventually periodic maps are dense in $\mathcal{T}_{a,b}$.
    \item 
    \label{item:BT2}
    Theorem \ref{thm:lin-dep} - Linear Dependence of itinerary vectors for maps in $\widehat{\ITM}(3)$: We need to add the itinerary vectors for $\beta_0^+$ and $\beta_r^-$ to the statement and proof. This is where the trick of considering the extended parameter space $\widehat{\ITM}(r)$ becomes relevant: we need $0$ and $1$ to be parameters so that the calculations in the proof still apply to them. The proof remains unchanged otherwise. 
    \item 
    \label{item:BT3}	
    Theorem \ref{thm:stability=accm} - Matching and ACC is equivalent to Stability in $\widehat{\ITM}(3)$: The left and right boundary points $\beta_0^+$ and $\beta_r^-$ are not discontinuities of a map in $\widehat{\ITM}(3)$, so they are not critical points (but they are parameters of the map, which is the important point for Theorem \ref{thm:lin-dep}). Thus these points do not need to be included in the definition of ACC, Matching or stability, and so the proof that ACC and Matching characterizes stability does not change. Note that a map $T$ which is stable in $\widehat{\ITM}(3)$ and contained in $\mathcal{T}_{a,b}$ is stable within this parameter family as well, in the sense that there is neighbourhood of $T$ in $\mathcal{T}_{a,b}$ which satisfies all of the properties of stability as in Definition \ref{defn:stable}, 
    \item 
    \label{item:BT4}
    Theorem \ref{thm:ep->acc+m} - Approximation of eventually periodic maps in $\mathcal{T}_{a,b}$ by stable maps in $\widehat{\mathcal{T}}_{a,b}$: By item \ref{item:BT2}, the vectors corresponding to the defining critical connections of the family $\widehat{\mathcal T}_{a,b}$ are included in Theorem \ref{thm:lin-dep}. Thus these critical connections may be preserved every time we make a perturbation in the proof Theorem \ref{thm:ep->acc+m}. Thus the same proof as before shows that we may make a finite sequence of perturbations to an eventually periodic map that makes it stable.
\end{enumerate}
By item~\ref{item:BT1}, if we start with a map $T$ in $\mathcal{T}_{a,b}$, we can perturb it to an eventually periodic map that is still in $\mathcal{T}_{a,b}$. Because of item~\ref{item:BT4}, we may then perturb it to a map in the extended parametric family $\widehat{\mathcal{T}}_{a,b}$ that satisfies ACC and Matching and is thus stable by item~\ref{item:BT3}. This map $\Tilde{T}$ is possibly not contained $\ITM(3)$, but because the perturbations we make are arbitrarily small, there is an arbitrarily small rescaling of it that is contained in $\ITM(3)$. Indeed, we may first translate all of the $\beta$-parameters by $-\beta_0$, so that the left boundary of the interval is $0^+$. Note that this preserves the itineraries of all points and the (dynamically relevant) equations satisfied by the parameters of the map, so it does not change the dynamics. Moreover, this change in parameters can be chosen arbitrarily small, since $\tilde{T}$ is arbitrarily close to a map in $\mathcal{T}_{a,b}$ for which $\beta_0^+ = 0$. Next, we multiply all of the parameters by ${1}/{\beta_1^-}$. This also does not change the dynamics, is arbitrarily small, and gives that the left boundary point is now $1^+$. Thus by this rescaling, arbitrarily close to the stable map in $\widehat{\mathcal{T}}_{a,b}$ we get a map in $\ITM(3)$ with the same dynamical properties, which must therefore be stable in $\widehat{\ITM}(3)$ and contained in $\mathcal{T}_{a,b}$. This map is therefore stable in $\mathcal{T}_{a,b}$, as described in item~\ref{item:BT3}. Thus the stable maps are dense in $\mathcal{T}_{a,b}$.
\end{proof}

\subsection{Future work: geometry and dynamics in the boundary of stable regions}
\label{subsec:param-space}
In this final subsection, we discuss the properties of the full parameter space $\ITM(r)$ of all $\ITM$s on $r$ intervals, and we present a few conjectures. Recall that a connected component of the set of stable maps $\mathcal{S}(r)$ corresponds to the maximal neighbourhood $[T]$ of a stable map $T$ that satisfies the three properties from the definition of stability (Definition \ref{defn:stable}). Such sets $[T]$ are called \textit{stable regions}. A very interesting problem is to describe the geometry of these regions.

\begin{question}
What are the geometric properties of stable regions and their boundaries?
\end{question}

In a companion paper, we will give an example using ghost preimages (see Definition \ref{def:ghost-tree} and Example \ref{ex:ghost-preimage}) which shows that stable regions are not convex in general. Since ghost preimages are codimension $1$ phenomena, it seems plausible that by taking the closure of a stable region, we can make it convex.

\begin{question}
Is the closure of a stable region always convex?
\end{question}

Because the stable regions are generally not convex, the following also appears to be a non-trivial question:

\begin{question}
Are the stable regions simply connected?
\end{question}

Because of ghost preimages, the boundary of stable regions can contain countably many hyperplanes. The accumulation properties of these hyperplanes seem difficult to describe, and we are led to the following question:

\begin{question}
Is the boundary of every stable region locally connected?    
\end{question}

Besides their geometry, a step towards the proof of Boshernitzan--Kornfeld Conjecture is to understand how the stable regions are arranged in the parameter $\ITM(r)$. Are they necessarily separated from each other, or can they `tile' the parameter space?

\begin{question}
Are the stable regions (partially) separated from each other? In other words, does for every stable region $[T]$ exist an open subset of the (outer) boundary on which the set of points that are contained in the boundary of some other stable component is small (e.g. meager or countable)?
\end{question}

Figure \ref{fig:bt-special-family} suggests that this is true: every closed coloured triangle has at least one side on which the set of points where it touches other triangles is countable.

Another interesting problem is describing the dynamics in the boundary of stable regions: 

\begin{question}
What are the dynamical properties of a map in the boundary of a stable region?
\end{question}

The boundary of a stable set is where the topology of a non-wandering set changes, so it is contained in \textit{bifurcation locus} of $\ITM(r)$. Thus the answer to this question would constitute the start of the bifurcation theory of $\ITM$s. Intuitively, the bifurcations should be associated with the changes of critical itineraries, and they should in turn be associated with critical connections. This is confirmed by Theorem \ref{thm:cc-in-bdry} for finite type maps. It is not clear that this should also hold for infinite-type maps:

\begin{question}
Do infinite type maps in the boundary of a stable region have critical connections?
\end{question}

Since a necessary condition for a finite type map to be in the boundary of a stable region is that it has a critical connection, a natural question to ask is whether the maximal possible number of critical connections is sufficient for a map to be in the closure of a stable region. The eventually periodic maps are exactly such maps, so we conjecture the following:

\begin{conjecture}
Each eventually periodic map is contained in the closure of some stable region.
\end{conjecture}

The proof of Theorem \ref{thm:ep->acc+m} does not give the proof of this conjecture, even though the perturbation is arbitrarily small. The reason for this is that the perturbations we inductively make do not give us full control of the dynamical properties of the resulting map, just that the number of critical connections decreases and that the map remains eventually periodic. Thus it is possible that by making the perturbation smaller, we end up in different stable regions.  

\appendix
\section{Connection to Billiards}
\label{appendix:billiards}

Interval Translation Maps naturally arise in the context of billiards with spy mirrors. This was first observed in \cite{MR1356616} and later studied in more detail in \cite{MR3403406, MR3449199}. In this subsection, we elaborate on an example of this construction.

Let $P$ be a Euclidean polygon. A billiard path in $P$ is the trajectory of a particle moving inside $P$ by straight lines and bouncing off the sides of the polygon according to the physical law of reflection: the angle of incidence equals the angle of reflection. Such trajectories terminate when they meet the angles of the polygon (see \cite{MR2168892} for a standard introduction to billiards). The polygon $P$ is called \emph{rational} if all its angles are rational multiples of $\pi$. For such polygons, one can associate a \emph{canonical} surface $S$ with flat Euclidean metric in such a way that every billiard trajectory in $P$ corresponds to a straight line (geodesic) in $S$. The flat metric on $S$ is defined everywhere except at finitely many cone points, called \emph{singularities}, for which the cone angles are multiples of $ 2\pi $. The surface $S$ is constructed by \emph{unfolding} of $P$ (see \cite{MR0399423, MASUR20021015, MR2809109} for more details of the construction below). 

For each side $e_i$ of $P$, let $ \ell_i \subset \mathbb{R}^2 $ be the line passing through the origin parallel to $e_i$. Consider the group $\mathcal{A}_P$ generated by reflections $R_{e_i}$ in the line $\ell_i$, where $e_i$ runs over the complete list of edges of $P$. Since the angles of $P$ are rational, that is, have the form $2\pi m_i / n_i$ for some coprime integers $m_i < n_i$, the group $\mathcal{A}_P$ will contain precisely $2N$ elements, where $N$ is the least common multiple of the $n_i$'s. The direction of a billiard trajectory in $P$ is constant between bounces and changes after each bounce by the action of an element in $ \mathcal{A}_P $. Thus, every billiard trajectory in a rational billiard can have only finitely many directions. 

To construct $S$, consider $2N$ polygons of the form $g(P) + v_g$, where $g$ runs over the elements in $\mathcal{A}_P$ and the translations $v_g$ are chosen so that the resulting polygons are disjoint; denote them $P_1, \ldots, P_{2N}$. We now glue the edges of these polygons pairwise to form a surface: for every edge $e_i$ in $P$ and an element $g \in \mathcal{A}_P$, consider the element $g \circ R_{e_i} \in \mathcal{A}_P$. For these two elements, define the edges 
\begin{equation}
\label{Eq:Edges}
e' := g(e_i) + v_g \quad \text{ and } \quad e'' := g \circ R_{e_i}(e_i) + v_{g \circ R_{e_i}}.
\end{equation}
It is easy to check that these edges are parallel and belong to a pair of distinct polygons in the collection $P_1, \ldots, P_{2N}$. In this way, all edges of the polygons $P_1, \ldots, P_{2N}$ split in such pairs $e', e''$. We identify these edges by translation (explicitly given by $v_{g \circ R_{e_i}} - v_g$). After all the identifications, we obtain a closed orientable surface with a flat metric except at finitely many conical singularities (corresponding to the identified vertices of the polygons) where the angles are multiples of $2\pi$ (\cite[Lemma 17.3]{MR2809109}). This is the desired surface $S$; such surfaces are called \emph{translation surfaces} to indicate that they were obtained by identifying parallel pairs of sides in a collection of polygons by Euclidean translations.

For a given direction $\theta$, the billiard trajectories in $P$ in direction $\theta$ correspond to a parallel geodesic flow $T_\theta $ on $ S $ in direction $\theta$. Each trajectory, instead of reflecting off a side $e$ of $P$, continues as a straight line into the reflected copy of $P$ attached to $P$ through $e$. This flow is called the \emph{billiard flow}. If $I$ is a segment transverse to this flow, then the first return map of $ T_\theta $ to $I$, properly rescaled, is an \emph{interval exchange transformation} \cite[Section 1.7]{MASUR20021015}. This connection between rational billiards and $\IET$s has been explored in depth in many publications, as mentioned in the introduction.

Interval translation mappings appear in a similar way once we add \emph{spy mirrors} to our billiard table.

A \emph{spy mirror} in $P$ is a two-sided vertical segment, with one \emph{transparent} side and one \emph{reflective} side. The trajectory $\gamma$ of a particle reflects according to the standard rules if it approaches the polygon sides or a spy mirror from its reflective side, and $\gamma $ passes through unobstructed if it approaches a spy mirror from the mirror's transparent side. A \emph{rational billiard with spy mirrors} is a rational polygon $P$ together with finitely many spy mirror segments inside $P$ and the reflection law just described; the spy mirrors are assumed to make rational angles with the sides of $P$. We denote such a billiard table by $\widehat{P}$ \footnote{The polygon $P$ "wears" a hat $ \,\,\widehat{\phantom{}}\,\,\, $, as all typical spies do.}. Unlike classical polygonal billiards, billiards in polygons with spy mirrors are non-invertible dynamical systems. 

We can associate to $\widehat{P}$ a surface $\widehat{S}$ called a \emph{translation surface with spy mirrors}. We utilize a similar construction as above, where instead of $\mathcal{A}_P$, we use the group $\mathcal{A}_{\widehat{P}}$ generated by reflections in lines passing through the origin and parallel to all sides of $P$ \emph{as well as} to all the mirrors in $P$. By the same reasoning, $\mathcal{A}_{\widehat{P}}$ is a finite group with $2\widehat{N}$ elements for some $\widehat{N} \in \mathbb{N}$. The surface $\widehat{S}$ is obtained by gluing together $2\widehat{N}$ rational polygons $\widehat{P}_1, \ldots, \widehat{P}_{2\widehat{N}}$ with spy mirrors, each being a reflected copy of $\widehat{P}$. Each polygon in this collection is of the form $g(\widehat{P}) + v_g$, and we reflect and translate not only $P$ but also all the spy mirrors. In the images, the reflective and transparent sides of the spy mirror are swapped, and $\widehat{S}$ is obtained by gluing pairs $e', e''$ of parallel sides in $\widehat{P}_1, \ldots, \widehat{P}_{2\widehat{N}}$ using the relation \eqref{Eq:Edges}. Note that we do not glue along spy mirrors. In this way, we obtain a translation surface with two-sided slits. Every forward billiard trajectory $\gamma_\theta \subset \widehat{P}$ in direction $\theta$ unfolds into a straight geodesic ray $t_\theta \subset \widehat{S}$ with the property that if it approaches the reflective side of a spy mirror $z'$ in a copy $\widehat{P}_i = g_i(\widehat{P}) + v_{g_i}$, with $ z' = g_i(z) + v_{g_i} $ for some spy mirror $z \subset P$, then $t_\theta$ continues \emph{from the same position on the reflective side} of the parallel mirror $z'' := g_i \circ R_{z}(z) + v_{g \circ R_{z}}$ in the copy of $\widehat{P}$ generated by $g_i \circ R_{z}$ and $v_{g \circ R_{z}}$ (here, $R_z$ is the reflection in the line through the origin parallel to $z$). Thus, if $I \subset \widehat{P}$ is a segment in the complement of the spy mirrors, and $\theta'$ is a direction transverse to $I$, then all billiard trajectories in $\widehat{P}$ starting on $I$ in the direction $ \theta'$ correspond to a flow $T_\theta$ of parallel lines in $\widehat{S}$ starting in some segment $J$ in some transverse direction $\theta$ (that depends on $\theta'$ and the starting interval $I$). If this flow returns to $J$, then the \emph{properly rescaled first return map} is naturally an \emph{interval translation map} (the rescaling is done in the same way as for rational billiards). 

Figure~\ref{Fig:ExampleITM} illustrates the construction in the simplest possible case when $P$ is a square, and $\widehat{P}$ is $P$ together with two vertical spy mirrors starting at the horizontal side. In this case, $\widehat{N} = 2$, the group $\mathcal{A}_{\widehat{P}}$ is generated by reflections in the coordinate axes, and the corresponding translation surface $\widehat{S}$ is a torus obtained by gluing four copies of the square $P$. The first return map of the flow $T_\theta$ on $\widehat{S}$ to the global vertical transversal is shown in Figure~\ref{Fig:ExampleITM2}, left; this map is naturally an $\ITM$ on four intervals (same picture, right). 

\begin{figure}
\begin{center}
\includegraphics[scale=1.1, trim = 30 30 30 30, clip]{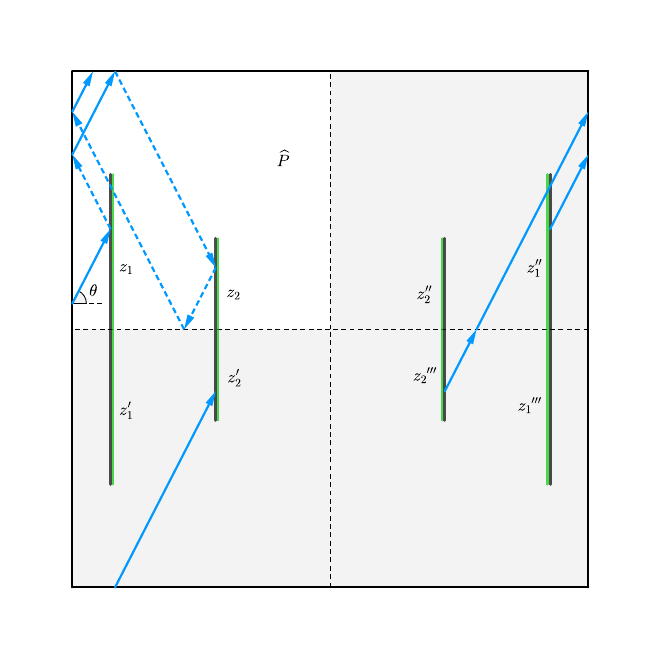}
\caption{The rational polygon $\widehat{P}$ is the white square $P$ with two vertical spy mirrors $z_1, z_2$ starting on the bottom side of $P$. The transparent sides of the spy mirrors are indicated in green, while all the reflective sides and the reflective edges are marked in black. The translation surface $\widehat{S}$ is the torus obtained by gluing four reflected copies of $\widehat{P}$ (one in white and three in gray). In the picture, the opposite sides of the large square are identified. The copies $z_1, z_1'$ and $z_2, z_2'$ of reflected spy mirrors are identified with the two copies $z_1'', z_1'''$ and $z_2'', z_2'''$ respectively, the reflective sides are flipped. The piece of the billiard trajectory in direction $\theta$ is shown in $P$ (solid and dashed segments in $P$). This trajectory unfolds into the straight line trajectory in $ \widehat{S} $ shown with solid blue lines.}
\label{Fig:ExampleITM}
\end{center}
\end{figure}

\begin{figure}
\begin{center}
\makebox[\linewidth][c]{\includegraphics[scale=0.9, trim = 20 20 20 20, clip]{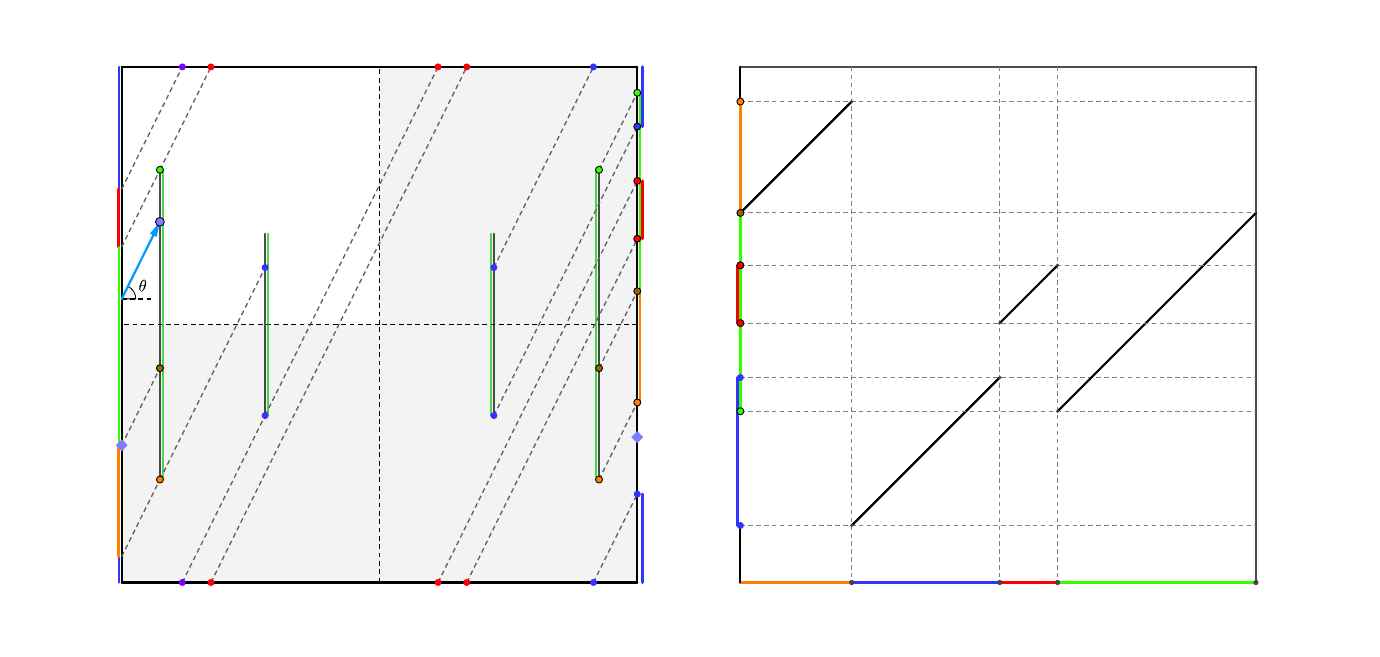}}
\caption{(Left) The flow $T_\theta$ starts on the entire vertical side on the left. Under this flow (with the corresponding jumps due to the presence of spy mirrors), the colored segments on the left (blue, red, green, and orange) are mapped isometrically to the colored segments on the right. In this way, we obtain an $\ITM$ (shown on the right as the graph of the map) as the first return map of the flow to the transversal, properly cut open from the circle to an interval; the cutting is done at the point marked with the rhombus. }
\label{Fig:ExampleITM2}
\end{center}
\end{figure}

\pagebreak

\section{ITM Program (due to Bj\"orn Winckler)}
\label{appendix:program}

The program used for drawing the parameter space of the BT-family (Figure \ref{fig:bt-fin}) is due to Bj\"orn Winckler. The program also draws the dynamical plane of maps from this family, and there is a user interface which allows for the selection of parameters defining a specific map (see the screenshot below). The file and the instructions for how to use it are available on GitHub, on the following \href{https://github.com/LeonStare/ITM-Parameter-and-Dynamical-Space?tab=readme-ov-file}{link}.

\begin{center}
\label{fig:code-interface}
\includegraphics[width=\textwidth]{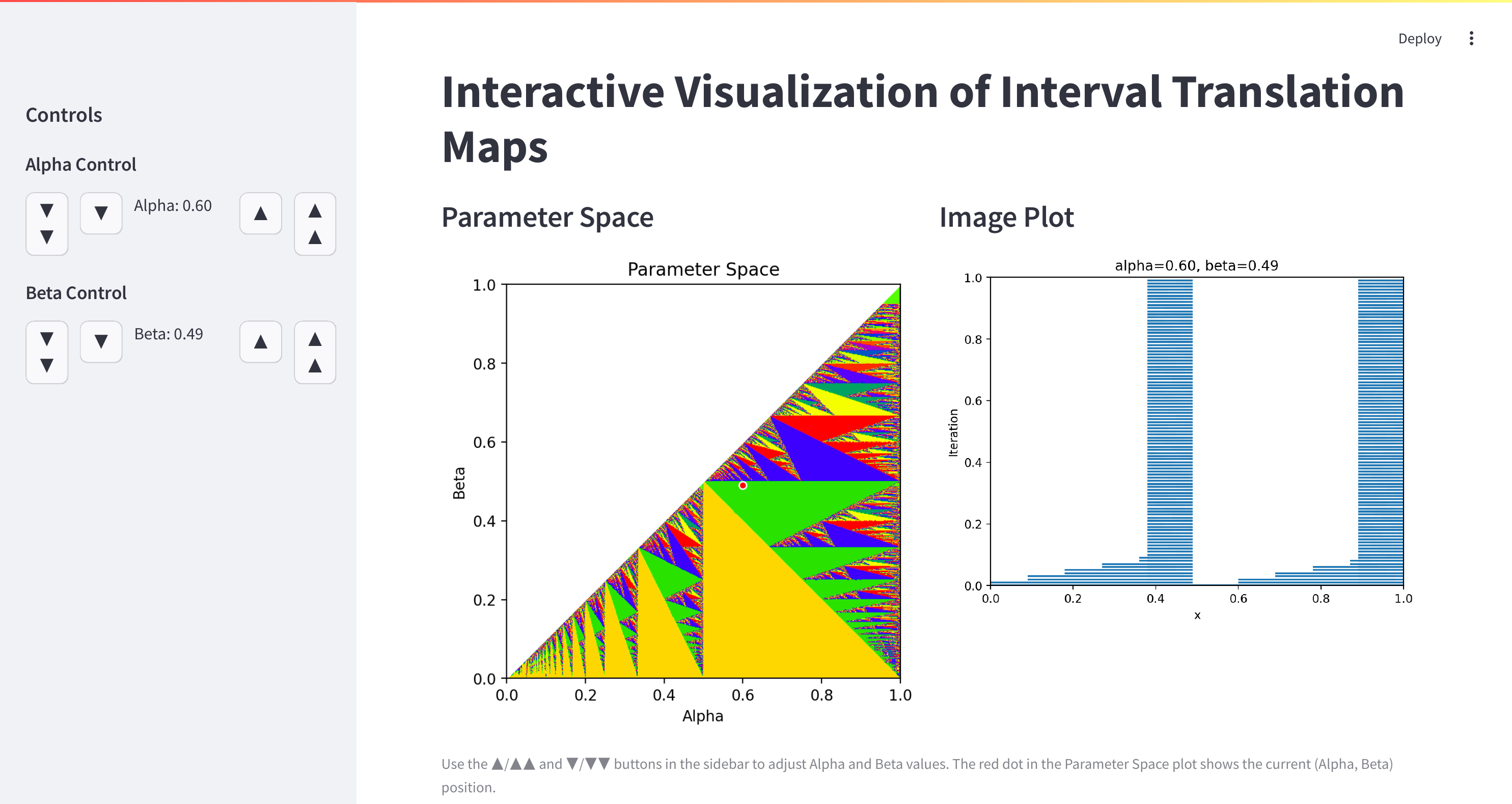}
\end{center}

\addcontentsline{toc}{section}{References}
\printbibliography

\end{document}